\documentclass[a4paper,11pt]{article}

    \usepackage[pdftex]{hyperref}

\usepackage{mathrsfs}

\usepackage[top=1.8in, bottom=1.8in, left=1.3in, right=1.3in]{geometry}

\usepackage{chicago}      %for bibliographystyle

\usepackage{graphicx}              % to include figures
\usepackage{amsmath}               % great math stuff
\usepackage{amsfonts}              % for Blackboard bold, etc
\usepackage{amsthm}                % better theorem environments
\usepackage{amssymb}

\usepackage{hhline,booktabs}

\usepackage[utf8]{inputenc}
\usepackage{emptypage}

\usepackage{imakeidx}
\makeindex

\usepackage{enumerate}

\usepackage{accents}

\usepackage{tikz}

\theoremstyle{plain}
\newtheorem{thm}{Theorem}[subsection]
\newtheorem{lem}[thm]{Lemma}

\newtheorem{cor}[thm]{Corollary}
\newtheorem{conj}[thm]{Conjecture}

\theoremstyle{definition}
\newtheorem{dfn}[thm]{Definition}
\newtheorem{fact}[thm]{Fact}
\newtheorem{rem}[thm]{Remark}
\newtheorem{dis}[thm]{Discussion}
\newtheorem{qst}[thm]{Question}
\newtheorem{notation}[thm]{Notation}
\newtheorem{ackn}[thm]{Acknowledgement}

\DeclareMathOperator{\id}{id}
\DeclareMathOperator{\wid}{wid}
\DeclareMathOperator{\fid}{fid}
\DeclareMathOperator{\lh}{lg}

\newcommand{\HC}{$H_\kappa$}
\newcommand{\landx}{\,\land \, }
\newcommand{\Ckappa}{S^{\kappa^+}_\kappa}   % E^(kappa^+)_kappa in Jech, p.94
\newcommand{\RR}{\mathbb{R}}      % for Real numbers
\newcommand{\PP}{\mathbb{P}}     
\newcommand{\QQ}{\mathbb{Q}}       
\newcommand{\CC}{\mathbb{C}}       
\newcommand{\HH}{\mathbb{H}}      
\newcommand{\TT}{\mathbb{T}}       
\newcommand{\LL}{\mathbb{L}}       
\renewcommand{\AA}{\mathbb{A}}     
\newcommand{\pp}{\mathsf{p}}      

\newcommand{\on}{{\upharpoonright}}
\newcommand{\forces}{\Vdash}

\newcommand{\incomp}{\perp}

\newcommand{\seq}{\subseteq}

\DeclareMathOperator{\dom}{dom}
\DeclareMathOperator{\ran}{ran}
\DeclareMathOperator{\supp}{supp}

\DeclareMathOperator{\tr}{tr}

\newcommand{\setmin}{{\setminus}}
\newcommand{\squ}{\accentset{\rightharpoonup}}
\newcommand{\powset}{\mathfrak{P}}
\DeclareMathOperator{\nacc}{nacc}
\DeclareMathOperator{\acc}{acc}
\DeclareMathOperator{\set}{set}
\DeclareMathOperator{\suc}{suc}

\DeclareMathOperator{\incOP}{inc}
\newcommand{\inc}{{\incOP}} 
\DeclareMathOperator{\amOP}{am}
\newcommand{\am}{{\amOP}}

\newcommand{\nst}{\mathbf{nst}}

\DeclareMathOperator{\Cohen}{Cohen}
\DeclareMathOperator{\Borel}{\mathbf{Borel}}

\newcommand{\hsl}{h\slalom}
\newcommand{\slalom}{\text{\rm -slalom}}
\DeclareMathOperator{\add}{add}
\DeclareMathOperator{\cov}{cov}
\DeclareMathOperator{\non}{non}
\DeclareMathOperator{\cof}{cof}
\DeclareMathOperator{\cf}{cf}

\DeclareMathOperator{\local}{\in^*}
\DeclareMathOperator{\plocal}{\pp{\in^*}}

\DeclareMathOperator{\ppp}{\mathrm{partial}}

\DeclareMathOperator{\CON}{CON}

\DeclareMathOperator{\pr}{pr}

\renewcommand{\b}{\mathfrak{b}}

\newcommand{\ifff}{\quad\Leftrightarrow\quad}

\newcommand{\tle}[1]{2^{< #1}}
\newcommand{\tleq}[1]{2^{\leq #1}}
\newcommand{\kle}[1]{\kappa^{< #1}}

\newcommand{\name}{\dot}

\newcommand{\mi}{{\mathbf i}}

\newcommand{\QA}{\QQ_{\kappa}^{\am,1}}
\newcommand{\QS}{\QQ_{\kappa,S}^{\am,2}}

\title{The Higher Cicho\'n Diagram}
\author{Thomas Baumhauer \and
Martin Goldstern \and
Saharon Shelah}
\begin{document}
\maketitle
\thanks{The first and the second author were partially supported by the 
Austrian Science Foundation through grant FWF P29575. 
All three authors were partially supported by the European Research 
Council under grant ERC-2013-ADG-338821.   Publication 1144 on http://shelah.logic.at}

\begin{abstract}
For a strongly inacessible cardinal~$\kappa$, we investigate
the relationships between the following ideals:
\begin{enumerate}
   \item the ideal of meager sets in the ${<}\kappa$-box product topology
   \item the ideal of ``null'' sets in the sense of \cite{Sh:1004}
   \item the ideal of nowhere stationary subsets of a (naturally defined) stationary set $S_{\pr}^\kappa \seq \kappa$.
\end{enumerate}
In particular, we analyse the provable inequalities between the 
cardinal characteristics for these ideals, and we give consistency
results showing that certain inequalities are unprovable.

While
some results from the classical case ($\kappa=\omega$) can 
be easily generalized to our setting, some key results (such as
a Fubini property for the ideal of null sets) do not hold; this 
leads to the surprising inequality cov(null)$\le$non(null). Also,
concepts that did not exist in the classical case (in particular,
the notion of stationary sets) will turn out to be relevant.

We construct several models to distinguish the various cardinal
characteristics; the main tools are iterations with $\mathord<\kappa$-support
(and a strong ```Knaster'' version of $\kappa^+$-cc) and one iteration
with ${\le}\kappa$-support (and a version of $\kappa$-properness). 
\end{abstract}

\newpage
\tableofcontents

\newpage
\setcounter{section}{-1}
\section{Introduction}\label{intro}

\emph{Set theory of the reals} deals with topological, measure-theoretic
and combinatorial properties of the real line, which set theorists
often do not interpret as the linear continuum $\mathbb R$, but (often 
for technical or notational convenience) as the Cantor space $2^\omega$
or the Baire space $\omega^\omega$. 

We will be interested in a natural generalization of such properties 
to the spaces $2^\kappa$ and $\kappa^\kappa$ for uncountable
(and in this paper: always inaccessible) cardinals $\kappa$.  
This area of research has progressed quickly in recent years; 
\cite{KLLS:2016} collected many questions inspired by workshops on 
generalized reals, and 
several recent results can be found in  
\cite{BBFM:2016}, 
\cite{FL:2017},
\cite{Sh:1004},
\cite{CnSh:1085}.
% \cite{CuSh:541},
% \cite{RoSh:860},

Concerning terminology, we suggest to use the adjective ``higher''
instead of the less specific ``generalized'' or ``generalised''.
 In analogy to 
\emph{higher Souslin trees} (Souslin trees on cardinals larger than $\omega_1$), \emph{higher recursion theory} (recursion theory on ordinals greater than $\omega$), 
\emph{higher descriptive set theory} we will speak of higher reals,
the higher Cantor space, 
higher random reals, the higher Cicho\'n diagram, etc. 

We will occasionally refer to results or definitions involving $2^\omega$; 
to emphasize the distinction between this framework and our setup, we
will use the adjective ``classical'' to refer to these concepts: 
the classical Cicho\'n diagram, classical random reals, etc. 

\begin{ackn}
We thank James Cummings for helpful discussions on indestructibility, and 
Yair Hayut for alerting us to  \cite{Ko:2006}.
\end{ackn}

\subsection*{Higher random reals}

There exists a straightforward generalization the meager ideal on $2^\omega$ (or $\omega^\omega$) to an ideal on $2^\kappa$ for (regular) $\kappa > \omega$, 
using the $\mathord<\kappa$-box product topology  and defining a set as meager if 
it can be covered by  $\le \kappa$ many (closed) nowhere dense sets. 

In \cite{Sh:1004} the third author introduced a generalization\footnote{Unlike \cite{Sh:1004}, we call our uncountable cardinal $\kappa$ rather 
than $\lambda$, mainly to help us resist the temptation of calling
the higher random reals ``$\lambda$andom reals''.}  $\QQ_\kappa$
of the random forcing to~$2^\kappa$ for inaccessible $\kappa$. The forcing $\QQ_\kappa$
is strategically $\kappa$-closed, satisfies the $\kappa^+$-chain condition and
for weakly compact $\kappa$ is $\kappa^\kappa$-bounding.
These are of course three properties that are satisfied by classical 
random forcing (i.e., on$\kappa = \omega$).
The ideal $\id(\QQ_\kappa)$
generated by all $\kappa$-Borel which are forced not to contain the
$\QQ_\kappa$-generic $\kappa$-real turns out to be orthogonal to the ideal $\Cohen_\kappa$
of all $\kappa$-meager sets.

In \cite{CnSh:1085} it is shown how to replace the requirement of $\kappa$ being
weakly compact by assuming the existence
of a stationary set that reflects only in inaccessibles and has a diamond sequence.
A construction of a $\kappa^+$-cc $\kappa^\kappa$-bounding forcing notion 
using a different diamond is given in 
\cite{FL:2017} but it implies $2^\kappa = \kappa^+$, so that setup 
does not allow us to investigate 
cardinal characteristics. 

A different approach can be found in \cite{BBFM:2016}  
where the authors use
the well known characterization of 
the additivity and cofinality of the null ideal by slaloms
(in the classical case
$(\kappa = \omega)$,  see for example \cite{BJ:1995}) to 
define their versions of $\add({\rm null})$ and $\cof({\rm null})$ 
on $2^\kappa$ for inaccessible $\kappa$.

In this paper we continue the work of \cite{Sh:1004}, and we also 
compare our cardinal characteristics to those derived from slaloms. 

\subsection*{Overview of the paper}
\begin{itemize}
\item
In section~\ref{preliminaries}
we repeat some key definitions and results from \cite{Sh:1004}, introduce some notations
and finally define the notion of a strengthened Galois-Tukey connection.
\item
In section~\ref{tools} we prove preservation theorems for iterations of ${<}\kappa$
and $\kappa$-support.
\item
In section~\ref{smaller_ideals} we introduce an ideal $\id^-(\QQ_\kappa) \seq \id(\QQ_\kappa)$ whose definition is slightly simpler than the definition
of $\id(\QQ_\kappa)$; however,  for 
weakly compact $\kappa$ the ideals $\id$ and $\id^-$ coincide.   
We improve the characterizations of the additivity and cofinality of
$\id(\QQ_\kappa)$ given in \cite{Sh:1004} and also give a new characterization of
additivity and
cofinality, using the additivity of the ideal of nowhere stationary sets
on $\kappa$.
\item
In section~\ref{fubini} we generalize a theorem from \cite{Sh:1004}
by introducing the notion
of an anti-Fubini set and showing the existence of such set implies the result for 
arbitrary ideals.
\item
In section~\ref{zfc_results} we repeat and elaborate results from \cite{Sh:1004}
and discuss the  Bartoszy\'nski-Raisonnier-Stern theorem for $\id(\QQ_\kappa)$. 
We can show it for inaccessible $\kappa$ only under additional assumptions,
and we conjecture that it does not hold in general.
\item
In section~\ref{models} we provide six  models separating characteristics of the generalized
Cicho\'n diagram using the tools developed in section~\ref{tools}.
Curiously we do exactly all possible vertical separations.
\item
In section~\ref{slaloms} we repeat some definitions and results from
\cite{BBFM:2016} and use a model from that paper to show that one of
the generalized slalom
characterizations of the additivity of null 
is not provably equal 
to the additivity of $\id(\QQ_\kappa)$.
\end{itemize}

\newpage
\section{Preliminaries} \label{preliminaries}

Some familiarity with the preceding work \cite{Sh:1004} is assumed
but for the convenience of the reader we recall some key definitions and results.
For missing proofs in this section see there.

\subsection{The Generalized Random Forcing $\QQ_\kappa$}

To motivate the main definition of this section, we first give a 
characterisation of random forcing; the definition of $\QQ_\kappa$
can then be seen as a generalization. 

\begin{dfn}
   A ``positive tree on $\omega$'' is a set $T\subseteq 2^{<\omega}$
  with the following properties: 
\begin{itemize}
   \item $T$ is a tree, i.e.: $T$ is nonempty, and for all $t\in T$ and 
    all initial segments $s\trianglelefteq t$ we also have $s\in T$.
  \item There is a family $(N_k:k\in \omega)$, with $N_k \subseteq 2^k$ such that: 
 \begin{itemize}
   \item The sets $N_k$ are small, more precisely: $\sum_{k} \frac{|N_k|}{2^k} < 1$. 
   \item For all $k$, all $s\in 2^k$:  $s\in T\qquad  \Leftrightarrow \qquad
      ( (\forall n< k)\  s\on n\in T \text{ and }s\notin N_k)$. 
 \end{itemize}
\end{itemize}
\end{dfn}
It is easy to see that a tree $T$ is positive in this sense if and only 
if the set $[T]$ of branches of $T$ has positive Lebesgue measure 
in $2^\omega$.  Thus, the set of positive trees is isomorphic to (a dense 
subset of) random forcing. 

It is well-known and easy to see that the ideal of null sets can be
defined from the random forcing in several ways:
\begin{fact}
Let $A\subseteq 2^\omega$. Then each of the following properties 
is equivalent to the statement ``$A$ is Lebesgue measurable with measure~$0$'':
\begin{itemize}
   \item For all positive trees $p$ there is a positive tree $q\subseteq p$
such that $[q]\cap A=\emptyset$.
   \item There is a predense set $\mathcal C$ of positive trees
    such that $A\cap \bigcup_{p\in \mathcal C} [p] = \emptyset$.
   \item There is a single positive tree $p$ such that not only
   $[p]\cap A=\emptyset$, but  for every $s\in 2^{<\omega}$
    we also have $(s+[p]) \cap A = \emptyset$. 
   \\ 
   Here, we write $s+X$ for the set $\{ s+x: x\in X\}$,  where
   $s+x\in 2^\omega$ is defined by $(s+x)(i)=s(i)+x(i)$ for $i\in \dom(s)$,
   and $(s+x)(i)=s(i)$ otherwise.  ($s+X$ is also called a
    ``rational translate'' of~$X$.) 
\end{itemize}
\end{fact}

\begin{dfn}
	\label{r6}
	Unless stated otherwise, $\kappa$ denotes an strongly inaccessible cardinal
	throughout this paper. When we write ``inaccessible'' we will always mean
	``strongly inaccessible'' and for the set of all inaccessible
	cardinals below $\kappa$ we write	
	$$S^\kappa_\inc = \{\lambda < \kappa : \lambda \text{ is inaccessible}\}.$$
\end{dfn}

\begin{dfn}
	\label{r7}
	Let $S \seq \kappa$. We say that $S$ is {\em nowhere stationary} if for every
	$\delta \le \kappa$ of uncountable cofinality the set $S \cap \delta$ is a 
    nonstationary subset of $\delta$.
	Typically we will  only care about being nonstationary in $\delta \in S_\inc^\kappa \cup \{\kappa\}$.
\end{dfn}

We will now inductively define, for every inaccessible cardinal $\kappa$, 
\begin{itemize}
   \item a forcing notion $\QQ_\kappa$ (this definition uses the ideals $\id(\QQ_\delta)$ for $\delta<\kappa$)
   \item two ideals $ \wid(\QQ_\kappa) \subseteq \id(\QQ_\kappa)$ on $2^\kappa$. 
    \\ (The ideals coincide for weakly compact $\kappa$, see \ref{d5}.)
\end{itemize}

\begin{dfn}
	\label{r8}
	We recall the inductively defined forcing $\QQ_\kappa$ from 
    \cite[1.3]{Sh:1004}.
	We have $p \in \QQ_\kappa$ if there exists $(\tau, S, \squ N_\delta:
	\delta \in S \rangle)$  (this tuple is called 
    the \emph{witness} for $p\in \QQ_\kappa$) where:
	\begin{enumerate}
		\item 
		$p \seq \tle \kappa$ is a tree, i.e.\ closed under initial segments.
		\item 
		$\tau\in 2^{<\kappa}$ is the trunk of~$p$, i.e., the least node which has 
          two successors. 
		\item 
		Above $\tau$ the tree $p$ is fully branching, i.e.\
		$\tau \trianglelefteq \eta \in p \Rightarrow \eta ^\frown 0, \eta ^\frown 1 \in p$.
		\item
		$S \seq S_\inc^\kappa$ is nowhere stationary.
		\item 
		For $\delta \in S$ the set $N_\delta\subseteq 2^\delta$
                is ``small'', more precisely:   $N_\delta\in \id(\QQ_\delta)$.
		\item 
		If $\delta \not \in S$ is a limit ordinal and $\eta \in 2^\delta$, then: \ 
		$\eta \in p$ iff $(\forall \sigma < \delta)\  \eta \on \sigma \in p$.
		\item 
		If $\delta \in S$ is a limit ordinal and $\eta \in 2^\delta$, then: \ 
		$\eta \in p$ iff
		\begin{enumerate}
			\item 
			$(\forall \sigma < \delta)\ \eta \on \sigma \in p$ and
			\item 
			$\eta \notin N_\delta$.
		\end{enumerate}
	\end{enumerate}
	For $p, q \in \QQ_\kappa$ we define $q$ stronger than $p$ if $q \seq p$.
	We write $q \leq p$ for ``$q$ stronger than $p$'' throughout this paper
	(and we use this convention for any forcing, not just $\QQ_\kappa$).
	
	If $G$ is a $\QQ_\kappa$-generic filter then we call
	$\eta = \bigcup_{p \in G} \tr(p) \in 2^\kappa$ a $\QQ_\kappa$-generic real or
	a ``random real'', where $\tr(p)$ is the trunk of~$p$.  
	Alternatively, $\eta$ is the unique element of $\bigcap_{p\in G}
	 [p]$, where $[p]$ is the set of cofinal branches of $p$.
\end{dfn}
\begin{rem}
Note that the set $S\cap \lh( \tau)$ (where $\lh(\tau)$ is the order type
of the predecessors of $\tau$) is really irrelevant; if we require 
$\min (S)>\lg(\tau)$, then 
	 $p$ is uniquely defined by
	its witness and vice versa.	
\end{rem}

\begin{rem}\label{rem.compatible}
   Let $p,q\in \QQ_\kappa$.  Then $p$ and $q$ are compatible in  
   $\QQ_\kappa$ iff at least one 
   of the following holds: 
   \begin{enumerate}
   \item  $\tr(p)\trianglelefteq \tr(q) \in p$
   \item  $\tr(q)\trianglelefteq \tr(p) \in q$
   \end{enumerate}
   In particular, two conditions with the same stem are always
   compatible.

Moreover, if $p$ and $q$ are compatible, then $p\cap q $ is 
the weakest condition in $\QQ_\kappa$  which is stronger than both.

As a consequence, any set $\mathcal C\subseteq \QQ_\kappa$ with 
the property 
   $$(\forall \eta\in 2^{<\kappa}) (\exists p\in \mathcal C)\ 
\tr(p)=\eta$$
is predense in $\QQ_\kappa$.
\end{rem}

   For inaccessible $\kappa$ we now define ideals on $2^\kappa$ as follows: 
\begin{dfn}
	\label{r9}
	$ $
	\begin{itemize}
		\item 
		For $\mathcal J \seq \QQ_\kappa$ we define
		$$
		\set_1(\mathcal J) = \bigcup_{p\in \mathcal J} [p],
	\qquad  	
		\set_0(\mathcal J) = 2^\kappa \setmin \set_1(\mathcal J).$$
		\item
		For a collection $\Lambda$ of subsets of $\QQ_\kappa$ we define
		$$
		\set_1(\Lambda) = \bigcap_{\mathcal J \in \Lambda} \set_1(\mathcal J),
		\qquad \set_0(\Lambda) = 2^\kappa \setmin \set_1(\Lambda).$$
	\end{itemize}
\end{dfn}

\begin{dfn}
For $A\subseteq 2^\kappa$: 
\begin{enumerate}
   \item $A\in \wid(\QQ_\kappa)$ iff there is a predense set $\mathcal C\subseteq \QQ_\kappa$ such that $A \subseteq \set_0(\mathcal C)$. 
    \\ Equivalently, $A\in \wid(\QQ_\kappa)$ iff 
    $$ (\forall p\in \QQ_\kappa)(\exists q\in \QQ_\kappa)\ q\le p \text{ and } 
[q] \cap A = \emptyset$$
    (We will discuss the ideal $\wid(\QQ_\kappa)$ in section~\ref{smaller_ideals}.)
   \item $\id(\QQ_\kappa)$ is the $\mathord\le\kappa$-closure of $\wid(\QQ_\kappa)$:
   \\ 
    $A\in \id(\QQ_\kappa)$ iff $A$ can be covered by the union of at most
    $\kappa$ many sets in $\wid(\QQ_\kappa)$. 
    \\ 
    Equivalently, $A\in \id(\QQ_\kappa)$ iff there is a family  $\Lambda$
     of $\kappa$ many 
    predense sets such that $A\subseteq \set_0(\Lambda)$.
\end{enumerate}
\end{dfn}

\begin{thm}
	\label{r10}
	The ideal  $\id(\QQ_\kappa)$ is the ideal of all sets $A$ such that
	there exists a $\kappa$-Borel set $\mathbf B \seq 2^\kappa$ such that
	$A \seq \mathbf B$ and 
	$$
	\QQ_\kappa \forces \dot \eta \not \in \mathbf B
	$$
	where $\dot \eta$ is the canonical generic $\kappa$-real added by $\QQ_\kappa$.

[More explicitly, we should say that
there is a $\kappa$-Borel code $c$ in $\mathbf V$ such that  the corresponding
Borel set $\mathscr B_c$ contains $A$ ($A \subseteq \mathscr B_c$)
and that in the $\QQ_\kappa$-extension, $\eta$ will not be
in the Borel set $\mathscr B_c$, computed in the extension:
      $  \QQ_\kappa \forces \dot \eta \not \in \mathscr B_c$.]
\end{thm}
\begin{proof}
	\cite[3.2]{Sh:1004}.
\end{proof}

\begin{thm}
	\label{r12}
        \ 
	\begin{enumerate}
		\item 
		$\QQ_\kappa$ is $\kappa$-strategically closed.
		\item 
		$\QQ_\kappa$ satisfies the $\kappa^+$-c.c.
		\item 
		If  $\kappa$ is weakly compact, 
        then $\QQ_\kappa$ is $\kappa^\kappa$-bounding.
		\qed
	\end{enumerate}
\end{thm}

\begin{proof}
	\cite[1.8, 1.9]{Sh:1004}.
\end{proof}

\begin{dfn}\label{r12.5}
   For every $\eta\in \tle \kappa $ we write $[\eta]$ for the set
of $x\in 2^\kappa$ extending~$\eta$; these are the basic clopen sets of
the box product topology (i.e., the $\mathord<\kappa$-box product
topology).

Let $\Borel_\kappa$ be the smallest family containing all
clopen sets which is closed under complements and unions/intersections
of at most $\kappa$-many sets. If $\mathbf B \in \Borel_\kappa$  then we
call $\mathbf B$ a $\kappa$-Borel set.

A \emph{Borel code} is a well-founded tree (with a unique root) 
with $\kappa$ many nodes whose leaves are labeled with elements of 
$2^{<\kappa}$; this assigns basic clopen sets to every leaf. This assignment
can be naturally extended to the whole tree:  if the successors 
of a node $\nu$ are labeled with set $(B_i:i\in \kappa)$, then 
$\nu$ is labeled with $2^\kappa \setmin \bigcup_{i<\kappa} B_i$. 

(Equivalently, a Borel code is an infinitary formula
in the propositional language $L_{<\kappa^+}$, where the propositional
variables are identified with the basic clopen sets.)

If $c$ is a Borel code, we write $\mathscr B_c$ for the Borel set associated
with it (i.e., the value of the assignment described above on the root
of the tree~$c$).   
\end{dfn}

\begin{fact}
	\label{r12.6}
	Let $\mathbf V, \mathbf W$ be two universes. Let
	$\eta \in 2^\kappa \cap \mathbf V \cap \mathbf W$
	and let  % $\mathbf B \in \Borel_\kappa \cap \mathbf V \cap \mathbf W$.
    $c$ be a Borel code in $\bf V\cap\bf W$.
	Then it follows from an easy inductive argument on the rank of~$c$:
	$$
	\mathbf V \models \eta \in {\mathscr B}_c \ifff  \mathbf W \models \eta \in \mathscr B_c.
	\qedhere
	$$
\end{fact}
This fact will allow us to speak about Borel sets when we should 
officially speak about Borel codes.

\begin{dfn}
	\label{r13}
	Let $S \seq S_\inc^\kappa$ be nowhere stationary. By $\QQ_{\kappa,S}$
	we mean the forcing that is inductively defined similarly to 
     $\QQ_\kappa$ but additionally
	for $\delta \in S_\inc^{\kappa+1}$ we require
	$p \in \QQ_{\delta, S \cap \delta}$ iff:
	\begin{enumerate}
		\item 
		$p \in \QQ_\delta$.
		\item 
		$p$ is witnessed by some $(\tau, W, \squ \Lambda)$ such that
		$W \seq S \cap \delta$.
	\end{enumerate}
	Note that this definition is different from~\ref{e4}.
\end{dfn}

\subsection{Quantifiers and Rational Translates}

\begin{dfn}
	\label{r4}
	Let $\mu$ be a regular cardinal. We use the following notation:
	
	\begin{itemize}
		\item
		Let $A, B \seq \mu$. We say
		$A \seq^*_\mu B$ if there exists $\zeta < \mu$ such that
		$A \setmin \zeta \seq B$. If $\mu$ is clear from the
		context we write $A \seq^* B$.
		\item
		``$(\exists^\mu \epsilon)\ \phi(\epsilon)$''
		is an abbreviation for
		``$\{\epsilon < \mu : \phi(\epsilon)\}$ is
		cofinal in $\mu$''.
		Similarly ``$(\forall^\mu \epsilon)\ \phi(\epsilon)$''
		is an abbreviation for
		``$\{\epsilon < \mu : \lnot \phi(\epsilon)\}$ is bounded in $\mu$''
		If $\mu$ is clear from the context
		we write $\exists^\infty$ and $\forall^\infty$.
		
		Note that these quantifiers satisfy the usual equivalence
		$$
		(\exists^\mu \epsilon)\ \phi(\epsilon)
		\ifff
		\lnot (\forall^\mu \epsilon)\ \lnot\phi(\epsilon).
		$$
		\item 
		For $\eta, \nu \in 2^\mu$ (or $\mu^\mu$) define
		\begin{enumerate}
			\item 
			$\eta =_\mu^* \nu \quad\Leftrightarrow\quad (\forall^\infty i < \mu)\ \eta(i) = \nu(i)$.
			\item 
			$\eta \leq_\mu^* \nu \quad\Leftrightarrow\quad (\forall^\infty i < \mu)\ \eta(i) \leq \nu(i)$.
		\end{enumerate}		
		and again we may just write $\eta =^* \nu$ and $\eta \leq^* \nu$.
	\end{itemize}
\end{dfn}

\begin{dfn}
	\label{r16}
	We define:
	\begin{enumerate}
		\item
		$\mathfrak b_\kappa = \min\{|B| : B \seq \kappa^\kappa
		\landx (\forall \eta \in \kappa^\kappa)(\exists \nu \in B)\ \lnot (\nu \leq^* \eta) 
		\}$.
		\item
		$\mathfrak d_\kappa = \min\{|D| : D \seq \kappa^\kappa
		\landx (\forall \eta \in \kappa^\kappa)(\exists \nu \in D)\ \eta \leq^* \nu) 
		\}$.
	\end{enumerate}
\end{dfn}

\begin{dfn}
	\label{r5}
	\begin{itemize}
		\item 
		For $p \in \QQ_\kappa$, $\alpha < \kappa$, $\nu \in 2^\alpha$, 
        and $\eta \in p \cap 2^\alpha$ (typically $\tr(p)\trianglelefteq \eta$) 
		we let $p^{[\eta, \nu]} $ be the condition obtained from~$p$
        by first removing all nodes not compatible with $\eta$, and 
        then replacing $\eta$ by $\nu$: 
		$$
		p^{[\eta, \nu]} = \{
		\rho : \rho \trianglelefteq \nu \lor \big((\exists\varrho)\ \eta^\frown \varrho
		\in p \landx \rho = \nu ^\frown \varrho \big)
		\}
		$$
		\item 
		For $\mathcal J \seq \QQ_\kappa$, $\alpha < \kappa$, a permutation $\pi$ of
		$2^\alpha$ let
		$$
		\mathcal J^{[\alpha,\pi]} =
		\{
		p^{[\eta,\nu]} : p \in \mathcal J, \eta \in (p \cap 2^\alpha) ,
		\nu = \pi(\eta)
		\}
		$$
		\item 
		For a collection $\Lambda$ of subsets of $\QQ_\kappa$ and $\alpha < \kappa$.
		$$
		\Lambda^{[\alpha]} = \{
		\mathcal J^{[\alpha, \pi]} : \mathcal J \in \Lambda, 
		\pi \text{ is a permutation of } 2^\alpha
		\}
		$$
	\end{itemize}
	
	Easily $|\Lambda^{[\alpha]}| \leq \kappa + |\Lambda|$.
	If $\Lambda^{\alpha} = \Lambda$ for all $\alpha < \kappa$ we say that
	$\Lambda$ is \emph{closed under rational translates}.
\end{dfn}

\subsection{The Property $\Pr(\cdot)$ and the Nowhere Stationary Ideal}

\begin{dfn}
	\label{r0}
	$\Pr(\kappa)$ means there exists $\Lambda = \{\Lambda_i : i < \kappa \}$
	where $\Lambda_i \seq \QQ_\kappa$ is a maximal antichain (or predense)
	such that for no $p \in \QQ_\kappa$ we have
	$$
	[p] \seq \set_1(\Lambda) = \bigcap_{i < \kappa} \set_1(\Lambda_i).
	$$
	We define
	$$S_{\pr}^\kappa = \{\lambda \in S_\inc^\kappa : \Pr(\lambda)\}.$$
\end{dfn}

\begin{lem}
	\label{r1}
	The set of $p \in \QQ_\kappa$ witnessed by $(\rho, S, \squ \Lambda)$ such that
	$S \seq S_{\pr}^\kappa$ is dense in~$\QQ_\kappa$.
	\qed
\end{lem}

\begin{proof}
	\cite[4.6]{Sh:1004}.
\end{proof}

\begin{lem}
	\label{r2}
	$ $
	\begin{enumerate}
		\item 
		If $\kappa$ is inaccessible but not Mahlo then $\Pr(\kappa)$.
		\item 
		If $\kappa$ is weakly compact then $\lnot \Pr(\kappa)$.
		\item 
		If $\kappa = \sup(S_\inc^\kappa)$ then $\kappa = \sup(S_{\pr}^\kappa)$.
		\item 
		If $\kappa$ is Mahlo then $S_{\pr}^\kappa$ is a stationary subset of~$\kappa$.
		\qed
	\end{enumerate}
\end{lem}

\begin{proof}
	\cite[4.4]{Sh:1004}.
\end{proof}

\begin{dfn}
	\label{r3}
	Define ideals:
	\begin{align*}
	\nst_\kappa = \{
	S \seq S_\inc^\kappa : S \text{ is nowhere stationary}
	\}\\
	\nst_\kappa^{\pr} = \{
	S \seq S^{\pr}_\kappa : S \text{ is nowhere stationary}
	\}	
	\end{align*}
	The order on these ideals is $\seq^*$, i.e.\ set-inclusion modulo bounded subsets.
	Note that by \ref{r2}(4),  for every Mahlo cardinal $\kappa$ the 
    set $S_{\pr}^\kappa$ is stationary;
	so $\kappa$ Mahlo is sufficient for $\nst_\kappa^{\pr}$ to be proper (i.e., $\kappa\notin \nst_\kappa^{\pr}$).
\end{dfn}

\subsection{Ideals and Strengthened Galois-Tukey Connections}

\begin{dfn}
	\label{j6}
     Let $X$ be a set and let $\mi \seq \powset(X)$ be an ideal.
     The equivalence relation $\sim_{\mi}$ on $\powset(X)$ is defined by 
     $A \sim_{\mathbf i} B \Leftrightarrow A \triangle B \in \mathbf i$. 
     We write  $X/\mathord\sim_{\mi}$ for the set of 
     equivalence classes.

     If $\mathbf j$ is an ideal containing~$\mi$, we write 
     $\mathbf j/\mi$ for the naturally induced ideal on $X/\mi$: 
     \[ 
         \mathbf j/\mi:= \{ A/\mathord\sim_\mi\mid A\in \mathbf j\}.
     \]
\end{dfn}

\begin{dfn}
	\label{j1}
	Let $X$ be a set and let $\mathbf i \seq \powset(X)$ be an ideal
       containing all singletons.
	Then:
	\begin{align*}
	\add(\mathbf i) :=& \min\{|\mathcal A| : \mathcal A \seq \mathbf i \landx \cup \mathcal A
	\not \in \mathbf i
	\} \\
	\cov(\mathbf i) :=& \min\{|\mathcal A| : \mathcal A \seq \mathbf i \landx \cup \mathcal A
	= X	
	\} \\
	\non(\mathbf i) :=& \min\{|A| : A \in \powset(X) \setmin \mathbf i	
	\} \\
	\cf(\mathbf i) :=& \min\{|\mathcal A| : \mathcal A \seq \mathbf i \landx 
	(\forall B \in \mathbf i)(\exists A \in \mathcal A)\ B \seq A
	\}.
	\end{align*}

	For two ideals $\mathbf i, \mathbf j \seq \powset(X)$ let
	\begin{align*}
	\add(\mathbf i, \mathbf j) :=& \min\{|\mathcal A| : \mathcal A \seq \mathbf i \landx \cup \mathcal A
	\not \in \mathbf j
	\}\\
	\cf(\mathbf i, \mathbf j) :=& \min\{|\mathcal A| : \mathcal A \seq \mathbf j \landx 
	(\forall B \in \mathbf i)(\exists A \in \mathcal A)\ B \seq A
	\}.	
	\end{align*}
\end{dfn}

\begin{fact}
	\label{j8}
	Let $X$ be a set and let $\mathbf i \seq \powset(X)$ be an ideal. Then
	\begin{enumerate}[(a)]
		\item
		$\add(\mathbf i) \leq \cov(\mathbf i) \leq \cf(\mathbf i)$.
		\item
		$\add(\mathbf i) \leq \non(\mathbf i) \leq \cf(\mathbf i)$.
	\end{enumerate}
\end{fact}

\begin{fact}
	\label{j2}
	Let $X$ be a set and let $\mathbf i^- \seq \mathbf i \seq \powset(X)$ be two ideals. Then:
	\begin{enumerate}[(a)]
		\item
		$\add(\mathbf i) \leq \add(\mathbf i^-, \mathbf i)$.
		\item
		$\add(\mathbf i^-) \leq \add(\mathbf i^-, \mathbf i)$.
		\item
		$\cf(\mathbf i^-, \mathbf i) \leq \cf(\mathbf i)$.
		\item
		$\cf(\mathbf i^-, \mathbf i) \leq \cf(\mathbf i^-)$.\qed
	\end{enumerate}
\end{fact}

\begin{fact}
	\label{j3}
	Let $X$ be a set and let $\mathbf i^- \seq \mathbf i \seq \powset(X)$
	be two ideals.  
Then:
	\begin{enumerate}[(a)]
		\item
		$\add(\mathbf i) \geq \min\{\add(\mathbf i^-), \add(\mathbf i / \mathbf i^-)\}$.
		\item
		$\cf(\mathbf i) \leq \cf(\mathbf i^-) + \cf(\mathbf i / \mathbf
		i^-)$.
	\end{enumerate}
\end{fact}

\begin{dfn}
	\label{j4}
	Consider ideals $\mathbf i^- \seq \mathbf i \seq \powset(X)$, $\mathbf j \seq \powset(U)$
	We call maps
	\begin{enumerate}
		\item 
		$\phi^+ : \mathbf i \to \mathbf j$
		\item 
		$\phi^- : \mathbf j \to \mathbf i^-$
	\end{enumerate}
	a {\em strengthened Galois-Tukey connection} if
	for all $A \in \mathbf i, B \in \mathbf j$:
	$$
	\phi^-(B) \seq A \quad\Rightarrow\quad B \seq \phi^+(A).
	$$
\end{dfn}

\begin{dis}
	\label{j7}
	Strengthened Galois-Tukey connections are a special case of what is
	called a {\em generalized Galois-Tukey connection} in 
	\cite{V:1993} and
	a {\em morphism} in \cite{B:2010}.
		
\end{dis}

\begin{lem}
	\label{j5}
	Consider $\mathbf i^- \seq \mathbf i \seq \powset(X)$, $\mathbf j \seq \powset(U)$
	and let $\phi^-, \phi^+$ be a strengthened Galois-Tukey connection between them.
	Then
	\begin{enumerate}[(a)]
		\item 
		$
		\add(\mathbf i^-, \mathbf i) \leq \add(\mathbf j).
		$
		\item 
		$
		\cf(\mathbf i^-, \mathbf i) \leq \cf(\mathbf j).
		$
	\end{enumerate}
\end{lem}

\begin{proof}
	$ $
	\begin{enumerate}[(a)]
		\item 
		Let $\langle B_\zeta : \zeta < \mu < \add(\mathbf i^-, \mathbf i) \rangle$ be a family
		of $B_\zeta \in \mathbf j$.
		Find $A \in \mathbf i$ such that $\bigcup_{\zeta < \mu} \phi^-(B) \seq A$
		thus $\bigcup_{\zeta < \mu} B_\zeta \seq \phi^+(A)$.
		\item 
		Let $\langle A_\zeta : \zeta < \mu = \cf(\mathbf i^-, \mathbf i) \rangle$ be a
		family of $A_\zeta \in \mathbf i$ cofinal for $\mathbf i^-$.
		Then for $B \in \mathbf j$ we can find $\zeta  < \mu$ such that
		$\phi^-(B) \seq A_\zeta$ thus $B \seq \phi^+(A_\zeta)$, i.e.\
		$\langle \phi^+(A_\zeta) : \zeta < \mu \rangle$ is a cofinal family
		of $\mathbf j$. \qedhere
	\end{enumerate}
\end{proof}

\subsection{Miscellaneous}
\begin{dfn}
	\label{r14}
	Let $X \seq \kappa$. Then
	\begin{enumerate}
		\item
		$\acc(X) := \{
		\alpha < \kappa : (\exists Y \seq X)\ \sup(Y) = \alpha$.
		\item
		$\nacc(X) := X \setmin \acc(X)$.
	\end{enumerate}
\end{dfn}

\begin{dfn}
	\label{r15}
	Let $\id(\Cohen_\kappa)$ be the ideal of meager subsets of $2^\kappa$.
\end{dfn}

\newpage
\section{Tools} \label{tools}
In this section we introduce/recall several concepts and tools that 
will be useful later. In particular, we give sufficent conditions
for the following properties to be preserved in iterations. 
\begin{itemize}
   \item \ref{tools-comp}: Closure properties, such as strategic closure.
   \item \ref{tools-knaster}: Stationary Knaster,      
       a property that is intermediate between 
       the $\kappa^+$-chain condition and $\kappa$-centeredness; this property       is preserved in $\mathord<\kappa$-support iterations. 
   \item \ref{tools-center}: a version of $\kappa$-centeredness.\\
    (Also, similarly to the classical case, sufficiently centered forcing notions
    will not add random reals, and will neither decrease $\non(\QQ_\kappa)$ nor
    increase $\cov(\QQ_\kappa)$.)
    \item \ref{tools-fusion} and \ref{tools-fusion2}: A property defined 
    by a game, which  allows
      fusion arguments in iterations with $\kappa$-support,
     and implies properness and $\kappa^\kappa$-bounding. 
\end{itemize}

\subsection{Closure}
\label{tools-comp}

\begin{dfn}
	\label{a18}
	Let $\QQ$ be a forcing notion.
	We say that $\QQ$ is {\em $\alpha$-closed} if 
	for every descending sequence
	$\langle p_i : i < i^* \rangle$ of length $i^*<\alpha$
	(with all $p_i \in \QQ$)
	there is a lower bound in~$\QQ$, i.e.\
	there exists $q \in \QQ$ such that
	for every $i < i^*$ the condition $q$ is stronger than~$p_i$.
	
	To avoid confusion we may write ${<}\alpha$-closed. 	
\end{dfn}

\begin{dfn}
	\label{24}
	Let $\QQ$ be a forcing notion.
	We say that $\QQ$ is {\em $\alpha$-directed closed} if every    
    directed set  $D \seq \QQ$ of cardinality $ < \alpha$  
     has a lower bound.  (A set $D$ is called directed if
     any two elements of~$D$ are compatible and moreover
     have a lower bound in~$D$.)  
	
	To avoid confusion we may write ${<}\alpha$-directed closed. 	
\end{dfn}

\begin{rem}\label{a23}
It is customary to write $\kappa$-closed and $\kappa$-c.c.\ for $\mathord<\kappa$-closed and $\mathord<\kappa$-c.c., respectively.

An iteration in which the domains of the conditions have 
size $\le \kappa$ should logically be called ``iterations with
$\mathord<\kappa^+$-supports'',
or abbreviated ``$\kappa^+$-supports''.  Convention, however, dictates
that such iterations are called ``iterations with $\kappa$-supports''; we
will follow this convention.  

Most of our forcing iterations will have
$<\kappa$-support and behave similarly to finite support iterations in the
classical case; some of our iterations will have $\kappa$-support, 
in analogy to countable support iterations.
\end{rem}

\begin{dfn}
	\label{a20}
	Let $\QQ$ be a forcing notion and let $q \in \QQ$.
	Define the game $\mathfrak C_\kappa(\QQ,q)$ between two players White and Black taking
	turns playing conditions of~$\QQ$ stronger than~$q$, i.e.\
	first White plays $p_0 \leq q$, then Black plays a condition
	$p'_0 \in \QQ$, then White plays $p_1 \in \QQ$ and so on.
	The game continues for $\kappa$-many turns and note that White plays first in limit steps.
	The rules of the game are:
	\begin{enumerate}
		\item
		For $i < \kappa$ we require $p_i' \leq p_i$.
		\item
		For $i < j < \kappa$ we require
		$p_j \leq p_i'$.
	\end{enumerate}
	White wins
	if he can follow the rules until the end.
		
	We say that $\QQ$ is {\em $\kappa$-strategically closed} if
	White has a winning strategy for $\mathfrak C_\kappa(\QQ,q)$ for every $q \in \QQ$.
\end{dfn}

\begin{fact}
	\label{a21}
	Let $\QQ$ be a forcing notion. Consider the following statements:
	\begin{enumerate}[(a)]
		\item 
		$\QQ$ is ${<}\kappa$-directed closed.
		\item
		$\QQ$ is ${<}\kappa$-closed.
		\item
		$\QQ$ is $\kappa$-strategically closed.
	\end{enumerate}
	Then: (a) $\Rightarrow$ (b) $\Rightarrow$ (c). \qed
\end{fact}

\begin{fact}
	\label{a22}
	Let $\PP = \langle \PP_\alpha, \dot \QQ_\alpha : \alpha < \delta \rangle$ be a forcing iteration
	with ${<}\lambda$-support.
	If for every $\alpha < \delta$ we have
	$
	\PP_\alpha \forces \text{``}
	\dot \QQ_\alpha \models \phi \text{''}
	$
	then also
	$
	\PP \models \phi
	$
	where $\phi \in \{$``${<}\kappa$-directed closed'',
	``${<}\kappa$-closed'', ``$\kappa$-strategically closed''$\}$
    whenever $\lambda\ge \kappa$. In particular, these properties
    are preserved in $\mathord<\kappa$-support iterations and
    in $\kappa$-support iterations.\qed
\end{fact}

\subsection{Stationary Knaster, preservation in ${<}\kappa$-support iterations}
\label{tools-knaster}

\begin{dis}
	\label{a9}
	To obtain independence results for the classical case $(\kappa = \omega)$ we often use
	finite support iterations of c.c.c.\ forcing notions. Such iterations are
	useful due to the well known fact that their finite support limits 
    will again
	satisfy the c.c.c.
	
	In this section we will quote a parallel for the case
	of uncountable $\kappa$, first
	appearing in \cite{Sh:1978}.

\end{dis}

\begin{dfn}
	\label{a0}
Let $\kappa$ be a cardinal.
Let $\QQ$ be a forcing notion. We say that $\QQ$ satisfies the stationary
$\kappa^+$-Knaster condition if for every $\{p_i : i < \kappa^+\} \seq \QQ$ there exists a club
$E \seq \kappa^+$ and a regressive function
$f$ on $E \cap \Ckappa$ such that any $i, j \in E \cap \Ckappa$
we have that
$$f(i) = f(j) \quad\Rightarrow\quad p_i \not \incomp p_j.$$
\end{dfn}

\begin{fact}
	\label{a1}
	The stationary $\kappa^+$-Knaster condition implies the $\kappa^+$-chain condition.
	\end{fact}
	
\begin{proof}	
	By Fodor's pressing down lemma
	the stationary $\kappa^+$-Knaster condition implies that
	for every $\{p_i : i < \kappa^+\} \seq \QQ$ there exists a stationary set $S \seq \kappa^+$ such for that any $i, j \in S$
	the conditions $p_i, p_j$ are compatible.
\end{proof}

\begin{dfn}
	\label{a2}	
Let $\kappa$ be a cardinal.
Let $\QQ$ be a forcing notion. We say that $\QQ$ satisfies $(*_\kappa)$ if the
following holds:
\begin{enumerate}[(a)]
	\item
	$\QQ$ satisfies the stationary $\kappa^+$-Knaster condition.
	\item 
	Any decreasing sequence $\langle p_i : i < \omega \rangle$
	of conditions of $\QQ$ has a greatest lower bound.
	\item
	Any compatible $p, q \in \QQ$ have a greatest lower bound.
	\item
	$\QQ$ does not add elements of $(\kappa^+)^{<\kappa}$ (e.g. $\QQ$ is $\kappa$-strategically closed).
\end{enumerate}
\end{dfn}

\begin{lem}
\label{a4}
Let $\kappa$ be a cardinal. Let $\QQ$ be a forcing notion
such that:
\begin{enumerate}
	\item 
	$\QQ$ satisfies the stationary $\kappa^+$-Knaster condition.
	\item
	$\QQ$ does not add new subsets of $\delta$ for $\delta < \kappa$
	(e.g.\ $\QQ$ is $\kappa$-strategically closed).
\end{enumerate}
Then $\QQ$ does not collapse cardinals.\qed
\end{lem}

\begin{thm}
\label{a6}
Let $\kappa$ be a cardinal.
Let $\langle \PP_\alpha, \name \QQ_\alpha : \alpha < \lambda \rangle$ be a
${<}\kappa$-support iteration such that for every $\alpha < \lambda$
$$\PP_\alpha \forces \name \QQ_\alpha \text{ satisfies } (*_\kappa) 
  \text{ from Definition~\ref{a2}}.$$
Then $\PP_\lambda$ satisfies the stationary $\kappa^+$-Knaster condition.
\end{thm}

\begin{fact}
	\label{a8}
	Let $\kappa$ be a cardinal.
	Let $\QQ$ be a $\kappa$-linked forcing notion. Then $\QQ$ 
	satisfies the stationary $\kappa^+$-Knaster condition.
	\qed
\end{fact}

\begin{proof}
	The proof appears in \cite[3.1]{Sh:1978} for the case $\kappa = \omega_1$ but easily
	generalizes to arbitrary $\kappa$.
\end{proof}

\subsection{$\kappa$-centered$_{{<}\kappa}$, preservation in ${<}\kappa$-support iterations}
\label{tools-center}

\begin{dfn}
	\label{b0}
	Let $\kappa$ be a cardinal,
	let $\PP$ be a forcing notion and let $X \seq \PP$.
	\begin{enumerate}
		\item
		We say that $X$ is linked if for every
		$p_0, p_1 \in X$ we have $p_0 \not \incomp p_1$.
		
		We say that $\PP$ is $\kappa$-linked if there exist
		$\langle X_i : i < \kappa \rangle$ such that $X_i \seq \PP$ is linked and
		$$
		\PP = \bigcup_{i < \kappa} X_i.
		$$
		\item
		We say that $X$ is centered$_{{<}\kappa}$ if for every
		$Y \in [X]^{<\kappa}$ there exists $q$ such that
		$q \leq p$ for every $p \in Y$.
		
		We say that $\PP$ is $\kappa$-centered$_{{<}\kappa}$ if there exist
		$\langle X_i : i < \kappa \rangle$ such that
                each  $X_i \seq \PP$ is centered$_{{<}\kappa}$ and
		$$
		\PP = \bigcup_{i < \kappa} X_i.
		$$
	\end{enumerate}
\end{dfn}

\begin{fact}
	\label{b11}
	Let $\kappa$ be a cardinal and
	let $\PP$ be a forcing notion. Consider the following statements:
	\begin{enumerate}[(a)]
		\item
		$\PP$ is $\kappa$-centered$_{{<}\kappa}$.
		\item
		$\PP$ is $\kappa$-linked.
		\item
		$\PP$ satisfies the $\kappa^+$-c.c.
	\end{enumerate}
	Then: (a) $\Rightarrow$ (b) $\Rightarrow$ (c). \qed
\end{fact}

\begin{dfn}
	\label{b1}
	Let $\kappa$ be a cardinal.
	We say that an iteration
	$\langle \PP_\alpha, \name \QQ_\alpha : \alpha < \zeta \rangle$ is {\em $\kappa$-centered} if it has ${<}\kappa$-support and
	$$\PP_\alpha \forces \name \QQ_\alpha \text{ is } \kappa \text{-centered}_{{<}\kappa}.$$
\end{dfn}

\begin{fact}
	\label{b2}
	Let $\langle \PP_\alpha, \name \QQ_\alpha : \alpha < \zeta \rangle$ be a $\kappa$-centered iteration. Then there exists a sequences
	$\langle \dot C_\alpha : \alpha < \zeta \rangle$,
	$\langle \dot c_\alpha : \alpha < \zeta \rangle$
	such that for all $\dot C_\alpha$ and $ \dot c_\alpha$ are  $\PP_\alpha$-names
	such that $\PP_\alpha$ forces:
	\begin{enumerate}[(a)]
		\item 
		$
		\dot C_\alpha$ is a function $
		\kappa \to \powset(\dot \QQ_\alpha)
		$
		\item
		$
		\ran(\dot C_\alpha) = \dot \QQ_\alpha
		$
		\item
		$
		i < \kappa \Rightarrow \dot C_\alpha(i) \text{ is centered}_{<\kappa}
		$
		\item
		$
		\dot c_\alpha$ is a function $
		\dot \QQ_\alpha \to \kappa
		$ 
		\item
		$\dot q \in \QQ_\alpha \Rightarrow
		\dot q \in \dot C_\alpha(\dot c_\alpha(\dot q))
		$		
	\end{enumerate}
	Without loss of generality we may also assume that each $\dot C_\alpha(n)$ 
        is nonempty and closed
	under weakening of conditions, in particular $1_{\QQ_\alpha} \in \dot C_\alpha(n)$ for each~$n$.
	
	We shall use this notation throughout this section.
\end{fact}

\begin{dfn}	
	\label{b3}
	Let $\PP = \langle \PP_\alpha, \name \QQ_\alpha : \alpha < \zeta \rangle$ be a $\kappa$-centered iteration. We call a condition $p \in \PP$
	{\em fine} if for each $\alpha \in \supp(p)$ the restriction
	$p \on \alpha$ decides some $n < \kappa$ such that $p \on \alpha \forces$
	``$p(\alpha) \in \dot C_\alpha(n)$''.
	Note that for $\alpha \not \in \supp(p)$ this is trivially true because $1_{\QQ_\alpha}$
	is in every~$\dot C_\alpha(n)$.
\end{dfn}

\begin{dfn}
	\label{b4}
	Let $\PP = \langle \PP_\alpha, \name \QQ_\alpha : \alpha < \zeta \rangle$
	be a $\kappa$-centered iteration. We say that $\PP$ is
	{\em finely ${<}\kappa$-closed} if for every $\alpha < \zeta$
	with $\cf(\alpha) < \kappa$ there exist
	$L_\alpha^1 \in \mathbf V$ and a $\PP_\alpha$-name $\dot L_\alpha^2$ such that:
	\begin{enumerate}[(a)]
		\item
		$L_\alpha^1$ is a function $\kappa^{<\kappa} \to \kappa$
		\item
		$\PP_\alpha \forces$``$\dot L_\alpha^2$ is a function $\dot \QQ_\alpha^{<\kappa} \to \dot \QQ_\alpha$.''
		\item
		If $\squ q = \langle \dot q_i : i < i^* \rangle$ is a descending sequence
		of length $i^* < \kappa$ 
         in $\dot \QQ_\alpha$ then $\PP_\alpha$ forces:
		\begin{enumerate}[(1)]
			\item 
			$\dot L_\alpha^2(\squ q)$ is a lower bound for $\squ q$.
			\item
			$\dot c_\alpha(\dot L_\alpha^2(\squ q)) = L_\alpha^1(
			\langle \dot c_\alpha(\dot q_i) : i < i^* \rangle
			)$.
		\end{enumerate}
	\end{enumerate}
	
	The typical situation here is that the coloring of the forcing is essentially
	some trunk function so if we find a lower bound $q$ for some descending sequence
	$\langle \dot q_i : i < \alpha \rangle$ the union of the trunks of the $p_i$ will
	tell us the color of~$q$.
\end{dfn}

\begin{lem}
\label{b5}
Let $\PP = \langle \PP_\alpha, \name \QQ_\alpha : \alpha < \zeta \rangle$ be a
$\kappa$-centered finely ${<}\kappa$-closed iteration of length
$\zeta < (2^\kappa)^+$ then:
\begin{enumerate}[(a)]
	\item
	$\PP' = 
	\{
	p \in \PP : p \text{ is fine}
	\}
	$
	is dense in~$\PP$.
	\item 
	$\PP$ is~$\kappa$-centered$_{{<}\kappa}$.
\end{enumerate}
\end{lem}

\begin{dis}
	\label{b5.1}
	The following proof closely follows \cite{B:2011} where the result is explained for
	the $\omega$-case. The only adjustment we have to make is the demand for
	fine closure (as defined in~\ref{b4}) to deal with the limit case that does not appear in the $\omega$-version of the proof.
	
	This theorem also appears in \cite{BBFM:2016}.
\end{dis}

\begin{proof}
	$ $
\begin{enumerate}[(a)]
	\item 	
	Let $p \in \PP$ be arbitrary.  We are going to find
	a condition $p'$ stronger than
	$p$ such that $p'$ is fine.
	We prove this by induction on $\delta \leq \zeta$ for
	$\PP_\delta$, constructing a decreasing sequence of conditions
	$\langle p_i : i \leq \delta \rangle$ with $p_i \in \PP_\delta$
	such that for each $i \leq \delta$ the condition $p_i \on (i+1)$ is fine:
	\begin{enumerate}[(i)]
		\item 
		$p_0 = p$
		\item
		$i = j + 1$:
		First find $q$ stronger than $p_i \on i$ such that $q$ decides the color of
		$p_j(i)$. Then use the induction hypothesis to find $q' \leq q$ such that
		$q'$ is fine and let $p_i =  q' \landx p$.
		\item
		$i$ a limit ordinal, $\cf(i) < \kappa$:
		Consider the condition
		$$
		q' = (\dot L_j^2(
		\langle q_k(j) : k < i \rangle)
		 : j < i) \in \PP_i
		$$
		and let $p_i = q' \landx p$.
		\item
		$i$ a limit ordinal,  $\cf(i) \geq \kappa$: 
		Remember that $\PP$ has ${<}\kappa$-support so this case is trivial.
	\end{enumerate}

	\item
	By the Engelking-Kar\l{}owicz theorem \cite{EK:1965}
	there exists a family of functions
	$
	\langle f_i : \zeta \to \kappa\ |\ i < \kappa \rangle
	$
	such that for any $A \in [\zeta]^{<\kappa}$ and every
	$f : A \to \kappa$ there exists $i < \kappa$ such that
	$f \seq f_i$.	
	
	For each $k < \kappa$ let
	$$D(i) = \{p \in \PP_\zeta: \forall \alpha < \kappa:
	p \on \alpha \forces p(\alpha) \in \dot C_\alpha(f_i(\alpha))\}.$$
	It is easy to see that each $D(k)$ is centered$_{{<}\kappa}$
	and that every fine $p \in \PP$ is contained in some~$D(i)$.
	So by (a) we are done. \qedhere
\end{enumerate}
\end{proof}

\begin{lem}
	\label{b7}	
	Let $\kappa$ be an inaccessible cardinal with $\sup(\kappa \cap S_\kappa^\inc) = \kappa$.	
	Let $\PP$ be a forcing notion that does not add new subsets of $\delta$ for $\delta < \kappa$
	(e.g.\ $\PP$ is $\kappa$-strategically closed).
	Then $\PP$ does not add a $\QQ_\kappa$-generic real if either:
	\begin{enumerate}[(a)]
		\item
		$\PP$ is $\kappa$-centered$_{{<}\kappa}$ or just
		\item
		$\PP$ is $(2^\kappa, \kappa)$-centered$_{{<}\kappa}$ meaning that
		any set $Y \seq \PP$ of cardinality at most $2^\kappa$ is included in
          the union of
		at most $\kappa$-many centered$_{{<}\kappa}$ subsets of $\PP$ or just
		\item
		if $p_\rho \in \PP, \rho \in 2^\kappa$ is a family of conditions, then for some
		non-meager $A \seq 2^\kappa$ we have
		$$
		u \in [A]^{<\kappa} \Rightarrow \{p_\rho : \rho \in u\} \text{ has a lower bound.}
		$$
	\end{enumerate}
\end{lem}

\begin{proof}	
	Clearly (a)$\Rightarrow$(b)$\Rightarrow$(c). The first implication is trivial.  
	The second implication follows from the $\kappa^+$-completeness of the meager
	ideal.
	So we shall assume (c).
	
	Let $p^* \forces$ ``$\dot \nu$ is a counterexample and thus
	$\dot \nu \on \epsilon \in \mathbf V$ for all
	$\epsilon < \kappa$''. (Recall that $\QQ_\kappa$ is strategically $\kappa$-closed.)
	Let $\langle \lambda_\epsilon : \epsilon < \kappa \rangle$ be an increasing
	enumeration of $\{\lambda \in S^\kappa_\inc : \lambda > \sup(\lambda \cap S^\kappa_\inc)\}$.
	Now for $\eta \in 2^\kappa$ let
	$$A_\eta = \{ \rho \in 2^\kappa :
	(\forall^\infty \epsilon < \kappa)\ 
	(\exists^\infty \alpha < \lambda_\epsilon)\ 
	\eta(\alpha) \neq \rho(\alpha)\}.$$
	Clearly $2^\kappa \setmin A_\eta \in \id^-(\QQ_\kappa) \seq \id(\QQ_\kappa)$
	as defined in~\ref{d3} but we may argue
	$2^\kappa \setmin A_\eta \in \id(\QQ_\kappa)$ as follows:
	For $\eta \in 2^\kappa$ and
	$\epsilon < \kappa$ let 
	$B_{\eta, \epsilon} = \{\rho \in 2^{\lambda_\epsilon} : \rho =^* \eta\} $
	and note that $|B_{\eta, \epsilon}| = \lambda_\epsilon$, hence
	$B_{\eta,\epsilon} \in \id(\QQ_{\lambda_\epsilon})$.
	Let $S = \{\lambda_\epsilon : \epsilon < \kappa\}$ and clearly $S$
	is nowhere stationary. So for every $\eta \in 2^\kappa$ the set
	$$
	\mathcal J_\eta = \{p \in \QQ_\kappa : S \seq S_p \landx (\forall \epsilon < \kappa)\ [
	\lambda_\epsilon > \lh(\tr(p)) \Rightarrow B_{\eta, \epsilon} \in
	\set_0(\Lambda_{p, \lambda_\epsilon}) ] \}
	$$ is dense in $\QQ_\kappa$ and $p \in \mathcal J_\eta \Rightarrow
	p \forces$``$\nu \in A_\eta$''.
	
	Now because $2^\kappa \setmin A_\zeta \in \id(\QQ_\kappa)$
	we have $p^* \forces$``$\dot \nu \in A_\zeta$'' hence
	for $\eta \in 2^\kappa$ there are $(p_\eta, \zeta_\eta)$ such that
	$p_\eta \leq p^*$, $\zeta_\eta < \kappa$, and
	$$p_\eta \forces_\PP \text{ ``if } \epsilon \in [\zeta_\eta, \kappa) \text{ then }
	(\exists^\infty \alpha < \lambda_\epsilon) \eta(\alpha) \neq \dot \nu(\alpha)\text{''}.$$
	
	Hence there exists a non-meager set $Y \seq 2^\kappa$ such that any
	set $\{p_\rho : \rho \in Y\}$ of cardinality ${<}\kappa$ has a lower bound.
	Because the meager ideal is $\kappa^+$-complete there exists
	$\zeta^* < \kappa$ such that without loss of generality
	$\eta \in Y \Rightarrow \zeta_\eta = \zeta^*$.
	As $Y$ is non-meager it is somewhere dense. So there exists
	$\varrho^* \in \tle \kappa$ such that
	$$(\forall \varrho \in \tle \kappa)\ \varrho^* \triangleleft \varrho \in \tle \kappa
	\Rightarrow (\exists \rho \in Y)\ \varrho \triangleleft \rho.$$
	
	Without loss of generality
	$\lh(\varrho^*) = \zeta^*$ (we may increase either value to match the greater one).
	Choose $\epsilon < \kappa$ such $\lambda_\epsilon > \zeta^*$. Let
	$\Gamma = \{\varrho  \in 2^{\lambda_\epsilon} : \varrho^* \triangleleft \varrho\}$
	and for each $\varrho \in \Gamma$ let $\eta_\varrho \in Y$ be such that
	$\varrho \triangleleft \eta_\varrho$. Now $\{\eta_\varrho : \varrho \in \Gamma\}
	\in [Y]^{<\kappa}$ hence by the choice of $Y$ there exists a lower bound
	$q$ of $\{p_{\eta_\varrho} : \varrho \in \Gamma\}$.
	
	As $p^* \forces$ ``$\dot \nu \on \epsilon \in V$'' without loss of generality let
	$q$ force a value to~$\dot \nu \on \epsilon$, so call this value~$\nu$.
	Now $q$ is stronger than $p_{\eta_{\varrho^*{^\frown}\nu\on[\epsilon,\lambda_\epsilon)}}$
	and forces
	$\lambda_\epsilon = \sup\{\alpha < \lambda_\epsilon :
	\varrho^*{^\frown}\nu\on[\epsilon,\lambda_\epsilon)(\alpha) \neq \dot \nu(\alpha)\}$,
	which means $\lambda_\epsilon = \sup\{\alpha < \kappa : \nu(\alpha) \neq \dot \nu(\alpha)\}$.
	Contradiction to the choice of~$\nu$.
\end{proof}
\begin{rem}
  Lemma~\ref{b7} implies that $\QQ_\kappa$ is not $\kappa$-centered$_{<\kappa}$.  However, $\QQ_\kappa$ has, for every $\lambda<\kappa$, a dense subset 
  which is  $\kappa$-centered$_{<\lambda}$, namely the set of conditions 
  with trunk of length $>\lambda$.   This parallels the classical case of
  random forcing, which is not $\sigma$-centered, but $\sigma$-$n$-linked for
  all $n\in\omega$.
\end{rem}

\begin{dis}
	\label{b12}
	The following theorem~\ref{b8} is a straightforward 
         generalization of \cite[6.5.30]{BJ:1995}.
	We formulate it in terms of the ideal $\id^-(\QQ_\kappa) \seq \id(\QQ_\kappa)$.
	For the definition see~\ref{d3}.
\end{dis}
\begin{lem}
	\label{b8}
		Let $\kappa$ be weakly compact. Let
		$\PP$ be a  forcing notion such that
		\begin{enumerate}[(a)]
			\item 
			 $\PP$ is $\kappa$-centered$_{<\kappa}$.
			 \item 
			 $\PP$ does not add new subsets of $\delta$ for $\delta < \kappa$
			 (e.g.\ $\PP$ is $\kappa$-strategically closed).
		\end{enumerate}
		Let $(\mathbf N, \in) \prec (H(\chi), \in)$ for some $\chi$ large enough with $\mathbf N^\kappa \seq \mathbf N$ and $\PP \in \mathbf N$.
		Then for $A \in \id^-(\QQ_\kappa)$
		we have
		$$
		\mathbf N \cap 2^\kappa \seq A \quad\Rightarrow\quad
		\PP \forces \text{``}\mathbf N[G] \cap 2^\kappa \seq A\text{''}
		$$
		where $G$ is the generic filter of $\PP$. (As usual, $A$ is to be read
as a definition of a null set, to be interpreted in $\bf V$ and~$\bf V^{\PP}$.)
\end{lem}

\begin{proof}
	Let $A \in \id^-(\QQ_\kappa)$ be witnessed by
	$\squ A = \langle A_\delta : \delta \in S \rangle$, i.e.\ 
	$A = \set_0^-(\squ A)$, and let $\PP = \bigcup_{\alpha < \kappa} \PP_\alpha$
	and each $\PP_\alpha$ is centered$_{<\kappa}$.
	
	Assume there exists $\PP$ name of a $\kappa$ real
	$\dot \eta \in \mathbf N$ and $p^* \in \PP$
	such that
	$$
	p^* \forces \text{``} \dot \eta \not \in A\text{''}
	$$
	and without loss of generality even
	\begin{equation}
	\label{b8eq1}
	p^* \forces \text{``} (\forall \delta \geq \delta_0)\ 
	\dot \eta \on \delta \not \in A_\delta
	\text{''}
	\end{equation}
	for some~$\delta_0 < \kappa$.
	For $\alpha < \kappa, \delta \in S$ we define
	$$
	T_{\alpha,\delta} = \{
	\nu \in 2^\delta : 	(\forall p \in \PP_\alpha)
	(\exists q \in \PP)\ q \leq p \text{ and } q \forces \text{``}
	\dot \eta \on \delta = \nu
	\text{''}
	\}.
	$$
	 
	 Note that in general we will have~$p^*\notin \mathbf N$.   However, we will
	 have $p^*\in \PP_\alpha$ for some~$\alpha$, and the partition 
	 $(P_\alpha:\alpha<\kappa)$ is in~$\mathbf N$, as is the 
	 family~$(T_{\alpha,\delta}: \alpha<\kappa, \delta\in S)$.

	None of the sets $T_{\alpha, \delta}$ (for all~$\alpha < \kappa$,
	$\delta \in S$) is empty.  We prove this indirectly:
	Assume~$T_{\alpha,\delta} = \emptyset$.
	Then for every $\nu \in 2^\delta$ there exists $p_\nu \in \PP_\alpha$
	such that~$p_\nu \forces \nu  \neq \dot \eta \on \delta$.
	Now because $\PP_\alpha$ is centered$_{<\kappa}$ there exists a lower bound $q$
	for $\{
	p_\nu : \nu \in 2^\delta
	\}$.
	Thus for all $\nu \in 2^\delta$ we have $q \forces \nu \neq \eta \on \delta$,
	contradicting  our assumption that $\PP$ does not add short sequences.

	For $\alpha < \kappa$ consider the tree that is the downward closure of
	$\bigcup_{\delta \in S} T_{\alpha, \delta}$. Because
	$\kappa$ is weakly compact, $\kappa$ has the tree property
	thus there exists a branch $\eta_\alpha \in 2^\kappa$
	through this tree, i.e.\  for every $\delta \in S$
	we have~$\eta_\alpha \on \delta \in T_{\alpha, \delta}$.
	Note that $f_\alpha$ can be calculated from~$\dot \eta$
	hence $f_\alpha \in \mathbf N$ so by our assumption
	$\eta_\alpha \in A$, i.e.\  $(\exists^\infty \delta \in S)\ \eta_\alpha \in A_\delta$.
	Find $\alpha^* < \kappa$ such that $p^* \in \PP_{\alpha^*}$
	and find $\delta^* \geq \delta_0$ such that
	$\eta_{\alpha^*} \on \delta^* \in A_{\delta^*}$.
	
	Let $\nu = \eta_{\alpha^*} \on \delta^* \in T_{\alpha^*, \delta^*}$
	so there exists $q \leq p^*$ such that
	$$
	q \forces  \dot \eta \on \delta^* = \nu = \eta_{\alpha^*} \on \delta^*
	\in A_\delta^*.
	$$
	Contradiction to (\ref{b8eq1}).	
\end{proof}

\begin{cor}
	\label{b9}
	Let $\kappa$ be weakly compact. Let
	$\PP$ be  a forcing notion such that
	\begin{enumerate}[(a)]
		\item 
		$\PP$ is $\kappa$-centered$_{<\kappa}$.
		\item 
		$\PP$ does not add new subsets of $\delta$ for $\delta < \kappa$
		(e.g.\  $\PP$ is $\kappa$-strategically closed).
	\end{enumerate}
	Then:
	\begin{enumerate}[(1)]
	\item 
	$\PP$ does not decrease~$\non(\QQ_\kappa)$, i.e.\ 
	if $\non(\QQ_\kappa) = \lambda$ then $\PP \forces$``$\non(\QQ_\kappa) \geq \lambda$''.
	\item 
	$\PP$ does not increase~$\cov(\QQ_\kappa)$, i.e.\ 
	if $\cov(\QQ_\kappa) = \lambda$ then $\PP \forces$``$\cov(\QQ_\kappa) \leq \lambda$''.
	\end{enumerate}	
\end{cor}

\begin{proof}
	$ $
	\begin{enumerate}
		\item 
		Let $\mu < \lambda$ and assume $\PP \forces$``$X = \{
		\dot \eta_i : i < \mu
		\}$ is a set of size $\mu$''.
		Find $\mathbf N$ as in $\ref{b8}$ with $\dot \eta_i \in \mathbf N$ for each
		$i < \mu$ and~$|\mathbf N| = \mu$.
		Now because $\kappa$ is weakly compact by~\ref{d6}
		we have $\mu < \non(\id^-(\QQ_\kappa))$ so find
		$A \in \id^-(\QQ_\kappa)$ such that~$\mathbf N \cap 2^\kappa \seq A$.
		By~\ref{b8} we have $\PP \forces$``$X \seq \mathbf N[G] \seq A$''.
		I.e.: For any set $X \in \mathbf V^\PP$ of size $\mu < \lambda$
		we have~$X \in \id^-(\QQ_\kappa)$.
		\item 
		We show: $\PP$ does not add a $\QQ_\kappa$-generic real.
		Assume $\PP \forces$``$\dot \eta$ is $\QQ_\kappa$-generic''. Find
		$\mathbf N$ as in~\ref{b8} with $\dot \eta \in \mathbf N$ and 
		$|\mathbf N| = \kappa$. Find $A \in \id^-(\QQ_\kappa)$ be such that
		$\mathbf N \cap 2^\kappa \seq A$. Now by~\ref{b8}
		we have $\PP \forces$``$\dot \eta \in \mathbf N[G] \seq A \in \id^-(\QQ_\kappa)
		\seq \id(\QQ_\kappa)$'',  a 
		contradiction to~$\dot \eta$ being $\QQ_\kappa$ generic.
		 \qedhere 
	\end{enumerate}
\end{proof}

\begin{rem}
	\label{b10}
	So~\ref{b9}(2) duplicates~\ref{b7} but there we do not require $\kappa$ weakly compact.
\end{rem}

\subsection{The Fusion Game, preservation in $\kappa$-support iterations}
\label{tools-fusion}

The work in this subsection can be considered a generalization of \cite[Section 6]{K:1980},
where it is shown how to iterate $\kappa$-Sacks forcing for inaccessible~$\kappa$.
The games defined in this subsection and the iteration theorem~\ref{h3} first appeared in \cite{RoSh:860} where~$\mathfrak F_\kappa^*$, $\mathfrak F_\kappa$
(defined below)
are called $\Game^{\text{rcA}}_{\squ \mu}$ and $\Game^{\text{rca}}_{\squ \mu}$
respectively.
However~$\mathfrak F_\kappa^*$, $\mathfrak F_\kappa$ are slightly more general in the sense
that White may freely decide the length $\mu_\zeta$ of the $\zeta$-th round
during the game (i.e.\  our games do not require an additional parameter~$\squ \mu$).

\begin{dfn}
	\label{h0}
	Let $\QQ$ be a forcing notion and let $q \in \QQ$. We define two (very similar) games
	$\mathfrak F_\kappa(\QQ,q), \mathfrak F_{\kappa}^*(\QQ,q)$ between two players White and Black.
	A play in either of the games 
        consists of $\kappa$-many rounds and for each $\zeta < \kappa$ the
	$\zeta$-th round lasts $\mu_\zeta$-many moves.
	The rules of the $\zeta$-th round of the game $\mathfrak F_\kappa(\QQ,q)$ are:
	\begin{enumerate}
		\item
		First White plays $0 < \mu_\zeta < \kappa$.
		So White decides the length of the new round.
		\item
		On move $i < \mu_\zeta$:
		\begin{enumerate}
			\item
			White plays $q_{\zeta,i} \leq q$.
			\item
			Black responds with $q'_{\zeta,i} \leq q_{\zeta,i}$
		\end{enumerate}
	\end{enumerate}
	
	The rules of the $\zeta$-th round of the game $\mathfrak F_{\kappa}^*(\QQ,q)$ are:
	\begin{enumerate}
		\item
		First White plays $0 < \mu_\zeta < \kappa$.
		For $\zeta$ a limit ordinal we additionally
		require $\mu_\zeta \leq \sup_{\epsilon < \zeta} \mu_\epsilon$.
		\item
		On move $i < \mu_\zeta$:
		\begin{enumerate}
			\item
			White plays $q_{\zeta,i} \leq q$ but without looking at any
			$q'_{\zeta,j}$ for~$j < i$. (Equivalently: White plays all moves
			of the current round at once
			at the start of the round.)
			\item
			Black responds with $q'_{\zeta,i} \leq q_{\zeta,i}$
		\end{enumerate}
	\end{enumerate}
	The winning condition of both games is the same:
	$$
	\text{White wins} \quad\Leftrightarrow\quad
	(\exists q^* \leq q)\ q^* \forces \text{``}
	(\forall \zeta < \kappa)\ \{
	q'_{\zeta, i} : i < \mu_\zeta
	\} \cap \dot G_\QQ \neq \emptyset
	\text{''}.
	$$	
	where $\dot G_\QQ$ is a name for the generic filter of~$\QQ$.
\end{dfn}

\begin{dis}
	\label{h7}
	In point (1.) of the definition of
	$\mathfrak F_{\kappa}^*(\QQ, q)$ we could be
	slightly more general: Instead of $\sup$ any function $f: \kappa^{<\kappa} \to \kappa$
	that gives us an upper bound for $\mu_\zeta$ based on upper bounds
	for the $\mu_\epsilon$ will do.
	(This is simply a technical requirement for the proof of~\ref{h3}.)
	So we could define $\mathfrak F_{\kappa,f}^*(\QQ,q)$
	and let $\mathfrak F_{\kappa}^*(\QQ,q) = \mathfrak F_{\kappa,f}^*(\QQ,q)$.
\end{dis}

\begin{fact}
	\label{h1}
	The game $\mathfrak F_{\kappa}^*$ is slightly harder for White than the
	game $\mathfrak F_\kappa$ hence:
	If White has a winning strategy for $\mathfrak F_{\kappa}^*(\QQ,q)$
	then
	White has a winning strategy for $\mathfrak F_\kappa(\QQ,q)$.
\end{fact}

\begin{dfn}
	\label{h4}
	For technical reasons we define the game
	$\mathfrak F_{\kappa}^*(\QQ,q,\lambda)$ for $\lambda < \kappa$.
	The rules are the same as for $\mathfrak F_{\kappa}^*(\QQ,q)$ except the
	first $\lambda$ rounds are skipped and the game starts with the $\lambda$-th round.
	So this is really just an index shift.
	Of course $\mathfrak F_{\kappa}^*(\QQ,q) = \mathfrak F_{\kappa}^*(\QQ,q,0)$ and
	easily
	for every $\lambda < \kappa$ White has a winning strategy for
	$\mathfrak F_{\kappa}^*(\QQ,q)$ iff he has a winning strategy for $\mathfrak F_{\kappa}^*(\QQ,q,\lambda)$.
\end{dfn}

\begin{fact}
	\label{h5}
	Assume White has a winning strategy for
	$\mathfrak G \in \{\mathfrak F_{\kappa}(\QQ,q),
	\mathfrak F_{\kappa*}(\QQ,q)\}$.
	Then without loss of generality during a run of $\mathfrak G$
	White only plays moves $q_{\zeta, i}$ such that there exists
	$\theta_{\zeta,i} \in \prod_{\epsilon < \zeta} \mu_\epsilon$ with
	\begin{enumerate}[1.]
		\item 
		$\epsilon < \delta < \zeta \Rightarrow
		q'_{\epsilon, \theta_{\zeta,i}(\epsilon)}
		\leq q_{\delta, \theta(\delta)}$.
		\item
		$\epsilon < \zeta \Rightarrow
		q'_{\epsilon, \theta_{\zeta,i}(\epsilon)}
		\leq q_{\zeta, i}.$
	\end{enumerate}
	Consider the tree
	$$
	T = \bigcup_{\zeta < \kappa}\bigcup_{i < \mu_\zeta} \theta_{\zeta, i}.
	$$
	Then a condition $q^*$ witnesses a win for White iff
	$q^* \forces$``for every $\zeta < \kappa$ there exists a
	branch $\dot \theta$ of $T$ of length $\zeta$ such that for every $\epsilon < \zeta$
	we have $q'_{\epsilon,\dot \theta(\epsilon)} \in \dot G_\QQ$''.
	
\end{fact}

\begin{thm}
	\label{h2}
	Let $\QQ$ be a forcing notion.	If
	for every $q \in \QQ$
	Black does not have a winning strategy
	for the game
	$\mathfrak F_\kappa(\QQ,q)$
	then:
	\begin{enumerate}[(a)]
		\item
		If $\dot A$ is a $\QQ$-name such that
		$q \forces$``$|\dot A| \leq \kappa$'' then there exists
		$B \in \mathbf V$, $|B| \leq \kappa$ and $q^* \leq q$
		such that $q^* \forces \dot A \seq B$. 
		
		In particular $\QQ$ does not collapse~$\kappa^+$.
		\item
		$\QQ$ does not increase~$\cf(\Cohen_\kappa)$, and in fact: 
		if $\langle A_i : i < \mu \rangle$ are a cofinal family of meager sets in
		$\mathbf V$ then this family remains cofinal in $\mathbf V^\QQ$.	
		
		\item 
		$\QQ$ is $\kappa^\kappa$-bounding.
	\end{enumerate}
\end{thm}

\begin{proof}
	$ $
	\begin{enumerate}[(a)]	
		\item
		Like (b), just easier. But let us do it for warmup.
		
		Let $\langle \dot a_\zeta : \zeta < \kappa \rangle$ be such that
		$q \forces \{a_\zeta : \zeta < \kappa\} = \dot A$.
		Now consider a run of $\mathfrak F_\kappa(\QQ,q)$ where Black's
		strategy is to play in such way that
		for each $\zeta < \kappa, i < \mu_\zeta$ there is $b_{\zeta,i}$
		such that $q'_{\zeta, i} \forces$``$\dot a_\zeta = _{\zeta,i}$''.
		I.e., every move Black makes during the $\zeta$-th round decides
		$\dot a_\zeta$.
		
		By our assumption White can beat this strategy thus there exists
		$q^* \leq q$ such that
		$q^* \forces \dot A \seq \{b_{\zeta,i} : \zeta < \kappa, i < \mu_\zeta < \kappa
		\}$.
		\item
		Let us show: if $\dot M$ is a $\QQ$-name and $q \forces$``$\dot M$ is nowhere dense''
		then there exists a nowhere dense set $N \in \mathbf V$ and $q^* \leq q$
		such that $q^* \forces    \dot M \seq N$. Since meager sets are the union of
		$\kappa$-many nowhere dense sets, we can then use (a) to 
        conclude the proof. 
		
		We are going to find $q^* \leq q$ such that for each $s \in \tle \kappa$
		there exists $t_s \trianglerighteq s$ such that $q^* \forces$''$\dot M \cap [t] = \emptyset$'' so
		$$
		N = 2^\kappa \setmin \bigcup_{s \in \tle \kappa} [t_s]
		$$
		is as desired.

		Let $\langle s_\zeta : \zeta < \kappa \rangle$
         be an enumeration of $\tle  \kappa$.
       We will define a strategy for player Black. In addition to 
       his moves $(q'_{\zeta,i}$, he will construct elements $t_{\zeta,i}\in 
       \tle \kappa$ satisfying the following properties:
		% Consider a run of the game $\mathfrak F^\kappa_\QQ(q)$ where Black's strategy
		% for the $i$-th move of the $\zeta$-th round is to find
		% $t_{\zeta, i} \in \tle \kappa$ such that
		\begin{enumerate}
			\item
			$s_\zeta \trianglelefteq t_{\zeta, j}$.
			\item
			$(\bigcup_{j < i} t_{\zeta, j}) \trianglelefteq t_{\zeta, i}$.
			\item
			$q'_{\zeta, i} \forces$``$\dot M \cap [t_{\zeta, i}] = \emptyset$''.
             (and of course $q'_{\zeta,i}\le q_{\zeta,i}$, as required by
             the rules of the game).
		\end{enumerate}
		Why can Black play like that?
		\begin{enumerate}
			\item
			Obvious.
			\item
			Obvious for $i$ successor.
			For $i$ a limit ordinal just remember $i < \mu_\zeta < \kappa$.
			\item
			Remember $q'_{\zeta, i} \leq q \forces$``$\dot M$ is nowhere dense''.
		\end{enumerate}
		Let $t_\zeta = \bigcup_{i < \mu_\zeta} t_{\zeta, i}$. Again White
		can beat this strategy so there exists $q^* \leq q$ as required.
		\item 
		Like (b).
		\qedhere
	\end{enumerate}
\end{proof}

\begin{thm}
	\label{h3}
	Let $\PP = \langle \PP_\alpha, \dot \QQ_\alpha : \alpha < \alpha^* \rangle$
	be a $\kappa$-support iteration and let $p \in \PP$ such that for all
	$\alpha < \alpha^*$:
	\begin{enumerate}[(a)]
		\item
		$p \on \alpha \forces$``$\dot \QQ_\alpha$
		is $\kappa$-strategically closed''.
		\item
		$p \on \alpha \forces$``White has a winning strategy for
		$\mathfrak F_{\kappa}^*(\dot \QQ_\alpha,q)$ for every
		$q \leq p(\alpha)$''.
	\end{enumerate}
	Then:
	\begin{enumerate}[(1)]
		\item
		White has a winning strategy for $\mathfrak F_\kappa(\PP,p)$.
		\item
		If White plays according to his strategy from (1) in a run
		$\langle p_{\zeta, i}, p'_{\zeta, i} : \zeta < \kappa, i < \mu_\zeta \rangle$
		of $\mathfrak F_\kappa(\PP,p)$
		then there exists $p^*$ witnessing White's win such that for all $\alpha < \alpha^*$
		we have
		$
		p^* \on \alpha \forces $``$\langle p_{\zeta,i}(\alpha), p'_{\zeta, i}(\alpha) : \zeta < \kappa, i < \mu_\zeta \rangle$
		is a run of $\mathfrak F_{\kappa}^*(\dot \QQ_\alpha,p(\alpha))$
		won by White and White's win is witnessed by $p^*(\alpha)$''.
	\end{enumerate}
\end{thm}

\begin{dis}
	\label{h8}
	Note that the proof of~\ref{h3} also works for $\kappa = \omega$.
\end{dis}

\begin{proof}
	Let $p \in \PP$ and we are going to show how White can win
	$\mathfrak F_\kappa(\PP,p)$ by finding $p^* \leq p$ witnessing
	White's victory while also being as required by (2). We are going
	to construct at sequence
	$\langle p_\zeta : \zeta \leq \kappa \rangle$ such that
	\begin{enumerate}
		\item
		$\zeta < \kappa \Rightarrow p_\zeta \in \PP$.
		\item
		$p_0 = p$.
		\item
		$\epsilon < \zeta \Rightarrow p_\epsilon \geq p_\zeta$.
	\end{enumerate}
	of which $p^*$ is going to be a lower bound (but remember that under our assumptions
	the lower bound of a $\kappa$-sequence does not exist in general
	so we will have to construct~$p^*$).
	We are also going to construct a sequence
	$\langle F_\zeta : \zeta < \kappa \rangle$ such that
	\begin{enumerate}
		\item
		$F_0 = \emptyset$
		\item
		$\zeta < \kappa \Rightarrow F_\zeta \seq \supp(p_\zeta)$.
		\item
		$\zeta < \kappa \Rightarrow |F_\zeta| < \kappa$.
		\item
		$\epsilon < \zeta \Rightarrow F_\epsilon \seq F_\zeta$.
	\end{enumerate}
	and we are going to use bookkeeping to ensure
	$F = \bigcup_{\zeta < \kappa} F_\zeta = \bigcup_{\zeta < \kappa} \supp(p_\zeta)$
	which is also going to be the support of~$p^*$.
	
	Furthermore	we are implicitly going to construct strategies for Black in the games
	$\mathfrak F_{\kappa}^*(\dot \QQ_{\alpha},p(\alpha))$ for $\alpha \in F$.
	Then we will choose $p^* = \langle \dot q^*_\alpha : \alpha \in F \rangle$
	where $\dot q^*_\alpha$ witnesses that White can beat Black's strategy.
			
	\underline{What does White play in the $\zeta$-th round?}
	
	Let $\langle \alpha_{\zeta,\xi} : \xi < \xi^*_\zeta \rangle$
	enumerate~$F_\zeta$.
	For $\xi < \xi^*_\zeta$ we want to play the $\zeta$-th round of
	the game $\mathfrak F_{\kappa}^*(\dot \QQ_{\alpha_{\zeta,\xi}},p(\alpha_{\zeta,\xi}))$ where White plays
	according to the name of a winning strategy (White sticks to same strategy throughout the proof of course).
	To make notation easier we do not want to
	keep track of when $\alpha_{\zeta,\xi}$
	first appeared $F_\epsilon$ for some $\epsilon \leq \zeta$.
	Instead let $\epsilon_{\zeta, \xi} = \min\{
	\epsilon  \leq \zeta: \alpha_{\zeta, \xi} \in F_\epsilon
	\}$ and assume we are playing
	$\mathfrak F_{\kappa}^*
	(\dot \QQ_{\alpha_{\zeta,\xi}},p_{\epsilon_{\zeta, \xi}}(\alpha_{\zeta,\xi}),
	\epsilon_{\zeta,\xi})$.
	I.e., we are playing in the $\zeta$-th round for each $\alpha_{\zeta,\xi}$.
	See~\ref{h4}.
	
	By induction (we are going to address this further down) we assume
	for each $\xi < \xi^*_\zeta$ that $p_\zeta \on \alpha_{\zeta,\xi} \forces$
	``$\dot \mu_{\alpha_{\zeta, \xi}, \zeta} \leq \mu_{\alpha_{\zeta, \xi}, \zeta}$''
	for some $\mu_{\alpha_{\zeta, \xi}, \zeta} < \kappa$ where $\dot \mu_{\alpha_{\zeta, \xi}, \zeta}$
	is the length of $\zeta$-th round of  $\mathfrak F_{\kappa}^*
	(\dot \QQ_{\alpha_{\zeta,\xi}},p_{\epsilon_{\zeta, \xi}}(\alpha_{\zeta,\xi}),
	\epsilon_{\zeta,\xi})$ as decided by
	the name of White's winning strategy. Then there exist (in $\mathbf V$ where we are trying to construct a winning strategy) not necessarily injective
	enumerations $\langle \dot q_{\alpha_{\zeta,\xi},\zeta, i} : i < \mu_{\alpha_{\zeta, \xi}, \zeta} \rangle$
	of the moves that White plays according to the name of his winning strategy
	in the $\zeta$-th round of $\mathfrak F_{\kappa}^*
	(\dot \QQ_{\alpha_{\zeta,\xi}},p_{\epsilon_{\zeta, \xi}}(\alpha_{\zeta,\xi}),
	\epsilon_{\zeta,\xi})$. To make notation easier
	easier we only do the proof for the special case where
	White always plays an antichain (but the proof works even if White doesn't).
	
	Let $\mu_\zeta = |\prod_{\xi < \xi_\zeta^*} \mu_{\alpha_{\zeta, \xi}, \zeta}|$ and this is what
	White decides to be the length of the $\zeta$-th round of $\mathfrak F_\kappa(\PP,p)$.
	Remember that $\kappa$ is inaccessible so indeed $\mu_\zeta < \kappa$.
	Let $\langle \lambda_{\zeta,i} : i < \mu_\zeta \rangle$ enumerate $\prod_{\xi < \xi_\zeta^*} \mu_{\alpha_{\zeta, \xi}, \zeta}$.
	Now we construct a sequence $\langle p_{\zeta, i} : i < \mu_\zeta \rangle$
	(of course anything that is not explicitly stated to be done by Black is part
	of White's strategy that we are currently constructing):
	
	\begin{enumerate}
		\item
		First we find $p_{\zeta,0} \leq p_{\epsilon}$ for every $\epsilon < \zeta$
		as follows:
		\begin{itemize}
			\item 
			If there is no $\xi < \xi^*_\zeta$ such that
			$\alpha = \alpha_{\zeta,\xi}$ then let
			$p_{\zeta,0}(\alpha)$ be such that
			$p_{\zeta,0} \on \alpha \forces    p_{\zeta,0}(\alpha) \leq
			p_{\epsilon}(\alpha)$ according to a winning strategy for White in
			$\mathfrak C(\QQ_\alpha)$.
			\item 
			If there is $\xi < \xi^*_\zeta$ such that
			$\alpha = \alpha_{\zeta,\xi}$ then let
			$p_{\zeta,0}(\alpha)$ be such that
			$p_{\zeta,0} \on \alpha \forces$``$p_{\zeta_0}(\alpha) =
			\bigvee_{\gamma < \mu_{\alpha,\zeta}} \dot q_{\alpha, \zeta, \gamma}$''.
			
			Remember~\ref{h5} so without loss of generality this implies\\
			$p_{\zeta,0} \on \alpha \forces$``$p_{\zeta_0}(\alpha) \leq~p_\epsilon(\alpha)$''.
		\end{itemize}
		\item
		For the $i$-th move of the $\zeta$-th round White plays $p_{\zeta,i}'$ where
		$$
		p'_{\zeta, i}(\alpha) = \begin{cases}
			p_{\zeta,i}(\alpha_{\zeta,\xi}) \landx \dot q_{\alpha_{\zeta,\xi},\zeta,\lambda_i(\xi)}
			& \text{if } \alpha = \alpha_{\zeta,\xi} \text{ for some } \xi < \xi^*_{\zeta} \\
			p_{\zeta,i}(\alpha) & \text{otherwise.}
		\end{cases}		
		$$
		\item
		Black responds with $p''_{\zeta,i} \leq p'_{\zeta,i}$.
		\item
		Let $p'''_{\zeta, i}$ be such that for $\alpha < \alpha^*$ we have
		\begin{align*}
		p'''_{\zeta, i} \on \alpha \forces \text{``}& 
		p'''_{\zeta, i}(\alpha) \leq p''_{\zeta, i}(\alpha)
		\text{ and } p'''_{\zeta, i}(\alpha) \text{ is a according to}\\
			&\text{a winning
			strategy for White in } \mathfrak C(\dot \QQ_\alpha) \text{''}.
		\end{align*}
		\item
		Let $p''''_{\zeta,i}$ be defined by
		$$
		p''''_{\zeta, i}(\alpha) = \begin{cases}
		(p_{\zeta,i}(\alpha_{\zeta,\xi}) \setmin
		\dot q_{\alpha_{\zeta,\xi},\zeta,\lambda_i(\xi)l}) \lor
		p'''_{\zeta,i}(\alpha_{\zeta,\xi})
		& \text{if } \alpha = \alpha_{\zeta,\xi} \text{ for some } \xi < \xi^*_{\zeta} \\
		p'''_{\zeta,i}(\alpha) & \text{otherwise.}
		\end{cases}		
		$$
		and easily check $p''''_{\zeta,i} \leq p_{}$.
		\item
		If $i = j + 1$ then let $p_{\zeta, i} = p'''_{\zeta, j}$. If $i$ is a
 limit ordinal, then 
		we find 
		$p_{\zeta, i} \leq p_{\zeta,j}$ for every $j < i$
		as follows:
		\begin{itemize}
			\item
			If there is no $\xi < \xi^*_\zeta$ such that
			$\alpha = \alpha_{\zeta,\xi}$ then let
			$p_{\zeta,i}(\alpha)$ be such that
			$p_{\zeta,i} \on \alpha \forces$``$p_{\zeta,i}(\alpha)$ is according to
			a winning strategy for White in $\mathfrak C(\dot \QQ_\alpha)$ for the sequence
			$\langle p_{\zeta, j}(\alpha) : j < i \rangle$''.
			\item
			If there is $\xi < \xi^*_\zeta$ such that
			$\alpha = \alpha_{\zeta,\xi}$ then let
			$p_{\zeta,i}(\alpha)$ be such that
			$$
			p_{\zeta,i} \on \alpha \forces\text{``}p_{\zeta,i}(\alpha) =
			\bigvee_{\gamma < \mu_{\alpha, \zeta}}
			\dot r_{\zeta, i, \alpha, \gamma}\text{''}
			$$
			where $p_{\zeta,i} \on \alpha \forces$``$\dot r_{\zeta,i,\alpha,\gamma}$ is according to
			a winning strategy for White in $\mathfrak C(\dot \QQ_\alpha)$ for the sequence
			$\langle p_{\zeta, j}(\alpha) \landx
			\dot q_{\alpha, \zeta, \gamma} : j < i \rangle$''.
		\end{itemize}
	\end{enumerate}
	
	Finally let $p_\zeta$ be a lowerbound for $\langle p_{\zeta,i} : i < \mu_\zeta \rangle$
	as in 6.
	(but not really, we have to do some preparation work for the next step first, see below).
	Now the strategy for Black in $\mathfrak F_{\kappa}^*(\dot \QQ_{\alpha_{\zeta,\xi}},p(\alpha_{\zeta,\xi}))$ is to play
	$p_\zeta(\alpha_{\zeta,\xi}) \landx \dot q_{\alpha_{\zeta,\xi},\zeta,i}$.
		
	\underline{Preparation for the $\zeta+1$-th round.}
	
	We still have to address why the $\mu_{\alpha_{\zeta, \xi}, \zeta}$ exist	
	but having understood the proof to this point this is now easy.
	Let $F_{\zeta+1} = F_\zeta \cup \{\alpha\}$
	for some $\alpha \in \supp(p_\zeta) \setmin F_\zeta$,
	if such $\alpha$ exists	(and remember to use bookkeeping).
	Now for every $\alpha \in F_{\zeta+1}$ work as above on
	$p_{\zeta} \on \alpha$ and $F_{\zeta} \cap \alpha$ but instead of taking
	a response from Black in (3.) White responds to himself deciding~$\mu_{\alpha,\zeta+1}$.
	
	So we have prepared for~$\zeta+1$.
	But what about limit steps? Remember that the rules of $\mathfrak F^*_\kappa$ state that
	$\dot \mu_{\alpha,\zeta} \leq \sup_{\epsilon < \zeta} \dot \mu_{\alpha,\epsilon}$.	
	So if we let $F_{\zeta}	= \bigcup_{\epsilon < \zeta} F_\epsilon$ all is good
	because having an estimate for successor steps gives us an estimate for limit steps.
	
	\underline{Why does White win?}
	
	Because for $\alpha \in F = \bigcup_{\zeta < \kappa} F_\zeta$
	there exists a $\QQ_\alpha$-name $\dot q^*_\alpha$ such that
	$p \on \alpha \forces$``$q^*_\alpha$ witnesses that White wins if Black plays as described above in
	$\mathfrak F_{\kappa}^*(\dot \QQ_{\alpha},p(\alpha))$''.
		
	By construction  $p^* = \langle \dot q^*_\alpha : \alpha \in F \rangle$ is as required.
\end{proof}

\subsection{Fusion and Properness}
\label{tools-fusion2}
In this subsection we give a sufficient condition for a limit of 
a $\mathord\le\kappa$-support iteration to be $\kappa$-proper, namely,
the existence of winning strategies  for the games
		$\mathfrak F_{\kappa}^*(\dot \QQ_\alpha)$ 
for all iterands $\QQ_\alpha$ encountered in the iteration.  

We also show that if all iterands have cardinality $\le \kappa^+$, 
and the length $\delta$
 of the iteration is $<\kappa^{++}$, then the resulting
forcing $\PP_\delta$ has a dense set of size $\kappa^+$ and in particular
will still satisfy the $\kappa^{++}$-cc.   

\begin{dfn}
	\label{s11}
	In this section we consider an iteration
	$\PP = \langle \PP_\alpha, \QQ_\alpha : \alpha < \delta\rangle$
	with limit $\PP_\delta$
	such that:
	\begin{enumerate}
		\item 
		$\delta < \kappa^{++}$
		\item 
		$\PP$ has $\kappa$-support.
		\item 
		White has a winning strategy for
		$\mathfrak F_{\kappa}^*(\dot \QQ_\alpha, \dot q)$
		for every $\alpha < \delta$ and $\dot q \in \dot \QQ_\alpha$.
		\item 
		In $\PP_\alpha$ the forcing $\QQ_\alpha$ has size at most $\kappa^+$.
	\end{enumerate}	
	For $\alpha < \delta$ let $\dot b_\alpha$ be a $\PP_\alpha$-name of a one-to-one map
	from~$\kappa^+$ onto~$\dot \QQ_\alpha$.
\end{dfn}

\begin{lem}
	\label{s12}
	Let $(\mathbf N, \in)$ a model of size~$\kappa$, closed under ${<}\kappa$-sequences,
	let $\RR$ be an arbitrary forcing notion such that
	$\RR \in \mathbf N$
	and $(\mathbf N, \in) \prec (H(\chi), \in)$ for some $\chi$ large enough.
	If White has a winning strategy
	for	$\mathfrak F_{\kappa}(\RR,p)$	
	then for every $p \in \RR \cap \mathbf N$ there exists $q^* \in \RR$,
	$q^* \leq p$ such that $q^*$ is $\mathbf N$-$\RR$-generic.
	This means:
	\begin{enumerate}
		\item 
		For every maximal antichain $A$ of $\RR$ with
		$A \in \mathbf N$ we have
		$$
		q^* \forces A \cap \mathbf N \cap \dot G_\RR \neq \emptyset.
		$$
		\item 
		Or equivalently: for every name $\dot \tau$ of an ordinal
		with $\dot \tau \in \mathbf N$ we have
		$$
		q^* \forces \dot \tau \in \mathbf N.
		$$	
	\end{enumerate}
\end{lem}

\begin{proof}
	Note that because $|\mathbf N| = \kappa$ there are at most $\kappa$-many
	names of ordinals in~$\mathbf N$. 
	By our assumption White has a winning strategy for $\mathfrak F_{\kappa}(\RR,p)$
	and because $\mathbf N$ is an elementary submodel White has a winning strategy
	that lies in $\mathbf N$.
	Now consider a run of the game where:
	\begin{enumerate}
		\item 
		White plays according to his winning strategy in $\mathbf N$.
		By induction all these moves are in $\mathbf N$ by our assumption
		$\mathbf N^{<\kappa} \seq \mathbf N$.
		\item 
		Black decides all ordinals
		of $\mathbf N$ such that they lie in $\mathbf N$ by playing
		$p'_{\zeta,i} \in \mathbf N$ for $\zeta < \kappa, i < \mu_\zeta$.
	\end{enumerate}
	Now $q^*$ witnessing White's win is $\mathbf N$ generic.	
\end{proof}

\begin{dfn}
	\label{s25}
	Let $\RR$ be a forcing notion. Consider a run of the game $\mathfrak G \in~\{
	\mathfrak F_\kappa, \mathfrak F_{\kappa}^*
	\}$
	where:
	\begin{enumerate}
		\item White wins.
		\item
		Black plays $\squ p' = \langle p'_{\zeta,i}: \zeta < \kappa, i < \mu_\zeta \rangle$.
	\end{enumerate}
	Then we call $q^*$ witnessing White's win a {\em $\mathfrak G$-fusion limit} of $\squ p'$.	
\end{dfn}

\begin{cor}
	\label{s13}
	Let $\PP$ be as in~\ref{s11}. Then:
	\begin{enumerate}[(a)]
		\item
		For every $p \in \PP \cap \mathbf N$ there exists a generic condition $q^* \leq p$
		that is a $\mathfrak F_\kappa(\PP)$-fusion limit of $\squ p'$
		with $p'_{\zeta,i} \in \mathbf N$ for all $\zeta < \kappa$, 
                $ i < \mu_\zeta$.
                (However, in general we will have $q^* \not \in \mathbf N$.)
		\item
		Furthermore for $\alpha < \delta$ we have $q^* \on \alpha \forces$``$q^*(\alpha)$
		is a $\mathfrak F_{\kappa}^*(\dot \QQ_\alpha)$-fusion limit''.
	\end{enumerate}
\end{cor}

\begin{proof}
	$ $
	\begin{enumerate}[(a)]
		\item
		By~\ref{s11}(3) and~\ref{h3}(1) White has a winning strategy for
		$\mathfrak F_{\kappa}(\RR,p)$ so use~\ref{s11}.
		\item
		By \ref{h3}(2).
	\end{enumerate}
\end{proof}

\begin{dfn}
	\label{s14}
	For $\alpha < \kappa$ a condition $p \in \PP_\alpha$ is called a 
	\HC-condition if for every $\beta < \alpha$ the $\PP_\beta$-name 
	$p(\beta)$ is a \HC-$\PP_\beta$-name.
	
	For $\alpha < \delta$ we inductively define the notion of a \HC-$\PP_\alpha$-name.
	On the one hand we consider \HC-names for elements of $\kappa^+$,
	on the other hand for elements of~$\dot \QQ_\alpha$.
	\begin{enumerate}
		\item
		$\dot \tau$ is a \HC-name for an element of $\kappa^+$ iff
		$\dot b_\alpha(\dot \tau)$ is a \HC-name of an element of~$\dot \QQ_\alpha$.  ($b_\alpha$ was defined in \ref{s11}.)
		\item  
		 For every $\gamma \in \kappa^+$, the standard name 
		$\check \gamma$ is a \HC-name.
		\item
		For every sequence $\langle (p_i, \dot \tau_i) : i < \kappa \rangle$
		where $p_i$ are \HC-$\PP_\alpha$-conditions and $\dot \tau_i$ are 
		\HC-$\PP_\alpha$-names there exists a \HC-name $\dot \tau$ forced to be equal
		to~$\dot \tau_i$ where $i$ is the least index such that $p_i \in \dot G_\PP$
		if such $i$ exists, $\check 0$ otherwise.
		\item
		For every $\mathfrak F_{\kappa}^*(\dot \QQ_\alpha)$-fusion sequence
		$\squ p\,'$ where $p'_{\zeta, i}$ are \HC-$\PP_\alpha$-names for
		elements of~$\dot \QQ_\alpha$ there exists a \HC-name
		$\dot \tau$ that is forced to be equal to the condition witnessing White's win.
		(If it exists; $\check 0$ otherwise.)
	\end{enumerate}
\end{dfn}
\begin{rem}
The ``\HC''-names are an easy generalization of the  ``hereditarily 
countable'' names appearing in \cite[4.1]{pif}, see also 
\cite{GK:2016}.
\end{rem}

\begin{lem}
	\label{s15}
	For every condition $p \in \PP$ there exists a \HC-condition $q^* \leq p$.	
\end{lem}

\begin{proof}
	First let $\mathbf N$ be a model of size $\kappa$ with $p, \PP \in \mathbf N$ and let
	$q^*$ be a $\mathfrak F_\kappa(\PP)$-fusion limit
	with $p'_{\zeta, i} \in \mathbf N$ as in~\ref{s13}.
	
	Now we will try to find a \HC-name for $p'_{\zeta, i}(\alpha)$, for all
	$\zeta, \alpha < \kappa, i < \mu_\zeta$.
	
	For $\alpha \in \supp(q^*)$ we define $p''_{\zeta,i}(\alpha)$ as follows.
	We find (in $\mathbf N$) a maximal antichain
	$A = A_{\zeta, i, \alpha}$ that decides
	$\dot b_\alpha^{-1}(p'_{\zeta, i}(\alpha))$, i.e.\ there exists a function
	$f = f_{\zeta, i, \alpha} \colon A \to \kappa^+$, such that
	for all $r \in A$
	$$
	r \forces p'_{\zeta,i}(\alpha) = \dot b_\alpha(f(r)).
	$$
	Let $A' = A \cap \mathbf N$. 
	Consider the sequence $\langle (r, b_\alpha(f(r))) : r \in A' \rangle$.
	This family defines a \HC-name~$p''_{\zeta,i}(\alpha)$.
	
	Now because $q^* \on \alpha$ is $\mathbf N$-generic
	$$
	q^* \on \alpha \forces p'_{\zeta,i}(\alpha) = p''_{\zeta,i}(\alpha)
	$$
	Hence $q \on \alpha$ forces
	that $q^*(\alpha)$ is equal to a witness of White's win against $p''_{\zeta,i}(\alpha)$,
	i.e. $q^*(\alpha)$ is a $\mathfrak F_{\kappa}^*(\QQ_\alpha)$-fusion limit.
	Hence $q^*(\alpha)$ is a \HC-name so $q^*$ is a \HC-condition.
	
\end{proof}

\begin{cor}
	\label{s16}
	Let $\PP_\delta$ be as in~\ref{s11} (so particular $\delta < \kappa^{++}$).
	Then there exists $D \seq \PP_\delta$ such that
	\begin{enumerate}
		\item 
		$D$ is dense.
		\item
		$|D| = \kappa^+$.
		\item
		$\PP_\delta$ has the $\kappa^{++}$-c.c.
	\end{enumerate}
\end{cor}

\begin{proof}
	Follows immediately from~\ref{s15}.
\end{proof}

\newpage
\section{Smaller Ideals} \label{smaller_ideals}
In this section we first describe two ideals 
$\wid(\QQ_\kappa)$ and $\id^-(\QQ_\kappa)$, both of which are closely
related (and often equal)  to~$\id(\QQ_\kappa)$. We then give a 
more ``combinatorial'' characterization of $\add(\QQ_\kappa)$ and
$\cof(\QQ_\kappa)$, involving the additivity and cofinality
of the ideal $\nst_\kappa^{\pr}$ of nowhere stationary subsets of $S_{\pr}^\kappa \subseteq 
\kappa$.

\subsection{The ideal $\wid(\QQ_\kappa)$}

\begin{dfn}
	\label{d0}
	For $\id(\QQ_\kappa)$ we allow $\kappa$ many antichains to define
	$A \in \id(\QQ_\kappa)$. But we may also consider the weak ideal
	$\wid(\QQ_\kappa)$ of all sets $A \seq 2^\kappa$ such that
	for some maximal antichain $\mathcal A$ (or equivalently: 
        every predense set $\mathcal A$) we have 
	$A \seq \set_0(\mathcal A)$, where
        $\set_0(\mathcal A):= 2^\kappa \setmin \bigcup_{p \in \mathcal A}[p]$.
\end{dfn}

\begin{lem}
	\label{d1}
	$\ $
	\begin{enumerate}[(a)]
		\item
		$\wid(\QQ_\kappa) \seq \id(\QQ_\kappa)$.
		\item
		$\wid(\QQ_\kappa) = \id(\QQ_\kappa)$ iff
		$\lnot \Pr(\kappa)$.
		\item
		$\wid(\QQ_\kappa)$ is $\kappa$-complete.
	\end{enumerate}
\end{lem}

\begin{proof}
	$\ $
	\begin{enumerate}[(a)]
		\item
		Trivial: If $\mathcal A$ witnesses $A \in \wid(\QQ_\kappa)$ then
		$\Lambda = \{\mathcal A\}$ witnesses $A \in \id(\QQ_\kappa)$.
		\item
		Assume $\lnot \Pr(\kappa)$.
		Let $\Lambda$ be a set of at most $\kappa$-many maximal antichains
		of $\QQ_\kappa$ and without loss of generality assume that
		$\Lambda$ is closed under rational shifts, i.e.\ for all
		$\eta_1, \eta_2 \in 2^\kappa$ we have
		$$
		\eta_1 =^* \eta_2 \quad\Rightarrow\quad
		[\eta_1 \in \set_0(\Lambda) \Leftrightarrow \eta_2 \in \set_0(\Lambda)].
		$$
		Let $A \seq \set_0(\Lambda)$.
		By our assumption about $\kappa$ there exists $p \in \QQ_\kappa$ such that
		$[p] \seq \set_1(\Lambda)$ and let $p$ be witnessed by
		$(\tau, S, \squ \Gamma)$. Let
		$$
		\mathcal A = \{q \in \QQ_\kappa: q \text{ is witnessed by }
		(\rho, S, \squ \Gamma) \text{ for some } \rho \in \tle \kappa \}
		$$
		and check that $\mathcal A$ is predense. Now easily
		$q \in \mathcal A \Rightarrow [q] \seq \set_0(\Lambda)$ hence
		$\set_1(\mathcal A) \seq \set_1(\Lambda)$ hence
		$A \seq \set_0(\mathcal A)$, i.e.\ $A \in \wid(\QQ_\kappa)$.
		
		Conversely assume $\wid(\QQ_\kappa) = \id(\QQ_\kappa)$ and let
		$\Lambda$ be a set of no more than $\kappa$-many maximal antichains of
		$\QQ_\kappa$. By our assumption there exists a maximal antichain $\mathcal A$
		of $\QQ_\kappa$ such that
		$$
		\bigcup_{p \in \mathcal A} [p] = \set_1(\mathcal A) \seq \set_1(\Lambda).
		$$
		Hence for any $p \in \mathcal A$ we have $[p] \seq \set_1(\Lambda)$;
		as $\Lambda$ was arbitrary, we get  $\lnot \Pr(\kappa)$.
		\item
		Because $\QQ_\kappa$ is strategically $\kappa$-closed.
		\qedhere
	\end{enumerate}
\end{proof}

\begin{lem}
	\label{d2}
	Consider the usual forcing ideal
	$$
	\fid(\QQ_\kappa) = \{A \seq 2^\kappa : (\forall p \in \QQ_\kappa) 
	(\exists q \leq p)\ [q] \cap A = \emptyset\}.
	$$
	Then we have $\fid(\QQ_\kappa) = \wid(\QQ_\kappa)$.
\end{lem}

\begin{proof}
	Let $A \in \wid(\QQ_\kappa)$ be witnessed by $\mathcal A$.
	Now for any $p \in \QQ_\kappa$ there exists $p' \in \mathcal A$
	such that $p$ and $p'$ are compatible. Let $q = p \cap p'$
	and clearly $A \cap [q] = \emptyset$, hence $A \in \fid(\QQ_\kappa)$
	
	Conversely if $A \in \fid(\QQ_\kappa)$ then the set
	$\mathcal D = \{q : [q] \cap A = \emptyset\}$
	is dense. Choose any maximal antichain $\mathcal A \seq \mathcal D$,
	then  $\mathcal A$ will witness  $A \in \wid(\QQ_\kappa)$.
\end{proof}
	
\subsection{The ideal $\id^-(\QQ_\kappa)$}

Recall from that the ideal is generated by sets $\set_0(\mathcal A)$, where
$\mathcal A\subseteq \QQ_\kappa$ is any predense set. Recall also 
(from~\ref{rem.compatible}) that the set of all rational translates
(see~\ref{r5}) 
of any fixed condition is predense. 
This suggests the following definition:

\begin{dfn}
	\label{d3}
	The ideal $\id^-(\QQ_\kappa)$ consists of all sets $A \seq 2^\kappa$ for which
	there exists a condition $p$ such that 
       $A \subseteq  \set_0(\{ s+[p]: s\in 2^{<\kappa}\})$. 

Equivalently, $A\in \id^-(\QQ_\kappa)$ iff there are 
\begin{itemize}
   \item 
 a nowhere stationary set $S \seq S^\kappa_\inc$ 
   \item and
	a sequence $\squ N = \langle N_\delta : \delta \in S \rangle$
	such that each $ N_\delta$ is a ``rather small'' subset of $2^\delta$
    (in the sense that $N_\delta$ is in $\id(\QQ_\kappa)$) 
\end{itemize}
	such that
	\[
	A \seq \set_0^-(\squ N ) := 
	\{\eta \in 2^\kappa: (\exists^\infty \delta \in S)\ \eta \on \delta \in N_\delta \}.
   \]
\end{dfn}
	
	Note that we are often lazy and use the notation $\add(\QQ_\kappa)$.
	This always means $\add(\id(\QQ_\kappa))$, never $\add(\id^-(\QQ_\kappa))$.
	The same applies for $\cov, \non$ and $\cf$.

\begin{lem}
	\label{d4}
	$\id^-(\QQ_\kappa) \seq \wid(\QQ_\kappa)$.
\end{lem}

\begin{proof}
	Given $S \seq S^\kappa_\inc$ and $\squ \Lambda = \langle \Lambda_\delta : \delta \in S \rangle$
	let $p_\rho \in \QQ_\kappa$ be the condition witnessed by $(\rho, S, \squ \Lambda)$ and let
	$\mathcal D = \{p_\rho : \rho \in \tle \kappa\}$.
	It is easy to check that $\set_0^-(\squ \Lambda) \seq \set_0(\mathcal D)$.
\end{proof}

\begin{thm}
	\label{d5}
	Let $\kappa$ be a weakly compact cardinal. Then $\id^-(\QQ_\kappa) = \wid(\QQ_\kappa)$.
\end{thm}

\begin{lem}
	\label{d7}
	$\id^-(\QQ_\kappa)$ is ${<}\kappa^+$-complete.
\end{lem}

\begin{proof}[Proof of Lemma~\ref{d7}]
	For $i < \kappa$ let $(S_i, \squ \Lambda^i)$ represent $A_i =\set_0^-(\squ \Lambda_i)
	\in \id^-(\QQ_\kappa)$.
	Let
	$$
	S^* = \{\delta < \kappa : (\exists i < \delta)\ \delta \in S_i\}
	$$
	be the diagonal union of $S_i$ and for $\delta \in S^*$ let
	$
	\Lambda^*_\delta = \cup\{
	\Lambda_{i,\delta} : i < \delta
	\}
	$
	and easily
	$$
	\bigcup_{i < \kappa} A_i \seq \set_0^-(\squ \Lambda^*).
	$$
\end{proof}

\begin{proof}[Proof of Theorem~\ref{d5}]
	Let $D = \{p_\epsilon : \epsilon < \kappa\} \seq \QQ_\kappa$ be a maximal antichain
	witnessing  $A \seq \set_0(D) \in \wid(\QQ_\kappa)$. For $\epsilon < \kappa$
	let $p_\epsilon$ be witnessed by $(\tau_\epsilon, S_\epsilon, \squ \Lambda_\epsilon)$
	Using weak compactness we find a sequence 
	$\langle \delta_\alpha : \alpha < \kappa \rangle$ such that
	\begin{enumerate}
		\item
		$\delta_\alpha \in S^\kappa_\inc$.
		\item
		$\delta_\alpha > \sup_{\beta<\alpha}\delta_\alpha$.
		\item
		$D_\alpha = \{p_\epsilon \cap \tle{\delta_\alpha} : \epsilon < \delta_\alpha\}$ is a maximal
		antichain in~$\QQ_{\delta_\alpha}$.
	\end{enumerate}
	Let
	$$
	S^*_\alpha = (\bigcup_{\epsilon < \delta_\alpha} S_\epsilon) \setmin \delta_\alpha
	$$
	and let
	$$
	S^* = \bigcup_{\alpha < \kappa} S^*_\alpha \cup \{\delta_\alpha : \alpha < \kappa\}.
	$$
	It is easy to check that $S^*$ is nowhere stationary.
	For $\delta \in S^*$ we define
	$$\Lambda^*_\delta = \bigcup_{\epsilon < \delta} \Lambda_{\epsilon,\delta}
	\cup
	\begin{cases}
		\{D_\alpha\} & \text{if } \delta = \delta_\alpha \text{ for some } \alpha < \kappa\\
		\emptyset & \text{otherwise.}
	\end{cases}
	$$
	We claim that $\set_0(D) \seq \set_0^-(\squ \Lambda^*)$, witnessing $A \in \id^-(\QQ_\kappa)$.
	Let $\eta \in \set_0(D)$.
	
	\underline{Case 1:} $(\exists^\infty \alpha < \kappa)\ \eta \on \delta_\alpha \in \set_0(D_\alpha)$. Thus
	clearly $\eta \in \set_0^-(\squ \Lambda^*)$.
	
	\underline{Case 2:} $(\forall^\infty \alpha < \kappa)\ \eta \on \delta_\alpha \in \set_1(D_\alpha)$.
	So $\eta \on \delta_\alpha \in [p_{\epsilon_\alpha} \cap \tle{\delta_\alpha}]$ for some
	$\epsilon_\alpha < \delta_\alpha$
	for almost all (or just infinitely many) $\alpha < \kappa$. However
	$\eta \in \set_0(D_\alpha)$ implies that $\eta \not \in [p_{\epsilon_\alpha}]$.
	Hence there exists $\delta \in S_{\epsilon_\alpha} \setmin \delta_\alpha$ such that
	$\eta \on \delta \in \set_0^-(\Lambda_{\epsilon_\alpha, \delta})$. Recall that
	$\Lambda_{\epsilon_\alpha, \delta} \seq \Lambda^*_\delta$ and thus
	$\eta \in \set_0^-(\squ \Lambda^*)$.
\end{proof}

\begin{cor}
	\label{d6}
	Let $\kappa$ be a weakly compact cardinal. Then $\id^-(\QQ_\kappa) = \id(\QQ_\kappa)$.
\end{cor}

\begin{proof}	
	By
   \cite [Observation 4.4] {Sh:1004} 
	$\kappa$ weakly compact implies $\lnot \Pr(\kappa)$ which
	by~\ref{d1}(b) implies $\wid(\QQ_\kappa) = \id(\QQ_\kappa)$.
	So by~\ref{d5} the result follows.
\end{proof}

\begin{lem}
	\label{e3}
	Let $S \seq \kappa$ be nowhere stationary. Then we can find:
	\begin{enumerate}
		\item 
		A regressive
		function $f$ on~$S$.
		\item 
		A family $\{
		E_\alpha : \alpha \leq \kappa, \cf(\alpha) > \omega
		\}$
		where $E_\alpha \seq \alpha$ is a club disjoint from
		$S \cap \alpha$.
	\end{enumerate}
	such that:	
	\begin{enumerate}[(a)]
		\item
		$(\forall \delta \in  \kappa\setmin\omega)\ |\{\lambda \in S \setmin \delta : f(\lambda) \leq \delta\}| < \delta.$
		\item 
		$(\forall \alpha)(\forall \lambda \in E_\alpha)\ 
		\delta > \lambda \Rightarrow f(\delta) > \lambda$.
	\end{enumerate}
\end{lem}
\begin{proof}
	We prove by induction on $\beta\le \kappa $
    that we can find 
    a regressive function $f_\beta$ on $S\cap \beta$ and a family 
    $\{ E_\alpha:\alpha<\beta\}$ with the required properties. 
    For $\beta = \kappa$ the result follows.
	
	\underline{Case 1}: $\beta > \sup(S \cap \beta)$. Obvious.
	
	\underline{Case 2}: $\beta = \sup(S \cap \beta)$, $\cf(\beta )> \omega$.
	Let $E_\beta = \langle \alpha_\zeta : \zeta < \cf(\beta) \rangle$ be an increasing
	continuous cofinal sequence in~$\beta$, disjoint from~$S$.

	Let
	$$S_\zeta = S \cap [\alpha_\zeta, \alpha_{\zeta+1})$$
	and let
	$f_\zeta$ be a function on $S_\zeta$ from the induction hypothesis.
	Without loss of generality
	$\lambda \in S_\zeta \Rightarrow f_\zeta(\lambda) \geq \alpha_\zeta$. 
    [Why? Just round up, i.e., replace $f_\zeta(\lambda)$ by $\max(\alpha_\eta,f_\zeta(\lambda))$]. The new function is still regressive, because $\alpha_\zeta\notin S$.)
	So
	$$f = \bigcup_{\zeta < \cf(\beta)} f_\zeta$$
	is as required.
	
	\underline{Case 3}: $\beta = \sup(S \cap \beta)$, $\cf(\beta) = \omega$.
	This is similar to Case 2:  Fix   an increasing sequence $(\alpha_n:n\in \omega)$ cofinal in $\beta$.  Define $f(\alpha_{n+1}):=\alpha_n$, and
use the induction hypothesis to get $f\on (\alpha_n,\alpha_{n+1})$. 
This does not violate (a) because
	we require $\delta > \omega$ there.
	
By construction, the sets $E_\beta$ have the property (b).
\end{proof}

\begin{thm}
	\label{d8}
	Let $A \in \id^-(\QQ_\kappa)$ be represented by $\squ \Lambda =
	\langle \Lambda_\delta : \delta \in S \rangle$.
	Then there exists $A' \in \id^-(\QQ_\kappa)$ represented by
	$\squ \Lambda' = \langle \Lambda'_\delta : \delta \in S' \rangle$
	such that:
	\begin{enumerate}
		\item 
		$A \seq A'$
		\item 
		$S' \in \nst_\kappa^{\pr}$
		\item 
		$S \cap S_{\pr}^\kappa \seq S'$
		\item 
		$\delta \in S \cap S' \Rightarrow \Lambda_\delta \seq \Lambda'_\delta$.
	\end{enumerate}
\end{thm}

\begin{proof}
	First without loss of generality we assume $A$ is closed under rational translates  (see~\ref{r5}) and
	in particular $\Lambda_\delta$ are closed under rational translates.
	For $\delta \in S \setmin S_{\pr}^\kappa$ find $p_\delta \in \QQ_\delta$
	witnessed by $(\langle\rangle, \squ \Gamma_\delta, S_\delta)$
	such that
	$[p_\delta] \seq \Lambda_\delta$.
	By~\ref{r1} we may assume $S_\delta \seq S_{\pr}^\delta$.
	
	Now let $f$ be a regressive function on $S$ as in~\ref{e3} and let
	$$
	S' = (S \cap S_{\pr}^\kappa) \cup \bigcup_{\delta \in S \setmin S_{\pr}^\kappa}
	S_\delta \setmin (f(\delta) + 1)
	$$
	and for $\delta \in S'$ let
	$$
	\Lambda'_\delta = 
	\cup\{
	\Gamma_{\delta^*,\delta} : \delta^* > \delta > f(\delta)
	\}
	\cup
	\begin{cases}
	\Lambda_\delta & \delta \in S \cap S_{\pr}^\kappa \\
	\emptyset & \text{otherwise.}
	\end{cases}
	$$
	\underline{Why is $S'$ nowhere stationary?}
	Let $\alpha < \kappa$, $\cf(\alpha) > \omega$.
	Why is $S' \cap \alpha$ not stationary in~$\alpha$.
	\begin{itemize}
		\item 
		$\alpha > \sup(S \cap \alpha)$. 
		Use~\ref{e3}(a).
		\item 
		$\alpha = \sup(S \cap \alpha)$.
		For the part of $S'\cap  \alpha$ that comes from~$S_\delta$ with $\delta < \alpha$
		use~\ref{e3}(b) to show that the  club set $E_\alpha$ is disjoint to 
         $S_\delta \setmin (f(\delta)+1)$, for all $\delta<\alpha$.
      For the part that comes from~$S_\delta$ with $\delta > \alpha$
		use (a) as above. 
	\end{itemize}
	See~\ref{e2} for the same argument carried out in more detail.
	Similarly argue $|\Lambda'_\delta| \leq \delta$ that.
	
	Now check that $S', \squ \Lambda'$ define a set $A'\in \id^-$ 
    covering~$A$. 
\end{proof}

\subsection{Characterizing Additivity and Cofinality}

\begin{lem}[Null set normal form theorem]
	\label{e6}
	Let $\kappa = \sup(S_\inc \cap \kappa)$ and let $A \in \id(\QQ_\kappa)$.
	For $\epsilon < \kappa$ let $W_\epsilon \seq \kappa = \sup(W_\epsilon)$
	and otherwise arbitrary (e.g.\ disjoint). Then there exist
	$S$, $  \squ \Lambda = \langle \Lambda_\delta : \delta \in S \rangle$,
    $ \squ p$, $ 
	\squ {\mathcal J} = \langle \mathcal J_\epsilon: \epsilon < \kappa \rangle$
	such that
	\begin{enumerate}
		\item 
		$S \seq \kappa$ is nowhere stationary.
		\item 
		$S \seq S_{\pr}^\kappa$.
		\item
		$\squ p = \{p_\rho : \rho \in \tle \kappa\}$ where
		$p_\rho \in \QQ_\kappa$ is witnessed by
		$(\rho, S, \squ \Lambda)$.
		\item
		$\mathcal J_\epsilon \seq \{
		p_\rho : \rho \in \tle \kappa \landx \lh(\rho) \in W_\epsilon
		\}$ is predense in $\QQ_\kappa$ (or even a maximal antichain).
		\item
		$A \seq \set_0(\mathcal J)$.
	\end{enumerate}
\end{lem}

\begin{dis}
	\label{e7}
	So the idea is as follows: a general null set $A$ is represented by
	$\kappa$-many antichains each consisting of $\kappa$-many conditions that
	are all witnessed by different nowhere stationary sets $S$ and sequences
	$\squ \Lambda$. But using a diagonalization argument we find
	a representation of the null set using only a single $S$ and $\squ \Lambda$.
	
	Lemma~\ref{e6} first appears in \cite[3.16]{Sh:1004} but we repeat a sketch of the proof
	here for the
	convenience of the reader.
\end{dis}

\begin{proof}
	Let $A \in \id(\QQ_\kappa)$ be witnessed by
	$\langle \mathcal I_\epsilon : \epsilon < \kappa \rangle$ maximal antichains
	of~$\QQ_\kappa$. Let $\mathcal I_\epsilon = \{p_{\epsilon, i} : i < \kappa \}$
	and let $p_{\epsilon,i}$ be witnessed by $(\tau_{\epsilon,i}, S_{\epsilon, i},
	\squ \Lambda_{\epsilon,i})$.
	By~\ref{r1} we may assume without loss of generality $S_{\epsilon, i} \seq S_{\pr}^\kappa$.
	
	Let
	$$
	S = \{
	\delta \in \kappa : (\exists \epsilon, i < \delta)\ \delta \in S_{\epsilon,i}
	\}
	$$
	and it is easy to see that $S$ is nowhere stationary.
	For $\delta \in S$ let
	$$
	\Lambda_\delta = \cup \{
	\Lambda_{\epsilon,i,\delta} : \epsilon < \delta, i < \delta, \delta \in S_{\epsilon,i}
	\}
	$$
	and it is easy to see that $|\Lambda_\delta| \leq \delta$.
	Finally let
	$$
	\mathcal J_\epsilon = 
	\{p_\rho : (\exists i,\epsilon < \kappa)\
	i, \epsilon < \lh(\rho) \in W_\epsilon \landx \eta_{\epsilon, i} \trianglerighteq \rho
	\}.
	$$
	Now check.
\end{proof}

\begin{cor}[Baire's theorem for $\id(\QQ_\kappa)$]
	\label{e20}
	 The ideal $\id(\QQ_\kappa)$ is not trivial.
\end{cor}

\begin{proof}
	If $\kappa > \sup(S_\inc \cap \kappa)$ then $\id(\QQ_\kappa) = \id(\Cohen_\kappa)$ so the corollary follows from Baire's theorem for the
	meager ideal on $2^\kappa$.
	
	If $\kappa = \sup(S_\inc \cap \kappa)$
	let~$S$, $\squ p$, $\langle \mathcal J_\epsilon : \epsilon < \kappa \rangle$ be as in \ref{e6}.
	Let $E \seq \kappa$ be a club disjoint from~$S$.
	We construct an sequence $\langle \rho_\epsilon : \epsilon < \kappa \rangle$ of $\rho_\epsilon \in \tle \kappa$ such that:
	\begin{enumerate}
		\item
		$p_{\rho_\epsilon} \in \mathcal J_\epsilon$.
		\item 
		$\zeta < \epsilon \Rightarrow \rho_\zeta \trianglelefteq \rho_\epsilon$.
		\item
		 (As a consequence:)
		  $\zeta < \epsilon \Rightarrow 
		    p_{\rho_\epsilon} \le p_{\rho_\zeta}$, and in particular
		    $\rho_\epsilon\in p_{\rho_\zeta}$.
	\end{enumerate} 
	We work inductively:
	If $\epsilon = \zeta+1$ find $\rho_\epsilon \in \mathcal J_\epsilon$ such that:
	\begin{enumerate}[(a)]
		\item 
		$p_{\rho_\epsilon} \not \incomp p_{\rho_\zeta}$
		\item 
		$(\lg(\rho_\epsilon), \lg(\rho_\zeta)) \cap E \neq \emptyset$
	\end{enumerate}
	If $\epsilon$ is a limit then let $\rho_\epsilon' = \bigcup_{\zeta < \epsilon} \rho_\zeta$ and find $\rho_\epsilon \trianglerighteq \rho_\epsilon'$
	as above. (Letting $\delta:=\lg(\rho_\epsilon')$ we have $\delta\in E$, so no branches 
	die out in level~$\delta$, so
	$\rho_\epsilon'\in p_{\rho_\zeta}$ for all $\zeta<\epsilon$.)

	Finally let $\eta = \bigcup_{\epsilon < \kappa} \rho_\epsilon$ and
	clearly $\eta \in \set_1(\mathcal J)$, i.e. $\set_0(\mathcal J) \neq 2^\kappa$.
\end{proof}

\begin{lem}
	\label{x0}
	Let $\kappa$ be Mahlo
	(or at least $S_{\pr}^\kappa$ stationary).
	Then there exist maps
	\begin{enumerate}
		\item
		$\phi^+ : \id(\QQ_\kappa) \to \nst_\kappa^{\pr}$
		\item 
		$\phi^- : \nst_\kappa^{\pr} \to \id^-(\QQ_\kappa)$
	\end{enumerate}
	such that for all $S \in \nst_\kappa^{\pr}$, $A \in \id(\QQ_\kappa)$:
	$$
	\phi^-(S) \seq A \quad\Rightarrow\quad S \seq^* \phi^+(A). 
	$$
\end{lem}

\begin{dis}
	\label{e1}
	Lemma~\ref{x0} first appears implicitly in \cite{Sh:1004} but proving it in terms of the $\id^-(\QQ_\kappa)$ ideal and strengthened Galois-Tukey connections
	may be more transparent.
\end{dis}

\begin{proof}	
	For $\lambda \in S_\kappa^{\pr}$ let $\Lambda^*_\lambda$ witness $\lambda \in S_\kappa^{\pr}$.
	For $S\in \nst_\kappa^{\pr}$ define 
	$$
	\phi^-(S) = \{\eta \in \tle \kappa : (\exists^\infty \delta \in S)\ 
	\eta \on \delta \in \set(\Lambda^*_\delta)\}
	$$
   and for  $A \in \id(\QQ_\kappa)$
	define $
	\phi^+(A) = S
	$
	where $S$ is as in~\ref{e6}.
	
	Now let $A \in \id(\QQ_\kappa)$, $S^* \in \nst_\kappa^{\pr}$ be such that
	$S^* \not \seq^* \phi^+(A)$ and we are going to show $\phi^-(S^*) \not \seq A$.
	So let
	$(S, \squ \Lambda, \squ p, \squ {\mathcal J})$ be as in~\ref{e6}
     for $A$ (so $\phi^+(A) = S)$.	
	By our assumption $S' = S^* \setmin S$ is unbounded. Easily we can find an
	unbounded set $S'' \seq S'$ with its closure $E$ disjoint from~$S$.
	(Simply take a club $C$ disjoint from~$S$ and working inductively
	for $\epsilon \in C$ take $\lambda \in S'$ such that $\epsilon \leq \lambda$.)
	
	We are going to inductively construct a
	$\triangleleft$-increasing sequence $\langle \eta_i : i < \kappa \rangle$
	in $\eta_i \in \tle \kappa$
	and an increasing sequence $\langle \delta_i : i < \kappa \rangle$ of $\delta_i \in \kappa$
	such that
	for $i < \kappa$:
	\begin{enumerate}[$\quad$(a)]
		\item
		$|\eta_i| = \delta_i$
		\item
		$\delta_i \in E$ (thus in particular $\delta_i \not \in S$)
		\item
		$i = j + 1 \Rightarrow \delta_i \in S''$ (thus in particular $\delta_i \in S^*$)
		\item
		$[p_{\eta_i}] \seq \bigcap_{j < i} \set_1 (\mathcal J_j)$
		\item
		$i = j + 1 \Rightarrow \eta_i \in \set_0( \Lambda^*_{\delta_i})$
	\end{enumerate}	
	Now let $\eta = \bigcup_{i < \kappa} \eta_i$ and note that
	\begin{itemize}
		\item
		$\eta \in \phi^-(S^*)$ by clause (e).
		\item
		$\eta \not \in A$ by clause (d).
	\end{itemize}
	
	It remains to prove that we can indeed carry out this induction.
	The case $i = 0$ is trivial.
	For $i$ limit let $\eta_i = \bigcup_{j < i} \eta_j$.
	(remember (b)).
	
	For $i = j + 1$ consider~$p_{\eta_j}$.
	Because $\mathcal J_j$ is predense we find $\rho \in \tle \kappa$
	such that $p_\rho \in \mathcal J_j$ and $p_{\eta_j}, p_\rho$ are compatible
	with lower bound~$p_\nu$, $\nu = \rho \cup \eta_j$.
	Choose $\delta_i \in S''$ such that $\delta_i > |\nu|$.
	Now we have that $[p_\nu \cap \tle{\delta_i}] \not \seq \set_1(\Lambda^*_{\delta_i})$ so choose 
	$\eta_i \in [p_\nu \cap \tle{\delta_i}] \setmin \set_1(\Lambda^*_{\delta_i})$ and note
	that because $\delta_i \not \in S$ we have
	$\eta_i \in p_{\eta_j}$ hence $p_{\eta_i} \seq p_{\eta_j}$.
\end{proof}

\begin{thm}
	\label{e0}
	Let $\kappa$ be Mahlo
	(or at least $S_{\pr}^\kappa$ stationary). Then:
	\begin{enumerate}
		\item
		$\add(\id^-(\QQ_\kappa),\id(\QQ_\kappa)) \leq \add (\nst_\kappa^{\pr}).$
		\item 
		$\cf(\id^-(\QQ_\kappa), \id(\QQ_\kappa)) \geq \cf (\nst_\kappa^{\pr}).$
	\end{enumerate}
\end{thm}

\begin{proof}
	By~\ref{x0} and~\ref{j5}.
\end{proof}

\begin{cor}
	\label{e11}
	Let $\kappa$ be Mahlo
	(or at least $S_{\pr}^\kappa$ stationary). Then:
	\begin{enumerate}
		\item
		$\add(\id(\QQ_\kappa)) \leq \add (\nst_\kappa^{\pr}).$
		\item 
		$\add(\id^-(\QQ_\kappa)) \leq \add (\nst_\kappa^{\pr}).$
		\item 
		$\cf(\id(\QQ_\kappa)) \geq \cf (\nst_\kappa^{\pr}).$
		\item 
		$\cf(\id^-(\QQ_\kappa)) \geq \cf (\nst_\kappa^{\pr}).$
	\end{enumerate}
\end{cor}

\begin{dfn}
	\label{e4}	
	We define
	$$\QQ^*_{\kappa, S} = \{p \in \QQ_\kappa : S_p \seq S\}.$$
	Note that we have $\QQ_{\kappa,S} \seq \QQ_{\kappa,S}^*$ but in general
	equality does not hold.
\end{dfn}

\begin{thm}
	\label{e5}
	Let $\kappa$ be Mahlo (or let at least $S_{\pr}^\kappa$ be stationary).
	Then 
	$$\add(\id(\QQ_\kappa)) = \min\{\mu_1, \mu_2\}$$
	where
	\begin{itemize}
		\item
		$\mu_1 = \add(\nst_\kappa^{\pr})$.
		\item
		$\mu_2  = \min\{ \add(\id(\QQ^*_{\kappa, S}), \id(\QQ_\kappa) : S \in \nst_\kappa^{\pr}\}$.
	\end{itemize}
\end{thm}

\begin{proof}
	Let $\mu = \add(\QQ_\kappa)$. $\mu \leq \mu_1$ follows from Theorem~\ref{e0}
	(remember~\ref{j2}) and
	$\mu \leq \mu_2$ is trivial.
	So it remains to show that $\mu \geq \min\{\mu_1, \mu_2\}$.
	
	Let $A_i \in \id(\QQ_\kappa)$ for $i < i^* < \min\{\mu_1, \mu_2\}$
	and let $(S_i, \squ \Lambda_i, \squ{\mathcal J}_i, \squ p_i)$ be as in~\ref{e6}.
	By~\ref{r1}
	we may assume that $S_i \in \nst_\kappa^{\pr}$ and because
	$i^* < \mu_1$ there is $S \in \nst_\kappa^{\pr}$ such that
	$i < i^* \Rightarrow S_i \seq^* S$. Thus easily $A_i \in \id(\QQ^*_{\kappa, S})$
	and because $i^* < \mu_2$ we have $\bigcup_{i < i^*} A_i \in \id(\QQ_\kappa)$.
	
\end{proof}

\begin{thm}
	\label{e9}
	Let $\kappa$ be Mahlo (or let at least $S_{\pr}^\kappa$ be stationary).
	Then 
	$$\cf(\id(\QQ_\kappa)) = \mu_1 + \mu_2$$
	where
	\begin{itemize}
		\item
		$\mu_1 = \cf(\nst_\kappa^{\pr})$.
		\item
		$\mu_2  = \sup\{ \cf(\id(\QQ^*_{\kappa, S})), \id(\QQ_\kappa) : S \in \nst_\kappa^{\pr}\}$.
	\end{itemize}
\end{thm}

\begin{proof}
	Let $\mu = \cf(\QQ_\kappa)$. $\mu \geq \mu_1$ follows from Theorem~\ref{e0}
	(remember~\ref{j2})
	and $\mu \geq \mu_2$ is trivial. So it remains to show that $\mu \leq \mu_1 + \mu _2$.
	
	Let $\langle S_\zeta : \zeta < \mu_1 \rangle$ witness $\mu_1$ and for $\zeta < \mu$
	let $\langle A_{\zeta, \epsilon} : \epsilon < \mu_2 \rangle$ witness
	$\cf(\id(\QQ^*_{\kappa, S_\zeta})), \id(\QQ_\kappa) \leq \mu_2$.
	We claim that
	$$
	\{A_{\zeta, \epsilon} : \zeta < \mu_1, \epsilon < \mu_2\}
	$$
	is a cofinal family of~$\id(\QQ_\kappa)$.
	Thus let $A \in \id(\QQ_\kappa)$ be arbitrary and let
	$(S, \squ \Lambda, \squ {\mathcal J}, \squ p)$ be as in~\ref{e6}.
	By~\ref{r1}
	we may assume that $S \in \nst_\kappa^{\pr}$
	and find $\zeta < \mu_1, \alpha^* < \kappa$ such that
	$S \setmin \alpha^* \seq S_\zeta \setmin \alpha^*$.
	For $\delta \in S_\zeta$ define
	$$
	\Lambda'_\delta = \begin{cases}
	\Lambda_\delta & \text{ if } \delta \in S \setmin \alpha^* \\
	\emptyset & \text{ if } \delta \not \in S \text{ or } \delta < \alpha^*
	\end{cases}
	$$
	Now for each $i < \kappa$ correct $\mathcal J_i$ to~$\mathcal J'_i$ such that it uses only trunks of length greater than~$\alpha^*$.
	Thus we have found $A' \seq A$ and $A' \in \id(\QQ^*_{\kappa, S_\zeta})$.
	Hence there exists $\epsilon < \mu_2$ such that $A' \seq A_{\zeta, \epsilon}$.
\end{proof}

\begin{dfn}
	\label{e12}	
	Let $S \seq \kappa$ and we define
	$$
	\Pi_S = (
	\prod_{\delta \in S}(\id(\QQ_\delta)/
	\id^-(\QQ_\delta)), \leq^*)
	$$
	where the intended meaning of $\leq^*$ is pointwise set-inclusion
	for almost all places of the product.  Writing $[\Lambda_\delta]$
    for the $\id^-$-equivalence class of $\Lambda_\delta$, for
    $\squ \Lambda = \langle[\Lambda_\delta] : \delta \in S\rangle$,
	$\squ \Gamma = \langle[\Gamma_\delta] : \delta \in S\rangle\in \Pi_S$
	we define
   $$ \squ \Lambda \leq^* \squ \Gamma \ \Leftrightarrow \ 
	(\forall^\infty \delta \in S)\ \Lambda_\delta \setmin \Gamma_\delta \in \id^-(\QQ_\delta).
	$$
\end{dfn}

\begin{lem}
	\label{x1}
	
	Let $S \in \nst_\kappa$, $\sup(S) = \kappa$. Then there exist maps:
	\begin{enumerate}
		\item
		$\phi^+ : \id(\QQ_\kappa) \to \Pi_S$
		\item 
		$\phi^- : \Pi_S \to \id^-(\QQ_\kappa)$
	\end{enumerate}
	such that for all $\squ \Lambda \in \Pi_S$, $A \in \id(\QQ_\kappa)$:
	$$
	\phi^-(\squ \Lambda) \seq A \quad\Rightarrow\quad \squ \Lambda \leq^* \phi^+(A). 
	$$
\end{lem}

\begin{proof}
	Then for $\squ \Lambda = 
	\langle [\Lambda_\delta] : \delta \in S \rangle \in \Pi_S$
	define $\phi^-(\squ \Lambda) = \set_0^-(\langle \Lambda_\delta : \delta \in S \rangle)$.
    Given  
    $A \in \id(\QQ_\kappa)$, 
	find any $\Lambda$ as in~\ref{e6} and 
    define $\phi^+(A) = \Lambda \on S$.
	
	Now assume $A \in \id(\QQ_\kappa)$, $\Lambda^* \in \Pi_S$ such that
	$\Lambda^* \not \leq^* \phi^+(A)$ and we are going to show
	$\phi^-(\Lambda^*) \not \seq A$.
	Let 
	$\squ \Lambda^* = \langle [\Lambda^*_\delta] : \delta \in S \rangle$
	and for $A$ there are (as in~\ref{e6})
	$S_A, \squ {\mathcal J}$,
	$\squ \Lambda = \langle \Lambda_\delta : \delta \in S_A \rangle = \phi^+(A)$
	(without loss of generality ($S_A \supseteq S$) such that
	we have
	$$(\exists^\infty \delta \in S)\ \neg
	\big(\set_0(\Lambda	_{\delta}) \supseteq \set_0(\Lambda^*_\delta)\big) \mod \id^-(\QQ_\delta)).$$	
	Let $B_{\delta} = \set_1(\Lambda_{\delta}) \cap \set_0(\Lambda^*_\delta)$.
	Hence by the above we have
	$$(\exists^\infty \delta \in S)\ 
	B_{\delta} \not \in \id^-(\QQ_\delta).$$
	
	We are going to show 
\begin{itemize}
   \item [$(*)$]
    there exists
	$\eta \in (2^\kappa \setmin A) \cap \set_0^-(\squ \Lambda^*)$, witnessing
	$\set_0^-(\squ \Lambda^*) \not \seq A$.
\end{itemize}
	Without
	loss of generality assume closure under rational translates, i.e.\ 
	$\set_0(\Lambda_{\delta})^{[\beta]} = \set_0(\Lambda_{\delta})$ for $\beta < \delta \in S$,
	and clearly we may assume the same for $\squ \Lambda^*$.
	
	\underline{Claim:}
	Let $p_\rho \in \QQ_\kappa$ be the condition witnessed by $(\rho, S_A, \squ \Lambda)$. Then for all $\rho \in \tle \kappa$,
	there exists $\delta \in S \setmin (\lh(\rho) + 1)$ such that
	$$
	(p_\rho \cap 2^\delta) \cap \set_0(\Lambda^*_\delta) \neq \emptyset.
	$$
	To see this choose $\delta > \lh(\rho)$ such that $B_{\delta}
	\in \id^-(\QQ_\delta)$ and let
	$$
	C_{\delta} = \{
	\eta \in 2^\delta : (\forall^\infty \sigma \in S_A \cap \delta)\ 
	\eta \on \sigma \in \set_1(\Lambda_{\sigma}).
	\}
	$$
	The idea is that $C_{\delta}$ is a set of candidates for
	elements of $p_\rho \cap 2^\delta$.
	Towards contradiction assume that
	$$
	C_{\delta} \seq \set_0(\Lambda_{\delta}) \cup
	\set_1(\Lambda^*_\delta) = \lnot B_{\delta}
	$$
	i.e.\ every candidate either dies out at level $\delta$ by definition of
	$p_\rho$ or is not in~$\set_0(\Lambda^*_\delta)$.
	But clearly $C_{ \delta} = \set_1(\squ \Lambda \on \delta)$
	i.e.\ is a co-$\id^-(\QQ_\delta)$ set,
	contradicting $B_{\delta} \not \in 
	\id^-(\QQ_\delta)$.
	Hence there exists $\eta \in C_{\delta} \cap B_{\delta}$.
	Now use the closure under rational translates and
	choose $\beta \in (\lh(\tr(p_\rho)), \delta)$ large enough such that
	for $\nu \in 2^\beta \cap p_\rho$ we have
	$$
	\nu \on \beta ^\frown \eta \on (\beta, \delta) \in
	(p_\rho \cap 2^\delta) \cap \set_0(\Lambda^*_\delta).
	$$
	This concludes to proof of the claim.
	
	Now fix a club $E$ disjoint from~$S$ and
	work as in~\ref{x0} constructing a
	$\triangleleft$-increasing sequence $\langle \eta_{i} : i < \kappa \rangle$
	of $\eta_{i} \in \tle \kappa$
	and an increasing sequence $\langle \delta_{i} : i < \kappa \rangle$ of $\delta_{i} \in \kappa$
	such that
	for $i < \kappa$:
	\begin{enumerate}[$\quad$(a)]
		\item
		$|\eta_{i}| = \delta_{i}$.
		\item
		$i = j + 1 \Rightarrow \delta_{i} \in S$.
		\item
		$i$ limes $\Rightarrow \delta_{i} \in E$.
		\item
		$[p_{\eta_{i}}] \seq \bigcap_{j < i} \set_1 (\mathcal J_{j})$.
		\item
		$i = j + 1 \Rightarrow \eta_i \in \set_0( \Lambda^*_{\delta_i})$.
	\end{enumerate}
	
	Finally let $\eta = \bigcup_{i < \kappa} \eta_{i}$ and note that
	\begin{itemize}
		\item
		$\eta \in \set_0(\Lambda^*) = \phi^-(\squ \Lambda^*)$ by clause (e).
		\item
		$\eta \not \in A$ by clause (d).
	\end{itemize}
	So we  have shown~$(*)$. 
	
	It remains to check that we can carry out the induction.
	For $i = j+1$ we find $p_\rho \in \mathcal J_{i}$
	such that $p_\rho$ and $p_{\eta_{j}}$ are compatible.
	Now let $\nu = \rho \cup \eta_{j}$ and we find $\delta_{i} > |\nu|$
	such that
	$\delta_{i} \in B_{\delta}$ and $(\delta_{j}, \delta_{i}) \cap E
	\neq \emptyset$. Now using the claim we find
	$\eta_{i} \in p_{\nu} \cap 2^{\delta_{i}} \cap \set_0(\Lambda^*_{\delta_{i}})$. 		
\end{proof}

\begin{thm}
	\label{e13}
		Let $S \in \nst_\kappa, \sup(S) = \kappa$. Then:
	\begin{enumerate}
		\item
		$\add(\id^-(\QQ_\kappa),\id(\QQ_\kappa)) \leq \add(\Pi_S).$
		\item 
		$\cf(\id^-(\QQ_\kappa), \id(\QQ_\kappa)) \geq \cf (\Pi_S).$
	\end{enumerate}
\end{thm}

\begin{proof}
	By~\ref{x1} and~\ref{j5}.
\end{proof}

We will use the following definition 
and the revised GCH theorem from \cite{Sh:460}.

\begin{dfn}
	\label{e15}
	Let $\mu, \theta$ be cardinals such that $\theta < \mu$
	and $\theta$ regular. We define
	\begin{align*}
		\mu^{[\theta]} = \min\{
		|U| : U \seq \powset (\mu) \landx \varphi(U) 
		\}
	\end{align*}
	where $\varphi(U)$ iff:
	\begin{enumerate}
		\item
			All $u \in U$ have size $\theta$.
		\item
			Every $v \seq \mu$ of size $\theta$ is contained in the union of fewer than
			$\theta$ members of~$U$.
	\end{enumerate}
\end{dfn}

\begin{thm}[The revised GCH theorem]
	\label{e16}
	Let $\alpha$ be an uncountable strong limit cardinal, 
	i.e.\ $\beta < \alpha \Rightarrow 2^\beta < \alpha$. E.g.\ $\alpha$ = $|\mathbf V_{\omega+\omega}|=\beth_\omega$,
	the first strong limit cardinal.
	Then for every $\mu \geq \alpha$ for some $\epsilon < \alpha$ we have:
	$$
	\theta \in [\epsilon, \alpha] \landx \theta \text{ is regular } \Rightarrow \mu^{[\theta]} = \mu.
	$$
	\qed
\end{thm}

\begin{thm}
	\label{e2}
	Let $\kappa$ be Mahlo
	(or at least $S_{\pr}^\kappa$ stationary). Then:
	\begin{enumerate}[(a)]
		\item
		$\cf(\id^-(\QQ_\kappa)) = \mu_1 + \mu_2$.
		\item 
		$\cf(\id(\QQ_\kappa)) = \mu_1 + \mu_2 + \mu_3$.
	\end{enumerate}
	 where
	 \begin{itemize}
	 	\item
	 	$\mu_1 = \cf(\nst^\kappa_{\pr})$.
	 	\item
	 	$\mu_2 = \sup(\cf(\Pi_S) : S \in \nst^\kappa_{\pr})$
	 	\item
	 	$\mu_3 = \cf(\id(\QQ_\kappa)/\id^-(\QQ_\kappa))$.
	 \end{itemize}
\end{thm}
	
	\begin{proof}
		\underline{The inequality $\geq$:}
		
		\begin{enumerate}[(a)]
			\item
			Let $\mu^* = \cf(\id^-(\QQ_\kappa),\id(\QQ_\kappa))$.
			Then remembering~\ref{j2}:
			\begin{enumerate}[(1)]
				\item 
				$\mu^* \geq \mu_1$ by~\ref{e0}. 
				\item 
				$\mu^* \geq \mu_2$ by~\ref{e13}.
			\end{enumerate}
		
			\item
			Use the same theorems. Finally 		
			$\cf(\id(\QQ_\kappa)) \geq \mu_3$ is trivial.
		\end{enumerate}

		$ $
		
		\underline{The inequality $\leq$:}
		We only show (a) which using~\ref{j3} easily implies (b).
		
		\begin{enumerate}
			\item
			Let $\langle S_\zeta : \zeta < \mu_1 \rangle$ witness 
			$\mu_1 = \cf(\nst^\kappa_{\pr})$, i.e.\
			\begin{enumerate}
				\item
				$\zeta < \mu_1 \Rightarrow S_\zeta \in \nst^\kappa_{\pr}$.
				\item
				$(\forall S \in \nst^\kappa_{\pr})(\exists \zeta < \mu_1)\ S \seq^* S_\zeta$.
			\end{enumerate}
			
			\item
			For every $\zeta < \mu_1$ let $\langle \squ A_{\zeta,i} : i < \mu_2 \rangle$  witness
			$\mu_{2,S_\zeta} \leq \mu_2$, i.e.\
			\begin{enumerate}
				\item
				$\squ A_{\zeta,i} = \langle A_{\zeta,i,\delta} : \delta \in S_\zeta \rangle$.
				\item
				$A_{\zeta,i,\delta} \in \id(\QQ_\delta)$, representing the 
               equivalence class  $
				[A_{\zeta,i,\delta}] \in \id(\QQ_\delta) / \id^-(\QQ_\delta)$.
				\item
                for all
				$\squ A \in \prod_{\delta \in S_\zeta} \id(\QQ_\delta)$, there is some $i < \mu_2$ such that 
				for every $\delta$ large enough we have
				$A_\delta \seq A_{\zeta, i, \delta} \mod \id^-(\QQ_\delta)$.
				\item
				Changing the representative of $[A_{\zeta,i,\delta}]$ if necessary we may
				assume
				$$
				\{\eta \in 2^\delta : (\exists^\infty \sigma \in S_\zeta \cap \delta)\ 
				\eta \on \sigma \in A_{\zeta, i, \sigma}\} \seq A_{\zeta, i, \delta}.
				$$
				
			\end{enumerate}
			
			\item
			Let
			$$\theta = \min \{\theta : \theta = \cf(\theta) < |\mathbf V_{\omega+\omega}| \landx
			(\mu_1 + \mu_2)^{[\theta]} = \mu_1 + \mu_2\},$$
			see~\ref{e15} and~\ref{e16} for definition of notation and existence of $\theta$.
						
			For $u \in [\mu_1 \times \mu_2]^\theta$ 
			\begin{enumerate}
				\item
				$S_u = \cup \{S_\zeta : \{\zeta\} \times \mu_2 \cap u \neq \emptyset\}$.
				\item
				For $\delta \in S_u$ we inductively define
				$A_{u,\delta} = \cup\{A_{\zeta,i,\delta} : (\zeta,i) \in u\} \cup
				\{\eta \in 2^\delta : (\exists^\infty \sigma \in S_u \cap \delta)\ 
				\eta \on \sigma \in A_{u, \sigma}\}.$
				\item
				$A_u = \{\eta \in 2^\kappa : (\exists^\infty \delta \in S)\ 
				\eta \on \delta \in A_{u, \delta}\}$.
			\end{enumerate}
			
			\item
			Note that in (3) (because for any $\delta \in S_\inc$ we have
			$\delta > |\mathbf V_{\omega+\omega}| > \theta$).
			\begin{enumerate}
				\item
				$S_u \in \nst^\kappa_{\pr}$.
				\item
				$A_{u,\delta} \in \id(\QQ_\delta)$.
				\item
				$A_u \in \id^-(\QQ_\kappa)$.
			\end{enumerate}
			\item
			Remembering~\ref{e15},~\ref{e16} we find $\squ u$ such that
			\begin{enumerate}
				\item
				$\squ u = \langle u_\alpha : \alpha < \mu_1 + \mu_2\rangle$.
				\item
				$u_\alpha \in [\mu_1 \times \mu_2]^\theta$.
				\item
				If $u \in [\mu_1 \times \mu_2]^\theta$ then it is the union of fewer than
				$\theta$ members of $\{ u_\alpha : \alpha < \mu_1 + \mu_2\}$.
			\end{enumerate}
			
	\end{enumerate}
	
	We claim that $\langle A_{u_\alpha} : \alpha < \mu_1 + \mu_2\rangle$ is a cofinal family
	in~$\id^-(\QQ_\kappa)$.
	So let $A \in \id^-(\QQ_\kappa)$ be arbitrary and for
	$\epsilon < \theta$ we inductively define $A_\epsilon, \zeta_\epsilon, i_\epsilon,$ etc. such that:
	\begin{enumerate}[(a)]
		\item
		$A \seq A_0$.
		\item
		$\epsilon' < \epsilon \Rightarrow A_{\epsilon'} \seq A_{\epsilon}$.
		\item
		$A_\epsilon = \set_0^-(\squ \Lambda_\epsilon^1)  \in \id^-(\QQ_\kappa)$ where:
		\begin{enumerate}
		\item
		$\squ \Lambda_\epsilon^1 = \langle \Lambda_{\epsilon, \delta}^1 : \delta \in S_\epsilon^1 \rangle$.
		\item
		$S_\epsilon^1 \in \nst^\kappa_{\pr}$ (remember~\ref{d8})
		\item
		$\Lambda_{\epsilon, \delta}^1$ is a set of at most $\delta$-many maximal antichains of
		$\QQ_\delta$.
		\end{enumerate}
		\item
		$\zeta_\epsilon < \mu_1$ is minimal such that $S_\epsilon^1 \seq^* S_{\zeta_\epsilon}$.
		\item
		$\squ \Lambda_\epsilon^2 = \langle \Lambda_{\epsilon, \delta}^2 : \delta \in S_{\zeta_\epsilon} \rangle$ is such that $\delta \in S_\epsilon^1 \cap S_{\zeta_\epsilon} \Rightarrow
		\Lambda_{\epsilon,\delta}^1 = \Lambda_{\epsilon,\delta}^2$.
		(E.g.\ choose $\Lambda_{\epsilon,\delta}^2 = \emptyset$ for 
		$\delta \in S_{\zeta_\epsilon} \setmin S_\epsilon^1$.)
		\item
		$i_\epsilon < \mu_2$ is minimal such that for some
		$S_\epsilon^3 \seq S_{\zeta_\epsilon}$, $S_\epsilon^3 =^* S_{\zeta_\epsilon}$:
		$$
		(\forall \delta \in S_\epsilon^3)\ (\set_0(\Lambda_{\epsilon, \delta}^2) \seq
		A_{\zeta_\epsilon,i_\epsilon,\delta}) \mod \id^-(\QQ_\delta).
		$$
		\item
		$\squ \Lambda_\epsilon^4 = \langle \Lambda_{\epsilon, \delta}^4 : \delta \in S_\epsilon^4 \rangle$ is such that:
		\begin{enumerate}[(1)]
			\item
			$S_\epsilon^3 \seq S_\epsilon^4 \in \nst^\kappa_{\pr}$.
			\item
			If $\delta \in S_\epsilon^3$ then
			$A_{\zeta_\epsilon,i_\epsilon,\delta} \seq
			\set_0(\Lambda_{\epsilon,\delta}^4)$.
			\item
			If $\delta \in S_\epsilon^3$ then
			$\set_0(\Lambda_{\epsilon,\delta}^2) \seq
			\set_0 (\Lambda_{\epsilon,\delta}^4) \cup \set_0^-(\squ \Lambda_\epsilon^4 \on \delta)$.
			This point is the only non-explicit step, see below for why we can do this.
		\end{enumerate}
		\item
		If $\epsilon = \epsilon' + 1$ then
		$S_\epsilon^1 = S_{\epsilon'}^4$, $\squ \Lambda_\epsilon^1 = \squ \Lambda_{\epsilon'}^4$.
		\item
		If $\epsilon$ is a limit then
		$S_\epsilon^1 = \bigcup_{\epsilon' < \epsilon} S_{\epsilon'}^1$,
		$\Lambda_{\epsilon,\delta}^4= \bigcup_{\epsilon' < \epsilon} \Lambda_{\epsilon',\delta}^4$.
	\end{enumerate}
		
	\underline{Why is carrying out the induction enough?} 
	
	Note $\{
	(\zeta_\epsilon,i_\epsilon) : \epsilon < \theta
	\} \in [\mu_1 \times \mu_2]^\theta$ so
	we use (5)(c) to find $\alpha < \mu_1 + \mu_2$ such that
	\begin{equation}
	\label{e2eq1}
	(\exists^\infty \epsilon < \theta)\ (\zeta_\epsilon, i_\epsilon) \in u_\alpha.
	\end{equation}
	Remember $\theta < |\mathbf V_{\omega+\omega}| < \cf(\kappa)$ and find $\psi^* < \kappa$ such that
	$$
	(\forall \epsilon < \theta)\ 
	S_\epsilon^1 \setmin \psi^* \seq 
	S_\epsilon^2 \setmin \psi^* \seq 
	S_\epsilon^3 \setmin \psi^* \seq 
	S_\epsilon^4 \setmin \psi^* \seq 
	S_{\epsilon+1}^1 \setmin \psi^* 
	$$

\medskip
   We plan to show $A \subseteq A_{u_\alpha}$.  So 
	let $\eta \in A_0$ be arbitrary;  we will show $\eta\in A_{u_\alpha}$.
\medskip

 Let $W \seq S_0^1 \setmin \psi^*$, 
    $ \sup(W) = \kappa$ be such that
	$$
	(\forall \delta \in W)\ \eta \on \delta \in \set_0(\Lambda_{0,\delta}^1).
	$$
	Now we claim
	\begin{equation}
	\label{e2eq2}
	(\forall \delta \in W)(\forall^\infty \epsilon < \theta)\ 
	\eta \on \delta \in A_{\zeta_\epsilon, i_\epsilon, \delta}.
	\end{equation}
	We prove this by induction on $\delta \in S_\theta^1 \setmin \psi^*$.
	\begin{itemize}
		\item
		$\delta > \sup(\delta \cap S_\inc)$. Then
		$\id^-(\QQ_\delta)$ trivial so in (f) we always really
		(i.e.\ not just modulo $\id^-(\QQ_\delta)$) cover
		$\set_0(\Lambda_{\epsilon, \delta}^2)$.
		\item
		$\delta = \sup(\delta \cap S_\inc)$ and $\delta = \sup(\delta \cap S_\theta^1)$.
		By induction hypothesis we have
		$$
		(\forall \sigma \in S_\theta^1 \cap \delta)(\exists \epsilon_\sigma < \theta)
		(\forall \epsilon \geq \epsilon_\sigma)\ 
		\eta \on \sigma \in A_{\zeta_\epsilon, i_\epsilon, \sigma}
		$$
		$\delta$ is inaccessible so in particular regular, hence there exists
		$\epsilon'$ such that
		$$
		(\exists^\infty \sigma \in S_\theta^1 \cap \delta)\ \epsilon_\sigma = \epsilon'
		$$
		and for such $\sigma$ we have
		$$
		\epsilon \geq \epsilon' \Rightarrow
		\eta \on \sigma \in A_{\zeta_\epsilon, i_\epsilon, \sigma}
		$$
		and by (2)(d) this implies
		$\eta \on \delta \in A_{\zeta_\epsilon, i_\epsilon, \delta}$.
		\item
		$\delta = \sup(\delta \cap S_\inc)$ but $\delta > \sup(\delta \cap S_\theta^1)$.
		In this case always really $A_{\zeta_\epsilon,i_\epsilon,\delta} \supseteq
		\set_0(\Lambda_{\epsilon, \delta}^2)$ because otherwise
		$\delta$ would become a limit in $S_\epsilon^4$ by (g)(3), see below.
	\end{itemize}
	Now combine (\ref{e2eq1}) and (\ref{e2eq2})
	to see
	$$
	(\forall \delta \in W)(\exists^\infty \epsilon < \theta)\ 
	\eta \on \delta \in A_{\zeta_\epsilon, i_\epsilon, \delta}
	\landx (\zeta_\epsilon, i_\epsilon) \in u_\alpha.
	$$
	Thus $\eta \in A_{u_\alpha}$ and we are done.
	
	$ $
	
	\underline{How can we carry out the induction?}
	
	The only non-explicit part is
	how to get (g). The idea here is that in (f) we make some mistake
	because we only cover $\set_0(\Lambda_{\epsilon, \delta}^2)$
	modulo~$\id^-(\QQ_\delta)$, i.e.
	$$
	\set_0(\Lambda_{\epsilon,\delta}^2) \setmin A_{\zeta_\epsilon,i_\epsilon,\delta} =
	X_{\epsilon,\delta} \in \id^-(\QQ_\delta).
	$$
	Let $X_{\epsilon,\delta} = \set_0(\squ \Gamma_{\epsilon,\delta})$
	where $\squ \Gamma_{\epsilon,\delta} = \langle \Gamma_{\epsilon,\delta,\sigma} :
	\sigma \in S_{\epsilon,\delta} \seq \delta \rangle$.
	So in (g)(3) we want to fix this mistake by choosing some $S_\epsilon^4$ containing
	both $S_{\epsilon,\delta}$ and $S_{\zeta_\epsilon}$
	and then choosing $\squ \Lambda_\epsilon^4$ with all $\Gamma_{\epsilon, \delta, \sigma}$ added.
	The problem here of course is that we have to do this for all
	$\delta \in S_\epsilon^3$ but	$|S_\epsilon^3| = \kappa$ so fixing
	the mistake in such a naive way will in general yield a somewhere-stationary set
	and more than $\delta$-many antichains at level~$\delta$.
	Hence we work as follows:
	Choose a regressive function $f$ on $S_\epsilon^3$ as in~\ref{e3}, i.e.\ such that
	$$ (\forall \delta < \kappa)\ |\{\lambda \in S_\epsilon^3 \setmin \delta: f(\lambda) \leq \delta\}| < \delta$$
	i.e.\ $f$ is a regressive but in a very ``lazy'' way. The problem with fixing
	our mistakes earlier was that we tried to do it all at once so let us instead do it
	lazily as dictated by~$f$.
	Thus let let
	$$S_\epsilon^4 = S_\epsilon^3 \cup
	\bigcup_{\delta \in S_\epsilon^3} S_{\epsilon,\delta}\setmin(f(\delta)+1) 
	$$
	and for $\delta \in S_\epsilon^4$ let
	$$
	\Lambda^4_{\epsilon,\delta} = \Lambda^3_{\epsilon, \delta}
	\cup \{\Gamma_{\epsilon,\delta^*,\delta} : \delta^* > \delta > f(\delta^*) \}
	$$ 
	Now check that $S_\epsilon^4$ is nowhere stationary.
	
	\begin{itemize}
		\item
		$\delta < \sup(S_\epsilon^3 \cap \delta)$.
		Then $S_\epsilon^3 \cap \delta$ is disjoint from
		$S_{\epsilon,\delta'} \setmin (f(\delta') + 1)$ for every
		$\delta' \in S_\epsilon^3$ with $f(\delta') > \delta$ so by~\ref{e3}(a)
		the set $S_\epsilon^4 \cap \delta$ is the union of fewer than $\delta$-many non-stationary sets.
		\item
		$\delta = \sup(S_\epsilon^3 \cap \delta)$.
		Let
		$$
		S_{\epsilon,\delta}^{4*} = \bigcup_{\delta' \in S_\epsilon^3 \cap \delta}
		S_{\epsilon,\delta'}\setmin(f(\delta') + 1)
		$$
		$$
		S_{\epsilon,\delta}^{4**} = \bigcup_{\delta' \in S_\epsilon^3 \cap 
		(\kappa \setmin \delta)}
		S_{\epsilon,\delta'}\setmin(f(\delta') + 1) \cap \delta
		$$
		and clearly
		$$
		S_\epsilon^4 \cap \delta = (S_\epsilon^3 \cap \delta) \cup S_{\epsilon,\delta}^{4*} \cup
		S_{\epsilon,\delta}^{4**}.
		$$
		Let $E_\delta$ be as in
		\ref{e3} and it is easy to check using
		\ref{e3}(b) that $S_{\epsilon,\delta}^{4*}$
		is disjoint from~$E_\delta$, i.e.\ non-stationary.
		
		$S_{\epsilon,\delta}^{4**}$ is non-stationary by the argument from the previous point.
	\end{itemize}
	Similarly check $|\Lambda_{\epsilon,\delta}^4| \leq \delta$.
	\end{proof}	
	
\begin{thm}
	\label{e14}
	Let $\kappa$ be Mahlo
	(or at least $S_{\pr}^\kappa$ stationary).
	\begin{enumerate}[(a)]
		\item
		$\add(\id^-(\QQ_\kappa)) = \min\{\mu_1, \mu_2\}$.
		\item 
		$\add(\id(\QQ_\kappa)) = \min\{\mu_1, \mu_2, \mu_3\}$.
	\end{enumerate}
	where
	\begin{itemize}
		\item
		$\mu_1 = \add(\nst^\kappa_{\pr})$.
		\item
		$\mu_2 = \min(\add(\Pi_S) : S \in \nst^\kappa_{\pr})$
		\item
		$\mu_3 = \add(\id(\QQ_\kappa)/\id^-(\QQ_\kappa))$.
	\end{itemize}
\end{thm}

\begin{proof}		
	\underline{The inequality $\leq$:}
	Same as ``$\geq$'' in~\ref{e2}.
	
	$ $
	
	\underline{The inequality $\geq$:}
	We only show (a) which using~\ref{j3} easily implies (b).
	
	Let $\mu < \mu_1 + \mu_2$ and we are going to show $\mu < \add(\id(\QQ_\kappa))$.
	So let $\langle A_\zeta : \zeta < \mu \rangle$ be a family of
	$A_\zeta \in \id^-(\QQ_\kappa)$ and we are going to find $A \in \id^-(\QQ_\kappa)$
	such that
	$\bigcup_{\zeta < \mu} A_\zeta \seq A$.
	Let $A_\zeta$ be represented by
	$\langle A_{\zeta, \delta}^0 : \delta \in S^0_\zeta \rangle$
	and by~\ref{d8} we may assume $S_\zeta \in \nst_\kappa^{\pr}$.
	Now work inductively for $i < \omega$:
	\begin{enumerate}
		\item
		Let $S^i \in \nst^{\pr}_\kappa$
		be such that $\zeta < \mu \Rightarrow S^i_\zeta \seq^* S^i$.
		(Remember $\mu < \mu_1$.)
		\item
		Let $\squ A^i \in \Pi_{S^i}$ be such that
		$$
		(\forall \zeta < \mu)(\forall^\infty \delta \in S^i)\ 
		(A^i_{\zeta, \delta} \seq A^i_\delta) \quad \mod \id^-(\QQ_\delta).
		$$
		(Remember $\mu < \mu_2$.)
		\item
		For each $\zeta < \mu$ work as in~\ref{e2} using a regressive function to fix the error
		$$
		X_{\zeta, \delta}^i = (A^i_{\zeta, \delta} \setmin A^i_\delta) \in \id^-(\QQ_\delta).
		$$
		for $\delta \in S^i_\zeta$. I.e., we find $S^{i+1}_\zeta$,
		$\langle A^{i+1}_{\zeta, \delta} : \delta \in S^{i+1}_\zeta \rangle$ such that:
		\begin{enumerate}
			\item
			$S^i \seq S^{i+1}_\zeta \in \nst_\kappa^{\pr}$.
			\item
			$\delta \in S^{i+1}_\zeta \Rightarrow A^{i+1}_{\zeta, \delta} \in \id(\QQ_\delta)$.
			\item
			$\delta \in S^{i}_\zeta \Rightarrow A^i_{\zeta,\delta} \seq A^{i}_{\delta}
			\cup \set_0^-(\langle A^{i+1}_{\zeta,\epsilon} : \epsilon \in S^{i+1}_\zeta \cap \delta \rangle)$.
		\end{enumerate}
	\end{enumerate}
	Let $$S^\omega = \bigcup_{i< \omega}S^i.$$ For $\delta \in S^\omega$, $\zeta < \mu$ let
	\begin{itemize}
		\item
		$A^\omega_{\zeta,\delta} = \bigcup_{i < \omega}A^i_{\zeta,\delta}$.
		\item
		$A^\omega_\delta = \bigcup_{i < \omega}A^i_\delta$.
	\end{itemize}
	Finally let
	\begin{itemize}
		\item
		$A^\omega_\zeta = \set_0^-(\langle A^\omega_{\zeta,\delta} : \delta \in S^\omega \rangle).$
		\item
		$A^\omega = \set_0^-(\langle A^\omega_\delta : \delta \in S^\omega \rangle).$
	\end{itemize}
	For $\zeta < \mu$ we claim $A^\zeta \seq A^\omega$.
	Let $W = S^\omega \setmin \alpha^*$ 
	with $\alpha^* < \kappa$ large enough that in 
	all $\omega$-many steps of the construction in (1.) and (2.) the ``almost all'' quantifiers become ``for all''.
	
	We now claim that
	\begin{align}
	\label{e14eq1}
	(\forall \delta \in W) (\forall i < \omega)\ 
	\bigg(
	\eta \in A^i_{\zeta, \delta} \quad\Rightarrow\quad
	\Big(
	\eta \in A^\omega_\delta \lor
	(\exists^\infty \epsilon \in W \cap \delta)\ \eta \on \epsilon \in A^\omega_\epsilon
	\Big)\bigg)	
	\end{align}
	and clearly this suffices to show  $A_\zeta \seq A^\omega$. So towards contradiction assume there exists
	$\delta^* \in W$ such that there exists $i < \omega$, $\eta^* \in 2^{\delta^*}$ with
	\begin{align}
	\label{e14eq2}
	\eta^* \in A^i_{\zeta,\delta^*} 
	\landx
	\eta^* \not \in A^\omega_{\delta^*} \landx
	(\forall^\infty \epsilon \in W \cap \delta^*)\ \eta^* \on \epsilon \not \in A^\omega_\epsilon
	\end{align}
	and let $\delta^*$ be minimal with this property and
	without loss of generality $$i =~\min\{i :~\delta^* \in~S^i_\zeta\}.$$
	Now because
	$\eta^* \in A^i_{\zeta,\delta^*}$
	and
	$\eta^* \not \in A^\omega_{\delta^*}$ (thus in particular
	$\eta^* \not \in A^i_{\delta^*}$)
	so we have
	\begin{enumerate}[(i)]
		\item
		$\eta^* \in X^i_{\zeta,\delta^*}$.
		\item
		$\sup(W \cap \delta^*) = \delta^*$.
	\end{enumerate}
	By (3.)(c) there exists $W* \seq W \cap \delta^*$ unbounded such that
	$$
	(\forall \epsilon \in W^*)\ \eta^* \on \epsilon \in A^{i+1}_{\zeta,\epsilon}
	$$
	and because $W^* \seq \delta^*$ and we assumed $\delta^*$ to be minimal contradicting formula
	(\ref{e14eq1}) we have
	$$
	(\forall \epsilon \in W^*)\ 
	\Big(
	\eta \in A^\omega_\epsilon \lor
	(\exists^\infty \sigma \in W \cap \epsilon)\ 
	\eta^* \on \sigma \in A^\omega_\sigma
	\Big)
	$$
	contradicting the last conjunct of formula (\ref{e14eq2}) so we are done.
	
	Intuitively the proof showed: Because $\kappa$ is well ordered we cannot keep pushing our mistakes in (2.)
	down for $\omega$-many steps.
\end{proof}
	
\begin{cor}
	\label{e10}
	Let $\kappa$ be Mahlo
	(or at least $S_{\pr}^\kappa$ stationary).
	We get a strengthening of the general fact about ideals from
	\ref{j3}.
	\begin{enumerate}[(a)]
		\item 
		$
		\cf(\id(\QQ_\kappa)) = 
		\cf(\id^-(\QQ_\kappa)) +
		\cf(\id(\QQ_\kappa)/\id^-(\QQ_\kappa))
		$
		\item 
		$
		\add(\id(\QQ_\kappa)) = 
		\min\{ \add(\id^-(\QQ_\kappa)) ,
		\add(\id(\QQ_\kappa)/\id^-(\QQ_\kappa)) \}
		$
	\end{enumerate}
\end{cor}
		
\begin{proof}
	$ $
	\begin{enumerate}[(a)]
		\item
		By~\ref{e2}.
		\item
		By~\ref{e14} \qedhere
	\end{enumerate}
\end{proof}

\newpage
\section{$ \id(\QQ_\kappa)$ in the $\QQ_\kappa$-Extension}
 \label{fubini}
In this section we consider the relation between $\mathbf V$ and $\mathbf V^{\QQ_\kappa}$, and also more generally between $\mathbf V$ and any extension via
a strategically closed forcing.

In  \ref{anti} we show that (in contrast to the classical case), 
the ideal $\id(\QQ_\kappa)$ does not satisfy the Fubini theorem, and
in fact violates it in a strong sense.  This allows us to 
to show  $\cov(\QQ_\kappa)\le \non(\QQ_\kappa)$, in analogy
to the classical inequality cov(null)$\le$non(meager).  Also, 
the old reals become a measure zero set in the $\QQ_\kappa$-extension. 

In \ref{absolute}, we show that $\QQ_\kappa^{\mathbf V}$ is $\mathbf V$-completely embedded
into~$\QQ_\kappa^{\mathbf V^{\QQ_\kappa}}$.  This parallels the classical case, 
but the proof is necessarily different, as we do not have a measure. 

\subsection{Asymmetry}
\label{anti}

In this section we elaborate on the asymmetry of $\id(\QQ_\kappa)$ as promised in
\cite{Sh:1004}.
Anti-Fubini sets (defined below) are called 0-1-counterexamples to the Fubini property
in \cite{RZ:1999}

\begin{dfn}
	\label{m3}
	Let $\mathcal X$, $\mathcal Y$ be sets and
	let $\mathbf i \seq \mathfrak P(\mathcal X)$, $ \mathbf j \seq \mathfrak P(\mathcal Y)$
	be ideals.
	We call a set $\mathbf F \seq \mathcal X \times \mathcal Y$ an {\em anti-Fubini set}
	for  $(\mathbf i, \mathbf j)$ if:
	\begin{enumerate}[(a)]
		\item 
		For all $\eta \in 2^\kappa$ we have
		$2^\kappa \setmin \mathbf F_\eta \in \mathbf i$.
		\item
		For all $\nu \in 2^\kappa$ we have 
		$\mathbf F^\nu \in  \mathbf j$.
	\end{enumerate}
	where:
	\begin{enumerate}
		\item 
		$\mathbf F_\eta = \{
		\nu \in 2^\kappa  : (\nu, \eta) \in \mathbf F
		\}$.
		\item 
		$\mathbf F^\nu = \{
		\eta \in 2^\kappa  : (\nu, \eta) \in \mathbf F
		\}$.
	\end{enumerate}
	If $\mathbf i = \mathbf j$ then we simply call $\mathbf F$ an
	anti-Fubini set for $\mathbf i$.
\end{dfn}

\begin{lem}
	\label{m7}
	Let $\mathcal X$, $  \mathcal Y$ be sets and
	let $\mathbf i \seq \mathfrak P(\mathcal X)$, $  \mathbf j \seq \mathfrak P(\mathcal Y)$
	be ideals.
	Let $\mathbf F \seq \mathcal X \times \mathcal Y$ be such that:
	\begin{enumerate}[(a)]
		\item 
		There exists $\mathbf E_1 \in \mathbf j$ such that
		for all $\eta \in 2^\kappa \setmin \mathbf E_1$ we have
		$2^\kappa \setmin \mathbf F_\eta \in \mathbf i$.
		\item
		There exists $\mathbf E_0 \in \mathbf i$ such that
		for all $\nu \in 2^\kappa \setmin \mathbf E_0$ we have 
		$\mathbf F^\nu \in  \mathbf i$.
	\end{enumerate}
	Then there exists an anti-Fubini set $\mathbf F'$ for $(\mathbf i, \mathbf j)$.
\end{lem}

\begin{proof}
	Let
	$$
	\mathbf F' = \bigg(\mathbf F \cup \Big(\mathbf E_0 \times (2^\kappa \setmin \mathbf E_1)\Big) \bigg)
	\setmin \Big((2^\kappa \setmin \mathbf E_0) \times \mathbf E_1\Big)
	$$
	and check that $\mathbf F'$ is as required.
\end{proof}

\begin{lem}[Folklore]
	\label{m8}
	Let $\mathbf i, \mathbf j \seq \mathfrak P(\mathcal X)$ be ideals.
	If there exists
	an anti-Fubini set~$\mathbf F$ for $(\mathbf i, \mathbf j)$
	then $\cov(\mathbf i) \leq \non(\mathbf j)$.
\end{lem}

\begin{proof}
	Suppose $Y \seq \mathcal Y$, $ Y \not \in \mathbf j$. We claim that
	$$
	\cup \{
	2^\kappa \setmin \mathbf F_\eta : \eta \in Y
	\} = \mathcal X.
	$$
	Let $\nu \in \mathcal X$ be arbitrary.
    Now because $\mathbf F^\nu \in \mathbf j$ and $Y\notin \mathbf j$  we have
	$Y \setmin \mathbf F^\nu \neq \emptyset$,
    so choose $\eta_0 \in Y \setmin \mathbf F^\nu$.  
    We conclude 
    ${\eta_0}\notin \mathbf F^\nu \Rightarrow (\nu,{\eta_0})\notin \mathbf F
    \Rightarrow \nu\notin \mathbf F_{\eta_0}$, so $\nu\in \cup \{
      2^\kappa \setmin \mathbf F_\eta : \eta \in Y
      \} $.
    
	% Thus by definition $\nu \not \in \mathbf F_\eta$.
\end{proof}

\begin{lem}[Folklore]
	\label{m9}
	Let $\mathcal X$ be a set, let
	$\mathbf i, \mathbf j \seq \mathfrak P(\mathcal X)$ be ideals and let
	$\otimes : \mathcal X \times \mathcal X \to \mathcal X$ be a group
	operation
	satisfying for all $\mathbf k \in \{\mathbf i, \mathbf j\}$
	and for all $X \in \mathbf k$:
	\begin{itemize}
		\item 	$
		\eta \otimes X = \{\eta \otimes x : x \in X \} \in \mathbf k .
		$
		\item
		$X^{-1} = \{
		x^{-1} : x \in X
		\} \in \mathbf k
		$
	\end{itemize}
	 where $x^{-1}$ denotes the group inverse for $\otimes$.
	If there exists sets $A_0, A_1 \seq 2^\kappa$ such that:
	\begin{enumerate}[(a)]
		\item
		$A_0 \in \mathbf i$.
		\item
		$A_1 \in \mathbf j$.
		\item
		$A_0 \cap A_1 = \emptyset$.
		\item
		$A_0 \cup A_1 = 2^\kappa$.
	\end{enumerate}
	Then:
	\begin{enumerate}[(1)]
		\item
		There exists an anti-Fubini set for $(\mathbf i, \mathbf j)$.
		\item
		There exists an anti-Fubini set for $(\mathbf j, \mathbf i)$.
	\end{enumerate}
\end{lem}

\begin{proof}
	$ $
	\begin{enumerate}[(1)]
		\item
		Let
		$$
		\mathbf F = \{(\nu, \eta) : \nu \in \eta \otimes A_1
		\}.
		$$
		Clearly for any $\eta \in 2^\kappa$ we have
		$\mathbf F_\eta = \eta \otimes A_1$ hence
		$2^\kappa  \setmin \mathbf F_\eta = \eta \otimes A_0 \in \mathbf i$.
		For $\nu \in 2^\kappa$ we have
		$B^\nu = \{ \eta : \nu \in \eta \otimes A_1\} =
		\{\eta : \eta \in \nu \otimes A_1^{-1} \} =
		\nu \otimes A_1^{-1} \in \mathbf j$.
		So $\mathbf F$ is an anti-Fubini set for $(\mathbf i, \mathbf j)$.
		\item
		Same proof, interchanging $A_0$ and~$A_1$.
		\qedhere
	\end{enumerate}	
\end{proof}

\begin{thm}
	\label{m4}
	Let:
	\begin{enumerate}[(a)]
		\item 
		$\mathbf i = (\QQ, \dot \eta)$ is an ideal case, i.e.\
		\begin{enumerate}[(1)]
			\item 
			$\QQ$ is a $\kappa$-strategically closed forcing notion
			(or at least does not add bounded subsets of $\kappa$).
			\item
			$\dot \eta$ is a $\QQ$-name for a $\kappa$-real.
			\item
			The name $\dot \eta$ determines $\mathbf i$ in the following sense:
			$A \in \mathbf i$ iff there exists a (definition of) a $\kappa$-Borel
			set $\mathbf B \supseteq A$ such that $\QQ \forces$``$\dot \eta \not \in \mathbf B$''.
		\end{enumerate}
		%\item
		%$\dot \eta$ is generic, i.e.\ $\dot G_\QQ$ is determined by
		%$$\{\mathbf B : \mathbf B \text{ is a (definition of a) } \kappa \text{-Borel set from } \mathbf V
		%\text{ such that }
		%\dot \eta \in \mathbf B \}$$
		\item
		There exists an Borel $\mathbf F \seq 2^\kappa \times 2^\kappa$
		that is anti-Fubini for $\mathbf i$ both in $\mathbf V$ and $\mathbf V^{\QQ_\kappa}$.
	\end{enumerate}
	Then:
	\begin{enumerate}[(1)]
		\item 
		$\QQ \forces$``$(2^\kappa)^{\mathbf V} \in \mathbf i$''.
		\item 
		$\QQ$ is asymmetric, i.e.\ if $\eta_1$ is $\QQ$-generic
		over $\mathbf V$ and $\eta_2$ is $\QQ^{\mathbf V[\eta_1]}$-generic over $\mathbf V[\eta_1]$ then $\eta_1$ is
		not $\QQ$-generic over $\mathbf V[\eta_2]$.
		\item 
		$\cov(\mathbf i) \leq \non(\mathbf i)$.
	\end{enumerate}
\end{thm}

\begin{proof}
	$ $
	\begin{enumerate}[(1)]
		\item
		We want to show:
		$$
		\QQ \forces \mathbf V \cap \mathbf F_{\dot \eta} = \emptyset.
		$$
		So let $\nu \in 2^\kappa \cap \mathbf V$.
		Consider $\mathbf F^\nu = \{\eta : \nu \in \mathbf F_\eta\}$. Now because
		$\mathbf F^\nu \in \mathbf i$ we have $\dot \eta \not \in \mathbf F^\nu$
		thus $\nu \not \in \mathbf F_{\dot \eta}$.
		\item 
		By (1):
		$$
		\mathbf V[\eta_1,\eta_2] \models \eta_1 \in 2^\kappa \setmin \mathbf F_{\eta_2}.
		$$
		\item 
		By~\ref{m8}.
	\end{enumerate}
\end{proof}

\begin{lem}
	\label{m5}
	Assume $\kappa = \sup(S_\inc^\kappa)$. Then there exists an anti-Fubini set for $(\id(\QQ_\kappa), \id(\QQ_\kappa))$.
\end{lem}

\begin{dis}
	\label{m6}
	This is implicitly shown in \cite{Sh:1004} but we repeat it here for the convenience of the
	reader.
\end{dis}

\begin{proof}
	Let $\langle \delta_\epsilon : \epsilon < \kappa \rangle$ enumerate $S_\inc^\kappa$ and
	let $S = \{\delta_{\epsilon+1} : \epsilon < \kappa\}$.
	For $\eta \in 2^\kappa, \delta \in S$ define
	$$
	\mathbf F_{\eta, \delta} = \{
	\rho \in 2^\delta : (\forall^\infty \zeta < \delta)\ 
	\rho(\zeta) = \eta(\delta+\zeta) .
	\}
	$$
	Then clearly $\mathbf F_{\eta, \delta} \in \id(\QQ_\delta)$.  Let
	$$
	\mathbf F_\eta = \set_1^-(\langle \mathbf F_{\zeta,\delta} : \delta \in S\rangle)
	$$
	so $2^\kappa \setmin \mathbf F_\eta \in \id^-(\QQ_\kappa)$ by definition.
	Let $$\mathbf F = \{
	(\nu, \eta) \in 2^\kappa \times 2^\kappa : \nu \in \mathbf F_\eta
	\}$$
	and it remains to check $\mathbf F^\nu \in \id(\QQ_\kappa)$.
	Thus let $\nu \in 2^\kappa$ and consider $\mathbf F^\nu = \{
	\eta \in 2^\kappa : \nu \in \mathbf F_\eta
	\}$ and we want to show $\QQ_\kappa \forces$``$\nu \not \in \mathbf F_{\dot \eta}$''.
	Clearly for every $\zeta < \kappa$ the set 
	$$
	\{
	p \in \mathbf \QQ_\kappa : (\exists \delta \in S\setmin\zeta)
	(\forall \eta \in [p])\ 
	\nu \on \delta \in \mathbf F_{\eta, \delta}
	\}
	$$
	is a dense subset of $\QQ_\kappa$ so we are done.
\end{proof}

\subsection{Upwards absoluteness of $\id(\QQ_\kappa)$}
\label{absolute}
\begin{lem}
	
	\label{m0}
	Let $\mathcal J = \{q_i : i < \kappa\} \seq \QQ_\kappa$ be a maximal antichain
	and let $\PP$ be a strategically $\kappa$-closed forcing notion.
	Then
	$$
	\PP \forces \text{``}\check {\mathcal J} \text{ is a maximal antichain of } \QQ_\kappa
	\text{''}.
	$$
\end{lem}

\begin{proof}
	Towards contradiction assume
	there is some $p^* \in \PP$ such that
	$$
	p^* \forces \text{``} \dot q \in \QQ_\kappa, \text{ and }
        (\forall i<\kappa)\  
	\dot q \incomp q_i\text{''}.
	$$
	Without loss of generality even
	$$
	p^* \forces \text{``} \dot q \text{ is witnessed by }
	(\dot \eta, \dot S, \dot {\squ \Lambda})\text{''}.
	$$
	
	We choose $\langle p_j : j < \kappa \rangle$ decreasing in $\PP$
	according to a winning strategy for White in $\mathfrak C(\QQ_\kappa, p^*)$ such that
	\begin{enumerate}
		\item
		$p_0 \leq p^*$ forces a value to~$\dot \eta = \eta^*$.
		\item
		If $j$ is odd, then
		$p_{j}$ forces a value to~$\dot S \cap j = S_j$ and
		$\dot {\squ \Lambda} \on j = \Lambda^j$. 
	\end{enumerate}
	
	Let $q^* \in \QQ_\kappa$ be the condition witnessed by
	$(\eta^*, \bigcup_{j \text{ odd}} S_j, \bigcup_{j \text{ odd}} \Lambda^j)$.
	Now $q^* \in \mathbf V$ so there is $i < \kappa$ such that
	$q^* \not \incomp q_i$, so one of the following holds: 
       \begin{enumerate}
   \item $\tr(q_i) \trianglelefteq \eta^* \in q_i$
   \item $\eta^*\trianglelefteq \tr(q_i) \in q^*$.
       \end{enumerate} 

        If the first case holds, then 
        ``$\tr(q_i)\trianglelefteq\eta^*= \tr(\dot q)\in q_i$'' is 
	forced already by~$p_0$; if the second case holds, then $p_i$
	forces ``$\eta^*=\tr(\dot q) \trianglelefteq \eta_i=\tr(p_i)\in \dot q$'', so in either case we have $p_i \forces \dot q\not\perp q_i$. 

	Contradiction.
\end{proof}

\begin{cor}
	\label{m1}
	Let $\PP$ be a strategically $\kappa$-closed forcing notion. Then
        for every null set of the form 
        $\set_0^-(\langle A_\delta  : \delta \in S \rangle)$ 
       in $V$ we also have 
	$\PP \forces $``$\set_0^-(\langle A_\delta  : \delta \in S \rangle)
         \in \id(\QQ_\kappa)$'', or briefly: ``null sets remain null
	in the generic extension.''\qed
\end{cor}

\newpage
\section{ZFC-Results} \label{zfc_results}

\subsection{Cicho\'n's Diagram}
\begin{dis}
	\label{z6}
	In this subsection we establish some results about the relation between
	$\id(\QQ_\kappa)$ and
	the ideal of meager sets~$\id(\Cohen_\kappa)$.
	These theorems are either quotes of or promised elaborations on results
	first appearing in
	\cite{Sh:1004}.
\end{dis}

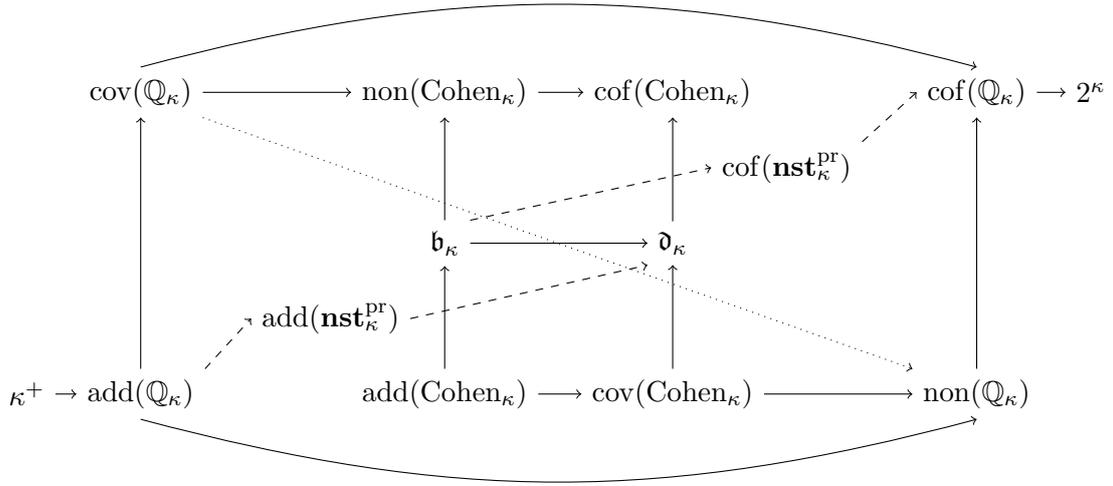
\begin{figure}[h]
	\centering
	
	\begin{tikzpicture}[
	big/.style={rectangle, draw=black!0, fill=black!100, thick, minimum size=1em},
	small/.style={rectangle, draw=black!0, fill=black!0, thick, minimum size=1em},
	]
	%Nodes
	\node[small]	at (-2.5, 0) (kappaplus)   {$\kappa^+$};
	\node[small]	at (-1, 0) (addn)   {$\add(\QQ_\kappa)$};
	\node[small]	at (-1, 4) (covn)   {$\cov(\QQ_\kappa)$};
	\node[small]	at (3, 0) (addm)   {$\add(\Cohen_\kappa)$};
	\node[small]	at (3, 4) (nonm)   {$\non(\Cohen_\kappa)$};
	\node[small]	at (6, 0) (covm)   {$\cov(\Cohen_\kappa)$};
	\node[small]	at (6, 4) (cfm)    {$\cof(\Cohen_\kappa)$};
	\node[small]	at (10, 0) (nonn)   {$\non(\QQ_\kappa)$};
	\node[small]	at (10, 4) (cfn)    {$\cof(\QQ_\kappa)$};
	\node[small]	at (11.5, 4) (cont)   {$2^\kappa$};
	
	\node[small]	at (3, 2) (b) {$\mathfrak b_\kappa$};
	\node[small]	at (6, 2) (d) {$\mathfrak d_\kappa$};

	\node[small]	at (1.5, 1) (addnst) {$\add(\nst_\kappa^{\pr})$};
	\node[small]	at (7.5, 3) (cofnst) {$\cof(\nst_\kappa^{\pr})$};
	
	%Lines
	\draw[->] (addn.north) -- (covn.south);
	\draw[->] (addm.north)  -- (b.south);
	\draw[->] (b.north)    -- (nonm.south);
	\draw[->] (covm.north) -- (d.south);
	\draw[->] (d.north)    -- (cfm.south);
	\draw[->] (nonn.north) -- (cfn.south);
	
	\draw[->] (covn.east)  -- (nonm.west);
	\draw[->] (addm.east)  -- (covm.west);
	\draw[->] (b.east)     -- (d.west);
	\draw[->] (nonm.east)  -- (cfm.west);
	\draw[->] (covm.east)  -- (nonn.west);

	\draw[->] (kappaplus.east)  -- (addn.west);
	\draw[->] (cfn.east)  -- (cont.west);
	
	\draw[->,dashed] (addn.north east)  -- (addnst.west);
	\draw[->,dashed] (addnst.east)  -- (d.south west);
	
	\draw[->,dashed] (b.north east)  -- (cofnst.west);
	\draw[->,dashed] (cofnst.north east)  -- (cfn.west);

	\draw[->] (addn.south)  to [out=345,in=195] (nonn.south);
	\draw[->] (covn.north)  to [out=15,in=165] (cfn.north);
	
	\draw[->,dotted] (covn.south east)  -- (nonn.north west);

	\end{tikzpicture}
	
	\caption{The general diagram including~$\nst_\kappa^{\pr}$, showing
		results established in this section.
		Dashed or dotted arrows have
		the same meaning as the solid ones but are intended to make the crossing arrows visually less confusing.
		To prove the implications represented by the dashed arrows 
        (those involving 
		$\add(\nst_\kappa^{\pr})$ and $\cf(\nst_\kappa^{\pr})$) 
        we need to assume that $\kappa$ is Mahlo
		(or at least $S_{\pr}^\kappa$ stationary).
	}
	\label{z7}
\end{figure}

\begin{fact}[Folklore?]
		\label{z11}
		$ $
		\begin{enumerate}[(1)]
			\item 
			$\add(\Cohen_\kappa) = \min(\mathfrak b_\kappa, \cov(\Cohen_\kappa))$.
			\item 
			$\cf(\Cohen_\kappa) = \max(\mathfrak d_\kappa, \non(\Cohen_\kappa))$.
		\end{enumerate}
\end{fact}

\begin{proof}
	See for example \cite{Sh:1004}.
\end{proof}

\begin{cor}
	\label{z12}
	Let $\kappa = \sup(S_\inc^\kappa)$. Then:
	\begin{enumerate}[(1)]
		\item
		$\cov(\Cohen_\kappa) \leq \non(\QQ_\kappa)$.
		\item
		$\cov(\QQ_\kappa) \leq \non(\Cohen_\kappa)$.
	\end{enumerate}
\end{cor}
\begin{proof}
	Let $\oplus$ be pointwise addition modulo 2.
	In \cite[3.8]{Sh:1004}  it is shown there exist sets $A_0 \in \id(\QQ_\kappa)$, $ A_1 \in \id(\Cohen_\kappa)$
	satisfying~\ref{m9}(a)--(d) for $\kappa = \sup(S_\inc^\kappa)$ so the conclusion follows by~\ref{m8}.
\end{proof}

\begin{cor}
	\label{m11}
	Let $\kappa = \sup(S_\inc^\kappa)$.	Then:
	\begin{enumerate}[(1)]
		\item
		 $\cov(\id^-(\QQ_\kappa)) \leq \non(\id(\QQ_\kappa))$
		 \item
		 and in particular $\cov(\QQ_\kappa) \leq \non(\QQ_\kappa)$.
	\end{enumerate}
\end{cor}

\begin{proof}
	By~\ref{m4} and~\ref{m5}.
\end{proof}

\begin{thm}
	\label{z8}
	$\ $
	\begin{enumerate}
	\item
	If $\mathfrak b_\kappa > \add(\Cohen_\kappa)$ then
	$\cov(\QQ_\kappa) \leq \add(\Cohen_\kappa)$.
	\item
	If $\mathfrak d_\kappa < \cf(\Cohen_\kappa)$ then
	$\cf(\Cohen_\kappa) \leq \non(\QQ_\kappa)$.
	\end{enumerate}
\end{thm}

\begin{proof}
	See \cite[5.5 and 5.7]{Sh:1004}.
\end{proof}

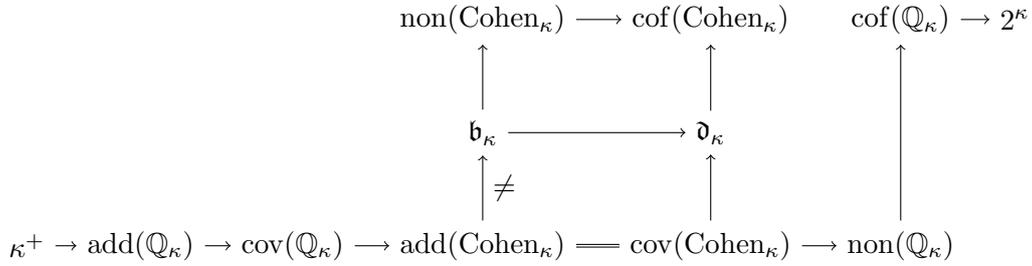
\begin{figure}[h]
	\centering
	
	\begin{tikzpicture}[
	big/.style={rectangle, draw=black!0, fill=black!100, thick, minimum size=1em},
	small/.style={rectangle, draw=black!0, fill=black!0, thick, minimum size=1em},
	]
	%Nodes
	\node[small]	at (-3, 0) (kappaplus)   {$\kappa^+$};
	\node[small]	at (-1.5, 0) (addn)   {$\add(\QQ_\kappa)$};
	\node[small]	at (0.5, 0) (covn)   {$\cov(\QQ_\kappa)$};
	\node[small]	at (3, 0) (addm)   {$\add(\Cohen_\kappa)$};
	\node[small]	at (3, 3) (nonm)   {$\non(\Cohen_\kappa)$};
	\node[small]	at (6, 0) (covm)   {$\cov(\Cohen_\kappa)$};
	\node[small]	at (6, 3) (cfm)    {$\cof(\Cohen_\kappa)$};
	\node[small]	at (8.5, 0) (nonn)   {$\non(\QQ_\kappa)$};
	\node[small]	at (8.5, 3) (cfn)    {$\cof(\QQ_\kappa)$};
	\node[small]	at (10, 3) (cont)   {$2^\kappa$};
	
	\node[small]	at (3, 1.5) (b) {$\mathfrak b_\kappa$};
	\node[small]	at (6, 1.5) (d) {$\mathfrak d_\kappa$};
	
	%Lines
	\draw[->] (addn.east) -- (covn.west);
	\draw[->] (addm.north)  -- (b.south) node[midway,right]{$\neq$};
	\draw[->] (b.north)    -- (nonm.south);
	\draw[->] (covm.north) -- (d.south);
	\draw[->] (d.north)    -- (cfm.south);
	\draw[->] (nonn.north) -- (cfn.south);
	
	\draw[->] (covn.east)  -- (addm.west);
	\draw[-,double] (addm.east)  -- (covm.west);
	\draw[->] (b.east)     -- (d.west);
	\draw[->] (nonm.east)  -- (cfm.west);
	\draw[->] (covm.east)  -- (nonn.west);

	\draw[->] (kappaplus.east)  -- (addn.west);
	\draw[->] (cfn.east)  -- (cont.west);
	\end{tikzpicture}
	
	\caption{The diagram for $\add(\Cohen_\kappa) < \mathfrak b_\kappa$}
\end{figure}

\begin{figure}[h]
	\centering
	
	\begin{tikzpicture}[
	big/.style={rectangle, draw=black!0, fill=black!100, thick, minimum size=1em},
	small/.style={rectangle, draw=black!0, fill=black!0, thick, minimum size=1em},
	]
	%Nodes
	\node[small]	at (-1, 0) (kappaplus)   {$\kappa^+$};
	\node[small]	at (0.5, 0) (addn)   {$\add(\QQ_\kappa)$};
	\node[small]	at (0.5, 3) (covn)   {$\cov(\QQ_\kappa)$};
	\node[small]	at (3, 0) (addm)   {$\add(\Cohen_\kappa)$};
	\node[small]	at (3, 3) (nonm)   {$\non(\Cohen_\kappa)$};
	\node[small]	at (6, 0) (covm)   {$\cov(\Cohen_\kappa)$};
	\node[small]	at (6, 3) (cfm)    {$\cof(\Cohen_\kappa)$};
	\node[small]	at (8.5, 3) (nonn)   {$\non(\QQ_\kappa)$};
	\node[small]	at (10.5, 3) (cfn)    {$\cof(\QQ_\kappa)$};
	\node[small]	at (12, 3) (cont)   {$2^\kappa$};
	
	\node[small]	at (3, 1.5) (b) {$\mathfrak b_\kappa$};
	\node[small]	at (6, 1.5) (d) {$\mathfrak d_\kappa$};
	
	%Lines
	\draw[->] (addn.north) -- (covn.south);
	\draw[->] (addm.north)  -- (b.south);
	\draw[->] (b.north)    -- (nonm.south);
	\draw[->] (covm.north) -- (d.south);
	\draw[->] (d.north)    -- (cfm.south) node[midway,right]{$\neq$};
	\draw[->] (nonn.east) -- (cfn.west);
	
	\draw[->] (covn.east)  -- (nonm.west);
	\draw[->] (addm.east)  -- (covm.west);
	\draw[->] (b.east)     -- (d.west);
	\draw[-,double] (nonm.east)  -- (cfm.west);
	\draw[->] (cfm.east)  -- (nonn.west);

	\draw[->] (kappaplus.east)  -- (addn.west);
	\draw[->] (cfn.east)  -- (cont.west);
	\end{tikzpicture}
	
	\caption{The diagram for $\mathfrak d_\kappa < \cof(\Cohen)_\kappa$}
\end{figure}
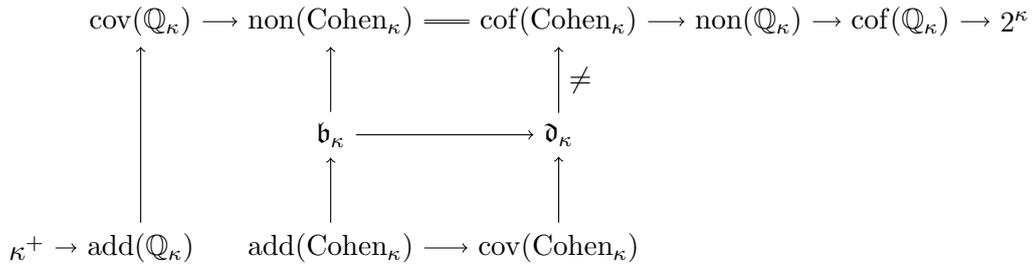

\subsection{On $\add(\QQ_\kappa) \leq \add(\Cohen_\kappa)$}

\begin{dis}
	\label{z9}
	For the classical case $(\kappa=\omega)$ the Bartoszy\'nski-Raisonnier-Stern theorem states that	$\add(\text{null}) \leq \add(\text{meager})$.
	By~\ref{z8} we know that $\add(\QQ_\kappa) \leq \add(\Cohen_\kappa)$ for large
	$\mathfrak b_\kappa$ and dually $\cf(\Cohen_\kappa) \leq \add(\QQ_\kappa)$ for
	small $\mathfrak d_\kappa$. But what about small $\mathfrak b_\kappa$, i.e.\
	$\add(\Cohen_\kappa) = \mathfrak b_\kappa$ and large $\mathfrak d_\kappa$, i.e.\
	$\mathfrak d_\kappa = \cf(\Cohen_\kappa)$?
	
	The original plan for this case was to first prove
	$\add(\QQ_\kappa) \leq \add(\nst^{\pr}_\kappa)$
	(see~\ref{e0}) and show that
	$\add(\nst^{\pr}_\kappa) \leq \mathfrak b_\kappa$. We conjecture that this second
	inequality does not hold (see~\ref{y0}).
	In \cite{Sh:1004} it was shown that we have it
	at least for sufficiently weak~$\kappa$
	(there exists a stationary non-reflecting subset of $\kappa$) and here we elaborate on this result
	as promised.
	
	Furthermore we offer a consolation prize: we show that at least
	$\add(\QQ_\kappa) \leq \mathfrak d_\kappa$ for $\kappa$ Mahlo and
	dually $\mathfrak b_\kappa \leq \cf(\QQ_\kappa)$.
\end{dis}

We begin by establishing a characterization of $\mathfrak b_\kappa$ and $\mathfrak d_\kappa$
via characteristics of the club filter of~$\kappa$. 

\begin{lem}
	\label{z0}
	$\ $
	\begin{enumerate}[(1)]
		\item 
		Let $\langle E_\alpha : \alpha < \mu < \mathfrak b_\kappa \rangle$ be a sequence of
		clubs of~$\kappa$. Then there exists a club $E$ of $\kappa$ such that
		$\alpha < \mu \Rightarrow E \seq^* E_\alpha$.
		\item
		There exists a sequence $\langle E_\alpha : \alpha < \mathfrak b_\kappa \rangle$ of
		clubs of $\kappa$ such that for no club $E$ of $\kappa$ we have
		$\alpha < \mathfrak b_\kappa \Rightarrow E \seq^* E_\alpha$.
		\item 
		$\mathfrak b_\kappa = \add(\mathbf{NS}_\kappa)$, where $\mathbf{NS}_\kappa$ is the
		ideal of non-stationary subsets of $\kappa$, ordered by eventual containment $\seq^*$.
	\end{enumerate}
\end{lem}
\begin{proof}
	$\ $
	\begin{enumerate}[(1)]
		\item
		Let $\langle E_\alpha : \alpha < \mu < \mathfrak b_\kappa \rangle$ be a sequence of
		clubs of~$\kappa$.
		We define
		$$
		f_\alpha(i) = \lceil i + 1 \rceil^{E_\alpha} =
		\min(E_\alpha \setmin (i+2)).		
		$$
		and find $f$ such $\alpha < \mu \Rightarrow f_i \leq^* f_\alpha$.
		Now let
		$$
		E = \{ \delta : f[\delta] \seq \delta \}
		$$
		and check that indeed $\alpha < \mu \Rightarrow E \seq^* E_\alpha$.
		\item
		Let $\langle f_\alpha : \alpha < \mathfrak b_\kappa \rangle$ witness $\mathfrak b_\kappa$ and
		let
		$$
		E_\alpha = \{ \delta : f_\alpha[\delta] \seq \delta \}.
		$$
		Assume there exists a club $E$ of $\kappa$ such that
		$\alpha < \mathfrak b_\kappa \Rightarrow E \seq^* E_\alpha$.
		Let
		$$
		f(i) = \lceil i + 1 \rceil^E
		$$
		and check that $\alpha < \mathfrak b_\kappa \Rightarrow f_\alpha \leq^* f$.
		Contradiction.
		\item 
		By (1.) and (2.).
		\qedhere
	\end{enumerate}
\end{proof}

\begin{lem}
	\label{z0.1}
	$\ $
	\begin{enumerate}[(1)]
		\item
		Let $\langle E_\alpha : \alpha < \mu < \mathfrak d_\kappa \rangle$ be a sequence of
		clubs of~$\kappa$. Then there exists a club $E$ of $\kappa$ such that
		for no $\alpha < \mathfrak d_\kappa$ we have $E_\alpha \seq^* E$.
		\item
		There exists a sequence $\langle E_\alpha : \alpha < \mathfrak d_\kappa \rangle$
		of clubs of $\kappa$ such that for all clubs $E$ of $\kappa$
		there exists $\alpha < \mathfrak d_\kappa$ such that $E_\alpha \seq^* E$.
		\item 
		$\mathfrak d_\kappa = \cf(\mathbf{NS}_\kappa)$.
	\end{enumerate}
\end{lem}

\begin{proof}
	Dual of~\ref{z0}.
\end{proof}

\begin{thm}	
	\label{z1}
	Let $\kappa$ be Mahlo (or just $S_{\pr}^\kappa$ stationary, see~\ref{r2}).
	Then
	$$\b_\kappa \leq \cf(\nst_\kappa^{\pr}).$$
\end{thm}

\begin{proof}
	Towards contradiction assume $\mu = \cf(\nst_\kappa^{\pr}) < \b_\kappa$ and
	let $\langle W_\alpha : \alpha < \mu \rangle$ be a sequence of nowhere stationary
	subsets of $S_\kappa^{\pr}$ witnessing $\mu = \cf(\nst_\kappa^{\pr})$.
	For $\alpha < \mu$ let $E_\alpha \seq \kappa$ be a club disjoint from~$W_\alpha$.
	Now we use~\ref{z0} to find a club $E$ such that
	$E \seq^* E_\alpha$ for every~$\alpha$.
	Now because $S_{\pr}^\kappa$ is stationary the closure of
	$E \cap S_{\pr}^\kappa$ is a club too so without loss of generality
	$W = \nacc(E) \seq S_{\pr}^\kappa$.
	Clearly $W$ is nowhere stationary so there exists $\alpha < \mu$ such that
	$W \seq^* W_\alpha$.
	
	Now because $E \seq^* E_\alpha$ and $W_\alpha \cap E_\alpha = \emptyset$
	we have $W_\alpha \cap E$ is bounded.
	On the other hand because $W$ is an unbounded subset of $E$ and $W \seq W_\alpha$
	we have $W_\alpha \cap E$ is unbounded. Contradiction.
\end{proof}

\begin{cor}
	\label{z10a}
	$\mathfrak b_\kappa \leq \cf(\QQ_\kappa)$.
\end{cor}

\begin{proof}
	Combine~\ref{z1} and~\ref{e11}
\end{proof}

\begin{thm}
	\label{z2}
	Let $\kappa$ be Mahlo (or just $S_{\pr}^\kappa$ stationary).
	Then
	$$
	\add(\nst_\kappa^{\pr}) \leq \mathfrak d_\kappa
	$$
\end{thm}
\begin{proof}
	Let $\langle E_\alpha : \alpha < \mu \rangle$
	witness $\mathfrak d_\kappa = \mu$ in the sense of~\ref{z0.1}, i.e.\ for every club
	$E$ of $\kappa$ there is $\alpha < \mu$ such that
	$E_\alpha \seq^* E$.
	If we restrict ourself to clubs $E$ such that
	$\nacc(E) \seq S_{\pr}^\kappa$ then
	we may also assume that $W_\alpha = \nacc(E_\alpha ) \seq S_{\pr}^\kappa$.
	Towards contradiction assume $\add(\nst_\kappa^{\pr}) > \mu$ and
	let $W \in \nst_\kappa^{\pr}$ such that
	$\alpha < \mu \Rightarrow W_\alpha \seq^* W$. Choose a club $E$
	disjoint from~$W$ such that $\nacc(E) \seq S_{\pr}^\kappa$.
	Now there exists $\alpha < \mu$ such that $E_\alpha \seq^* E$ hence
	$$
	\sup(E_\alpha \setmin E) < \delta \in W_\alpha \seq E_\alpha \Rightarrow \delta \in E \Rightarrow \delta \not \in W_\alpha.
	$$
	Contradiction.
	
\end{proof}

\begin{cor}
\label{z10b}
$\add(\QQ_\kappa) \leq \mathfrak d_\kappa$.
\end{cor}

\begin{proof}
Combine~\ref{z2} and~\ref{e11}
\end{proof}

\begin{thm}
	\label{z3}
	Let $\kappa$ be inaccessible and let $S \seq S_{\pr}^\kappa$ be stationary
	non-reflecting. Then
	\begin{enumerate}[(1)]
		\item
		$\add(\nst_\kappa^{\pr}) \leq \mathfrak b_\kappa$
		\item
		$\add(\nst_{\kappa,S}^{\pr}) = \mathfrak b_\kappa$.
	\end{enumerate}
\end{thm}

\begin{rem}
	\label{z4}
	Note that under these assumptions,
        by \cite[Claim 6.9]{Sh:1004}      the forcing~$\QQ_\kappa$ adds
	a $\kappa$-Cohen real.
\end{rem}

\begin{proof}
	First note that because $S$ is not reflecting we have $W \seq S$ is nowhere stationary \underline{iff} $W$ is not stationary.

	Recall~\ref{z0} and let
	$\langle E_\alpha : \alpha < \mathfrak b_\kappa \rangle$ be set of clubs of
	$\kappa$ such that for no club for every club $E$ of $\kappa$ there exist
	$\alpha < \mathfrak b_\kappa$
	such that $\lnot (E \seq^* E_\alpha)$.
	So the family $\langle S \setmin E_\alpha : \alpha < \mathfrak b_\kappa \rangle$ 
	is a set of nowhere stationary subsets of $S_{\pr}^\kappa$ with no upper bound
	in $\nst_{\kappa, S}^{pr}$ (and in particular not in $\nst_\kappa^{pr}$).
	
	Conversely let $\langle W_\alpha : \alpha < \mu \rangle$ witness
	$\add(\nst_{\kappa,S}^{\pr}) = \mu$ and let $E_\alpha$ be club disjoint from
	$W_\alpha$. Then $\langle E_\alpha : \alpha < \mu$ is an unbounded family in the sense
	of~\ref{z0}.
\end{proof}

\begin{thm}
		\label{z5}
		Let $\kappa$ be inaccessible and let $S \seq S_{\pr}^\kappa$ be stationary
		non-reflecting. Then
		\begin{enumerate}[(1)]
			\item
			$\mathfrak d_\kappa \leq \cf(\nst_\kappa^{\pr})$ 
			\item
			$\mathfrak d_\kappa = \cf(\nst_{\kappa,S}^{\pr})$.
		\end{enumerate}
\end{thm}

\begin{proof}
	Dual of~\ref{z4}.
\end{proof}

	We summarize the results of this section in the following corollary.
\begin{cor}
	\label{z13}
	If at least one of the following conditions is satisfied:
	\begin{enumerate}[(1)]
		\item 
		$\kappa > \sup(S_\inc^\kappa)$ or
		\item
		There exists a stationary non-reflecting $S \seq S_{\pr}^\kappa$ or
		\item
		$\mathfrak b_\kappa > \add(\Cohen_\kappa)$.
	\end{enumerate}	
	Then Bartoszy\'nski-Raisonnier-Stern theorem holds, i.e.\ we have
	$$
	\add(\QQ_\kappa) \leq \add(\Cohen_\kappa).
	$$
	Likewise if we let
	\begin{enumerate}[(1')]
		\addtocounter{enumi}{2}
		\item
		$\mathfrak d_\kappa < \cf(\Cohen_\kappa)$
	\end{enumerate}
	then $(1) \lor (2) \lor (3')$ implies
	$$
	\cf(\Cohen_\kappa) \leq \cf(\QQ_\kappa)
	$$
	Finally: if $(1) \lor (2) \lor \big( (3) \landx (3') \big)$, then 
	the Cicho\'n diagram for $\id(\QQ_\kappa)$ and
	$\id(\Cohen_\kappa)$ looks like the classical diagram.
	\qed
\end{cor}

\begin{conj}
	\label{y0}
	There exists a model $\mathbf V$ such that
	$$\mathbf V \models \add(\QQ_\kappa) > \add(\Cohen_\kappa)$$
	for some sufficiently strong cardinal $\kappa$.
	Note that by~\ref{z8} we necessarily have
	$$\mathbf V \models \mathfrak b_\kappa = \add(\Cohen_\kappa)$$
	so we really conjecture
	$$\CON(\add(\QQ_\kappa) > \mathfrak b_\kappa).$$
\end{conj}

\newpage
\section{Models} \label{models}

We follow the notation of \cite{BJ:1995}: Let~$\square = \kappa^+$, $\blacksquare = \kappa^{++}$.
This will allow us to graphically represent the values of the cardinal characteristics
in Figure~\ref{z7}.
E.g.\ $\square$ in the top left corner means $\cov(\QQ_\kappa) = \square$.
Note that in all diagrams of this section we have $2^\kappa = \blacksquare = \kappa^{++}$.

For visual clarity we omit the diagonal  arrow from
$\cov(\QQ_\kappa)$ to~$\non(\QQ_\kappa)$,
see~\ref{z12}.
Note again that the dashed arrows representing $\add(\QQ_\kappa) \leq \mathfrak d_\kappa$
and $\mathfrak b_\kappa \leq \cf(\QQ_\kappa)$ need $\kappa$ is Mahlo
(or at least $S_{\pr}^\kappa$ stationary).

If we would like $\QQ_\kappa$ to be $\kappa^\kappa$-bounding, i.e want $\kappa$ weakly compact, we may use Laver preparation to preserve supercompactness (so in particular weak compactness) in the forcing extension,
see \cite{Laver:1978}. Note that all forcing notions in this section, with the exception of Amoeba forcing, are ${<}\kappa$-directed closed and  Amoeba forcing
may be included in the preparation as well by \ref{almost.amoeba}.

\subsection{The Cohen Model}

\begin{dfn}
	\label{g0}
	Let
	$$
	\CC_\kappa = \tle \kappa
	$$
	and for $p, q \in \CC_\kappa$ define $q$ to be stronger than $p$ if
	$p \trianglelefteq q$.
	We call $\CC_\kappa$ the $\kappa$-Cohen forcing.
	If $G$ is a $\CC_\kappa$-generic filter then we call
	$\eta = \bigcup_{s \in G} s$ the generic $\kappa$-Cohen real (of $\mathbf V[G]$).
	Conversely we say $\nu \in 2^\kappa$ is a $\kappa$-Cohen real (over $\mathbf V$)
	if $G = \{s \in \tle \kappa : s \triangleleft \nu \}$ is
	a $\CC_\kappa$-generic filter.
\end{dfn}

\begin{fact}
	\label{g4}
	Let $\nu \in 2^\kappa$. Then $\nu$ is a $\kappa$-Cohen real over $\mathbf V$
	iff it is not contained in any meager set of $\mathbf V$. \qed
\end{fact}

\begin{lem}
	\label{g2}
	$ $
	\begin{enumerate}
		\item
		$\CC_\kappa$ is ${<}\kappa$-directed closed.
		\item
		$\CC_\kappa$ is $\kappa$-centered$_{<\kappa}$.
		\item
		$\CC_\kappa$ satisfies $(*)_\kappa$.
	\end{enumerate}
\end{lem}

\begin{proof}
	(1.) and (2.) are trivial. Then (3.) easily follows from \ref{a21}, \ref{b11}, \ref{a8}.
\end{proof}

\begin{dfn}
	\label{g3}
		Let $\mu$ be an ordinal.
		Let $\CC_{\kappa, \mu}$ be the limit of
		the ${<}\kappa$-support iteration 
		$\langle \CC_{\kappa,\alpha}, \dot \RR_\alpha : \alpha < \mu \rangle$
		where $\CC_{\kappa,\alpha} \forces$``$\dot \RR_\alpha = \CC_\kappa$''
		for every $\alpha < \mu$.
		
		It is easy to check that $\prod_{i < \mu} \CC_\kappa$
		can be canonically embedded as a dense subset into~$\CC_{\kappa,\mu}$.
\end{dfn}

\begin{lem}
	\label{g5}
	Let $\mu$ be an ordinal. Then $\CC_{\kappa,\mu}$ satisfies
	the stationary $\kappa^+$-Knaster condition and in particular
	$\CC_{\kappa,\mu}$ satisfies the $\kappa^+$-c.c.
\end{lem}

\begin{proof}
	By \ref{g2}, \ref{a6}, \ref{a1}.
\end{proof}

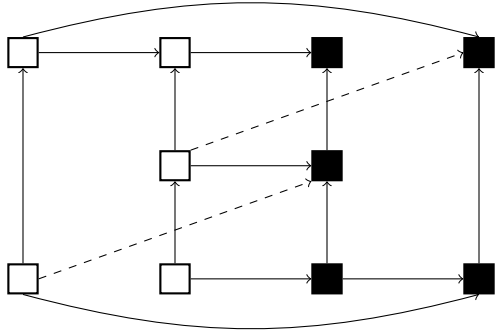
\begin{figure}[h]
	\centering
	
	\begin{tikzpicture}[
	big/.style={rectangle, draw=black!100, fill=black!100, thick, minimum size=1em},
	small/.style={rectangle, draw=black!100, fill=black!0, thick, minimum size=1em},
	]
	%Nodes
	\node[small]	at (0, 0) (addn)   { };
	\node[small]	at (0, 3) (covn)   { };
	\node[small]	at (2, 0) (addm)   { };
	\node[small]	at (2, 3) (nonm)   { };
	\node[big]		at (4, 0) (covm)   { };
	\node[big]		at (4, 3) (cfm)    { };
	\node[big]		at (6, 0) (nonn)   { };
	\node[big]		at (6, 3) (cfn)    { };
	
	\node[small]	at (2, 1.5) (b) { };
	\node[big]		at (4, 1.5) (d) { };
	
	%Lines
	\draw[->] (addn.north) -- (covn.south);
	\draw[->] (addm.north) -- (b.south);
	\draw[->] (b.north)    -- (nonm.south);
	\draw[->] (covm.north) -- (d.south);
	\draw[->] (d.north)    -- (cfm.south);
	\draw[->] (nonn.north) -- (cfn.south);
	
	\draw[->] (covn.east)  -- (nonm.west);
	\draw[->] (addm.east)  -- (covm.west);
	\draw[->] (b.east)     -- (d.west);
	\draw[->] (nonm.east)  -- (cfm.west);
	\draw[->] (covm.east)  -- (nonn.west);

	\draw[->,dashed] (addn.east)  -- (d.south west);
	\draw[->,dashed] (b.north east)  -- (cfn.west);
	
	\draw[->] (addn.south)  to [out=345,in=195] (nonn.south);
	\draw[->] (covn.north)  to [out=15,in=165] (cfn.north);
	
	\end{tikzpicture}
	
	\caption{The Cohen model}
\end{figure}

\begin{thm}
	\label{g1}
	Let $\mathbf V \models 2^\kappa = \kappa^+$. Then
	$\mathbf V^{\CC_{\kappa, \kappa^{++}}}$ satisfies:
	\begin{enumerate}
		\item
		$\non(\Cohen_\kappa) = \kappa^+$.
		\item
		$\cov(\Cohen_\kappa) = \kappa^{++}$.
		\item
		$2^\kappa = \kappa^{++}$.
	\end{enumerate}
	We call $\mathbf V^{\CC_{\kappa, \kappa^{++}}}$ the $\kappa$-Cohen model.
\end{thm}

\begin{proof}
	$ $	
	\begin{enumerate}
		\item
		This is a standard argument from the classical case but we give details.
		
		Let $\dot M = \{\dot \eta_\alpha : \alpha < \kappa^+\}$ 
		where $\dot \eta_\alpha$ is a name for the $\kappa$-Cohen real added by $\dot \RR_\alpha$. We claim $\CC_{\kappa,\kappa^{++}} \forces$``$\dot M$
		is a nonmeager set''.
		Towards contradiction assume that there are
		$\langle \dot A_i : i < \kappa \rangle$
		where $\dot A_i$ is a $\CC_{\kappa,\kappa^{++}}$-name for a
		closed, nowhere dense set and there exists $p \in \CC_{\kappa,\kappa^{++}}$
		such that $p \forces$``$\dot M \seq \bigcup_{i < \kappa} A_i$''.
		It is easy to see that any closed nowhere dense set
		$A_i \in \mathbf V^{\CC_{\kappa,\kappa^{++}}}$
		is decided by $|\tle \kappa| = \kappa$-many antichains
		$\langle \mathcal J_{i,s} : s \in \tle \kappa\rangle$ where
		$\mathcal J_{i,s}$ decides the hole of $A_i$ above~$s$, i.e.\  decides
		$\dot t_{i,s} \trianglerighteq s$ such that $[t_{i,s}] \cap A_i
		= \emptyset$.
		Remember~\ref{g5} and let
		$$
		\alpha \in \kappa^{+} \setmin \Big(\bigcup_{i < \kappa}
		\bigcup_{s \in \tle \kappa} \supp(p_{s,i}) \Big).
		$$
		Remember \ref{g3} and let $\Pi$ be the range of the dense
                embedding of
		$\prod_{i < \kappa^{++}} \CC_\kappa$ into~$\CC_{\kappa,\kappa^{++}}$.
		Without loss of generality $\mathcal J_{i,s} \seq \Pi$ for all
		$i < \kappa$ and  all $s \in \tle \kappa$ and also $p \in \Pi$.
		Find $p' \leq p$
		such that $p' \in \Pi$ and let $s = p(\alpha)$.
		Now for arbitrary $i < \kappa$ we can find $r \in \mathcal J_{i,s}$,
		$r \not \incomp p'$ and let $p'' = r \landx p'$.
		Now because $p', r \in \Pi$ we have $p''(\alpha) = s$ and $p''$ decides
		$t_s \trianglerighteq s$ to be missing from~$A_i$.
		Thus define $p''' \leq p''$ such that $p'''(\alpha) = t_s$
		and $p'''(\beta) = p''(\beta)$ for $\beta \in \kappa^{++} \setmin \{\alpha\}$.
		Clearly $\dot \eta_\alpha \trianglerighteq t_s$ thus
		$p''' \forces$``$\dot \eta_\alpha \not \in \dot A_i$''. Clearly $p''' \leq p$
		hence contradicting $p \forces$``$\dot M \seq \bigcup_{i < \kappa} A_i$''.
		\item
		Same argument as in~\ref{n4}. 
		\item
		Should be clear using nice names.
		\qedhere		
	\end{enumerate}
\end{proof}

\subsection{The Hechler Model}

\begin{dfn}
	\label{n1}
	Let
	$$
	\HH_\kappa = \kappa^{<\kappa} \times [\kappa^\kappa]^{<\kappa}
	$$
	and for $p_1 = (\rho_1,X_1), p_2 = (\rho_2,X_2) \in \HH_\kappa$ define
	$p_2$ to be stronger than $p_1$ if:
	\begin{enumerate}
		\item 
		$\rho_2 \trianglerighteq \rho_1$.
		\item
		$X_2 \supseteq X_1$.
		\item
		For all $i \in \dom(\rho_2) \setmin \dom(\rho_1)$ and for
		all $f \in X_1$ we have
		$\rho_2(i) > f(i)$.
	\end{enumerate}
	We call $\HH_\kappa$ the $\kappa$-Hechler forcing.	
	If $G$ is a $\HH_\kappa$-generic filter then we call
	$\eta = \bigcup_{(\rho, X) \in G} \rho$ the generic $\kappa$-Hechler real.
	
	The intended meaning of a condition $(\rho, X)$ is the promise that
	the $\kappa$-Hechler real will start with $\rho$ and from now on
	(i.e.\ past the length of $\rho$) dominate all
	functions in~$X$.
	
\end{dfn}

\begin{fact}
	\label{n2}
	Let $\eta$ a $\kappa$-Hechler real over $\mathbf V$. Then
	for every $\nu \in \kappa^\kappa \cap {\mathbf V}$ we have
	$\nu \leq^* \eta$.
	\qed
\end{fact}

\begin{fact}
	\label{n3}
	Let $\eta$ a $\kappa$-Hechler real over $\mathbf V$.
	Let $\nu \in 2^\kappa$ be such that for all $i < \kappa$
	$$
	\nu(i) \equiv \eta(i) \mod 2.
	$$
	Then $\nu$ is a $\kappa$-Cohen real over $\mathbf V$.
	\qed
\end{fact}

\begin{lem}
	\label{n5}
	$ $
		\begin{enumerate}
			\item
			$\HH_\kappa$ is ${<}\kappa$-directed closed.
			\item
			$\HH_\kappa$ is $\kappa$-centered$_{<\kappa}$.
			\item
			$\HH_\kappa$ satisfies $(*)_\kappa$.\qed
		\end{enumerate}
\end{lem}

\begin{proof}
	$ $
	\begin{enumerate}
		\item
		Let $D \seq \HH_\kappa$, $|D| < \kappa$, $p, q \in D \Rightarrow p \not \incomp q$.
		If $p = (\rho_1, X_1), q = (\rho_2, X_2) \in D$ then because $p, q$ are compatible
		we have $\rho_1 \trianglelefteq \rho_2 \lor \rho_2 \trianglelefteq \rho_1$.
		Hence $(\rho^*, X^*)$ is a lower bound for $D$ where
		$\rho^* = \bigcup_{(\rho, X) \in D} \rho$, 
		$X^* = \bigcup_{(\rho, X) \in D} X$.
		\item
		$\HH_\kappa = \bigcup_{\rho \in \kle \kappa} (\{\rho\} \times [\kappa^\kappa]^{<\kappa}$).
		\item
		By (1.), (2.), \ref{a21}, \ref{b11}, \ref{a8}.\qedhere
	\end{enumerate}
\end{proof}

\begin{dfn}
	\label{n7}
	Let $\mu$ be an ordinal.
	Let $\HH_{\kappa, \mu}$ be the limit of
	the ${<}\kappa$-support iteration 
	$\langle \HH_{\kappa,\alpha}, \dot \RR_\alpha : \alpha < \mu \rangle$
	where $\HH_{\kappa,\alpha} \forces$``$\dot \RR_\alpha = \HH_\kappa$''
	for every $\alpha < \mu$.
\end{dfn}

\begin{lem}
	\label{n8}.
	Let $\mu$ be an ordinal. Then:
	\begin{enumerate}
		\item
		$\HH_{\kappa,\mu}$ satisfies the stationary $\kappa^+$-Knaster condition and in particular
		$\HH_{\kappa,\mu}$ satisfies the $\kappa^+$-c.c.
		\item
		If $\mu < (2^\kappa)^+$ then $\HH_{\kappa,\mu}$ is $\kappa$-centered$_{<\kappa}$.
	\end{enumerate}
\end{lem}

\begin{proof}
	$ $
	\begin{enumerate}
		\item
		By \ref{n5}, \ref{a6}, \ref{a1}.
		\item
		Remember~\ref{n5}(2.). Easily check that $\HH_{\kappa,\mu}$ is finely ${<}\kappa$-closed
		so use~\ref{b5}.\qedhere
	\end{enumerate}
\end{proof}

\begin{figure}[h]
\centering

	\begin{tikzpicture}[
	big/.style={rectangle, draw=black!100, fill=black!100, thick, minimum size=1em},
	small/.style={rectangle, draw=black!100, fill=black!0, thick, minimum size=1em},
	]
	%Nodes
	\node[small]	at (0, 0) (addn)   { };
	\node[small]	at (0, 3) (covn)   { };
	\node[big]		at (2, 0) (addm)   { };
	\node[big]		at (2, 3) (nonm)   { };
	\node[big]		at (4, 0) (covm)   { };
	\node[big]		at (4, 3) (cfm)    { };
	\node[big]		at (6, 0) (nonn)   { };
	\node[big]		at (6, 3) (cfn)    { };
	
	\node[big]		at (2, 1.5) (b) { };
	\node[big]		at (4, 1.5) (d) { };
	
	%Lines
	\draw[->] (addn.north) -- (covn.south);
	\draw[->] (addm.north) -- (b.south);
	\draw[->] (b.north)    -- (nonm.south);
	\draw[->] (covm.north) -- (d.south);
	\draw[->] (d.north)    -- (cfm.south);
	\draw[->] (nonn.north) -- (cfn.south);
	
	\draw[->] (covn.east)  -- (nonm.west);
	\draw[->] (addm.east)  -- (covm.west);
	\draw[->] (b.east)     -- (d.west);
	\draw[->] (nonm.east)  -- (cfm.west);
	\draw[->] (covm.east)  -- (nonn.west);

	\draw[->,dashed] (addn.east)  -- (d.south west);
	\draw[->,dashed] (b.north east)  -- (cfn.west);
	
	\draw[->] (addn.south)  to [out=345,in=195] (nonn.south);
	\draw[->] (covn.north)  to [out=15,in=165] (cfn.north);
\end{tikzpicture}

\caption{The Hechler model}
\end{figure}
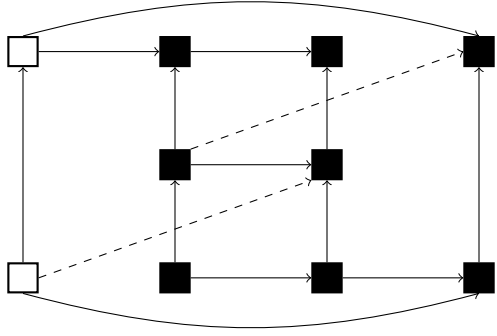

\begin{thm}
	\label{n4}
	Let $\mathbf V \models 2^\kappa = \kappa^+$. Then
	$\mathbf V^{\HH_{\kappa,\kappa^{++}}}$ satisfies:
	\begin{enumerate}
		\item
		$\cov(\QQ_\kappa) = \kappa^+$.
		\item
		$\mathfrak b_\kappa = \kappa^{++}$.
		\item
		$\cov(\Cohen_\kappa) = \kappa^{++}$.
		\item 
		$\add(\Cohen_\kappa) = \kappa^{++}$.
		\item
		$2^\kappa = \kappa^{++}$.
	\end{enumerate}
	We call $\mathbf V^{\HH_{\kappa,{\kappa^{++}}}}$ the $\kappa$-Hechler model.
\end{thm}

\begin{proof}
	We use the iteration theorems from section~\ref{tools} so the following proofs
	become standard arguments from the classical case.
		
	\begin{enumerate}
		\item 
		We claim that $\HH_{\kappa,\kappa^{++}}$ does not add $\QQ_\kappa$-generic reals.
		Remember~\ref{n8}(1.) so if we have a nice $\HH_{\kappa,\kappa^{++}}$-name
		$\dot \eta$ for a $\kappa$-real the antichains deciding $\dot \eta$
		are already antichains of $\HH_{\kappa,\alpha}$ for some $\alpha < \kappa$.
		Note that if we show that $\HH_{\kappa,\alpha}$ does not add
		$\QQ_\kappa$-generic reals for any $\alpha < \kappa^{++}$ we are done:
				
		If $\eta \in \mathbf V^{\HH_{\kappa,\alpha}}$ is not $\QQ_\kappa$-generic over $\mathbf V$
		then there is a Borel code $c \in \mathbf V$ of an $\id(\QQ_\kappa)$-set $\mathscr B_c$
		such that $\eta \in \mathscr B_c$. The same is still true in
		$\mathbf V^{\HH_{\kappa, \kappa^{++}}}$, see \ref{r12.6}.

By~\ref{n8} (2.) $\HH_{\kappa,\alpha}$ is a $\kappa$-centered$_{<\kappa}$
forcing notion for each $\alpha < \kappa^{++}$ and thus by~\ref{b7} does not add
a $\QQ_\kappa$-generic real.		
In $\mathbf V$ there exists a covering of $\id(\QQ_\kappa)$ of size $\kappa^+$
and because $\HH_{\kappa,\kappa^{++}}$ does not add $\QQ_\kappa$-generic reals this covering remains
a covering in $\mathbf V^{\HH_{\kappa,\kappa^{++}}}$.

\item
Assume there exists an unbounded family of size $\kappa^+$ in $\mathbf V^{\HH_{\kappa,\kappa^{++}}}$.
Argue as above to see that this family already appears in some
$\mathbf V^{\HH_{\kappa,\alpha}}$.
But by~\ref{n2} $\RR_\alpha$ adds a bound. Contradiction.

\item
Assume there exists an covering
of $\id(\Cohen_\kappa)$ of size $\kappa^+$ in $\mathbf V^{\HH_{\kappa,\kappa^{++}}}$.
Again this family already appears in some $\mathbf V^{\HH_{\kappa,\alpha}}$.
But by~\ref{n3} $\dot \RR_\alpha$ adds a $\kappa$-Cohen real hence the 
covering is destroyed. Contradiction.

\item
Remember~\ref{z11} so this
follows from (2.) and (3.).

\item
Should be clear.
\qedhere
\end{enumerate}
\end{proof}

\subsection{The Short Hechler Model}

\begin{figure}[h]
	\centering
	
	\begin{tikzpicture}[
	big/.style={rectangle, draw=black!100, fill=black!100, thick, minimum size=1em},
	small/.style={rectangle, draw=black!100, fill=black!0, thick, minimum size=1em},
	]
	%Nodes
	\node[small]	at (0, 0) (addn)   { };
	\node[small]	at (0, 3) (covn)   { };
	\node[small]	at (2, 0) (addm)   { };
	\node[small]	at (2, 3) (nonm)   { };
	\node[small]	at (4, 0) (covm)   { };
	\node[small]	at (4, 3) (cfm)    { };
	\node[big]		at (6, 0) (nonn)   { };
	\node[big]		at (6, 3) (cfn)    { };
	
	\node[small]	at (2, 1.5) (b) { };
	\node[small]		at (4, 1.5) (d) { };
	
	%Lines
	\draw[->] (addn.north) -- (covn.south);
	\draw[->] (addm.north) -- (b.south);
	\draw[->] (b.north)    -- (nonm.south);
	\draw[->] (covm.north) -- (d.south);
	\draw[->] (d.north)    -- (cfm.south);
	\draw[->] (nonn.north) -- (cfn.south);
	
	\draw[->] (covn.east)  -- (nonm.west);
	\draw[->] (addm.east)  -- (covm.west);
	\draw[->] (b.east)     -- (d.west);
	\draw[->] (nonm.east)  -- (cfm.west);
	\draw[->] (covm.east)  -- (nonn.west);

	\draw[->,dashed] (addn.east)  -- (d.south west);
	\draw[->,dashed] (b.north east)  -- (cfn.west);
	
	\draw[->] (addn.south)  to [out=345,in=195] (nonn.south);
	\draw[->] (covn.north)  to [out=15,in=165] (cfn.north);
	\end{tikzpicture}
	
	\caption{The short Hechler model}
\end{figure}
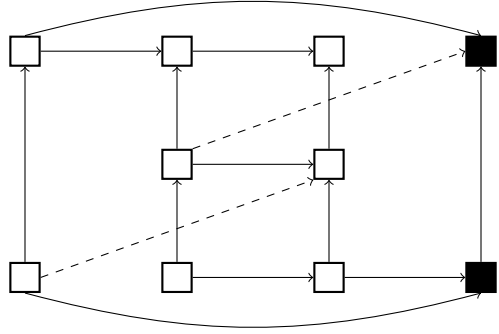

\begin{thm}
	\label{n6}
	Let $\mathbf V \models \kappa$ is weakly compact.
	Let $\mathbf V \models \non(\QQ_\kappa) = \kappa^{++}$ (e.g.\ $\mathbf V = \mathbf V_0^{\HH_{\kappa,\kappa^{++}}}$). 

        Let
        $ \HH_{\kappa,\kappa^+}$ be the $\mathord<\kappa$-support iteration
        of length $\kappa^+ $ of Hechler reals (see \ref{n7}).
	Then $\mathbf V^{\HH_{\kappa,\kappa^+}}$ satisfies:
	\begin{enumerate}
		\item 
		$\non(\QQ_\kappa) = \kappa^{++}.$
		\item 
		$\mathfrak d_\kappa = \kappa^+$.
		\item 
		$\non(\Cohen_\kappa) = \kappa^+$.
		\item 
		$\cf(\Cohen_\kappa) = \kappa^+$.
		\item
		$2^\kappa = \kappa^{++}$.
	\end{enumerate}
\end{thm}

\begin{proof}
	$ $
	\begin{enumerate}
		\item 
		Follows by~\ref{b5} and~\ref{b9}.
		\item 
		Remember~\ref{n2} so $\{\eta_\epsilon : \epsilon < \kappa^+
		\}$ is a dominating family
		where $\eta_\epsilon$ is the $\kappa$-Hechler real added by $\RR_\epsilon$.
		\item 
		We claim $\{\nu_\epsilon : \epsilon < \kappa^+\} \not \in \id(\Cohen_\kappa)$
		were $\nu_\epsilon \in 2^\kappa$ is
		the canonical $\kappa$-Cohen real added by $\RR_\epsilon$ (see~\ref{n3}).
		Argue as in~\ref{g1} but instead of using the product we find $\alpha$ greater
		than the support of all antichains.	
		\item
		Remember~\ref{z11} so this
		follows from (2.) and (3.).
		\item 
		Should be clear.\qedhere
	\end{enumerate}
\end{proof}

\subsection{Amoeba forcing, part 1}

\begin{dfn}
	\label{c5}
	Let $\QQ^{\am 1}_\kappa$ be the forcing consisting of tuples
	$(\epsilon, S, E)$ where:
	\begin{enumerate}
		\item 
		$\epsilon \in S_\inc^\kappa$.
		\item
		$S \seq S_\inc^\kappa$ is nowhere stationary.
		\item
		$E \seq \kappa$ is a club disjoint from~$S$.
	\end{enumerate}
    For $p\in \QA$ we will write $\epsilon^p, S^p, E^p$ for the respective
    components of~$p$.

	For $p = (\epsilon_p, S_p, E_p), q = (\epsilon_q, S_q, E_q)$
	we define  $q\le p$ ($q$ stronger than~$p$) 
	iff either~$q=p$, or:
	\begin{enumerate}
		\item 
		$\epsilon_p <  \epsilon_q$, and moreover the set 
          $E_q$ meets the interval $ (\epsilon_p,\epsilon_q)$.
		\item
		$S_p \cap \epsilon_p = S_q \cap \epsilon_p$
		\item
		$S_p \setmin \epsilon_p \seq S_q \setmin \epsilon_p$.
		\item 
		$E_p \cap \epsilon_p = E_q \cap \epsilon_p$.
		\item 
		$E_p \supseteq E_q$.
	\end{enumerate}
	The intended meaning of a condition $(\epsilon, S, E)$ is the promise to cover $S$
	from now on above $\epsilon$ but not tamper with it below $\epsilon$
	(to preserve the fact that $S \cap \epsilon$ is nowhere stationary in $\epsilon$).
	The purpose of $E$ is to ensure that  the generic set will not be stationary
	in $\kappa$.
\end{dfn}

\begin{lem}
	\label{c6}
	Let $G$ be a $\QQ_\kappa^{\am 1}$-generic filter and let
	$$
	S^* = \cup \{ S :  (\exists p\in G) \ S=S^p  \},
	$$
	$$
	E^* = \cap \{ E : (\exists p\in G) \  E=E^p \}.
	$$
	Then:
	\begin{enumerate}
		\item 
		$E^*$ is a club of $\kappa$ disjoint from~$S^*$.
		\item 
		$S^*$ is a nowhere stationary subset of $\kappa$.
		\item
		For any nowhere stationary set $S \seq \kappa$, $S \in \mathbf V$ we have 
		$
		\mathbf V^{\QQ_\kappa^{\am 1}} \models S \seq^* S^*$ (i.e., 
        the set $S\setmin  S^* $ is bounded. 
	\end{enumerate}
	We call $S^*$ the generic nowhere stationary set.
\end{lem}

\begin{proof}
	$ $
	\begin{enumerate}
		\item 
		Assume that $(\epsilon, S, E) \forces$``$E^* \seq \alpha < \kappa$''.
		Find $\beta \in E$, $\gamma\in  S_\inc^\kappa$ with 
          $\alpha<\beta<\gamma$.  Then $(\gamma,S,E)\le (\alpha,S,E)$
          and $(\gamma,S,E)\forces \beta\in E^*$, contradicting what
          $(\epsilon,S,S)$ forced. 
		So $E^*$ is unbounded.
		
		As an intersection of closed sets, $E^*$ must be  closed.
		$E^*$ is disjoint from~$S^*$ by definition.
		\item 
		To see $S^* \cap \alpha$ is non-stationary for $\alpha \in S_\inc^\kappa$ 
		argue as in (1.).
		To see $S^*$ is non-stationary in $\kappa$,
        remember that $E^*$ is a club disjoint from~$S^*$
		by (1.).
		\item 
		Let $p = (\epsilon, S, E) \in \QQ_\kappa^{\am 1}$ and let $S' \in \mathbf V$
		be nowhere stationary and
		let $E'$ be a club disjoint from~$S'$.
		Then $(\epsilon, S \cup (S' \setmin \epsilon), E \cap (E' \cup \epsilon)) \leq p$ forces $S \subseteq S^* \cup \epsilon$, hence also $S\subseteq^* S^*$.  As  $p$ was arbitrary we are done.
		\qedhere
	\end{enumerate}
\end{proof}

\begin{lem}
	\label{c16}
	$ $
	\begin{enumerate}
		\item
		$\QQ_\kappa^{\am 1}$ is $\mathord<\kappa$-closed.
		\item
		$\QQ_\kappa^{\am 1}$ is $\kappa$-linked.
		\item 
		$\QQ_\kappa^{\am 1}$ satisfies $(*)_\kappa$.
	\end{enumerate}
\end{lem}

\begin{proof}
	$ $
	\begin{enumerate}
		\item  Let $\langle p_i:i<\delta\rangle$ be
        a strictly decreasing sequence, $\delta<\kappa$ a limit ordinal,
        and let $p_i=(\epsilon_i, S_i, E_i)$. 
        Hence the sequence $\langle \epsilon_i:i<\delta\rangle$ is 
        strictly increasing, so in particular $\epsilon_i\ge i$:

        We define a condition $p^* = (\epsilon^*, S^*, E^*)$ as follows: 
			\begin{enumerate}
				\item
				$\epsilon^* = \sup_{j<\delta} \epsilon_j$.  (So $\epsilon^*\ge\delta$)
				\item
				 $S^* = \bigcup_{j<\delta}  S_j$.
				\item
				$E^* = \bigcap_{j<\delta}  E_j$. 
			\end{enumerate}
        Clearly $E^*$ is club in $\kappa$ and disjoint to~$S^*$, so 
        $S^*$ is nonstationary.

        For $\delta'<\delta$ the sequence
        $\langle S_i\cap \delta':i<\delta\rangle$ is eventually constant
        with value $S_{\delta'}\cap \delta'$, so $S^*\cap \delta'$
        is nonstationary in $\delta'$.

        For $\delta'>\delta$ the set $S^*\cap \delta'$ is the union of 
        a small number of nonstationary sets, hence is nonstationary.

        We have to check that $S^*\cap \delta$ is nonstationary
        in $\delta$ (if $\delta$ is inaccessible).  
        \begin{itemize}
           \item [Case 1]  $\epsilon^*=\delta$.  Then $E^*\cap (\epsilon_i,\epsilon_{i+1}) = E_{i+1}\cap (\epsilon_i,\epsilon_{i+1})$ is nonempty for all $i<\delta$, so $E$ is unbounded (hence club) in $\epsilon^*$.  Hence $S $ is nonstationary in $\epsilon^*$.
           \item [Case 2] $\epsilon^*>\delta$.  Then we can find $i<\delta$
           with $\epsilon_i>\delta$, and we see that $S^*\cap \epsilon_i=S_i\cap \epsilon_i$, so also $S^*\cap \delta=S_i\cap \delta$ is nonstationary.
        \end{itemize}
   Finally we show that $p^*\le p$: The main point is that 
   $(\forall j\ge i) \ S_j\cap \delta_i = S_i\cap \delta_i$, 
   so also $S^*\cap \delta_i = S_i\cap \delta_i$.

		\item 
		Consider $f : \QQ_\kappa^{\am 1} \to \kappa \times \tle \kappa \times \tle \kappa$
		where $f(\epsilon, S, E) = (\epsilon, S \cap \epsilon, E \cap \epsilon)$.
		Now check that for $p, q \in \QQ_\kappa^{\am 1}$ we have
		$f(p) = f(q) \Rightarrow p \not \incomp q$.
		\item 
		By (1.), (2.) and~\ref{a8}.
	\end{enumerate}
\end{proof}
We want to iterate Amoeba forcing (together with the 
forcing in the next subsection, and possibly other forcings)
and not lose the weak compactness 
of~$\kappa$.   So we will start in a model where $\kappa$ is supercompact, 
and this supercompactness is not destroyed by $\mathord<\kappa$-directed
closed forcing, and also not by our Amoeba forcings. 

As Amoeba forcing is not $\mathord<\kappa$-directed closed, we cannot
use Laver's theorem directly.  However, it is well known that a slightly
weaker property is also sufficient.

The following definition is copied from~\cite{Ko:2006}.
\begin{dfn}\label{def.almost}
  If $P$  is a partial ordering then we always let 
$\theta=\theta_P$ be the least regular cardinal such that
$ P \in H_\theta$.  Say that a set $X\in  \powset_\kappa(H_\theta)$
 is $ P$-complete if every $(X, P)$-generic filter has a lower bound
in~$P$.

 Define
  $\mathcal H(P) := \{X \in \powset_\kappa(H_\theta) \mid  X \text{ is $P$-complete}\}$.

Then a partial ordering $P$ is called almost $\kappa$-directed-closed
if $P$ is strategically $\kappa$-closed and 
$\mathcal H(P) $ is in every supercompact ultrafilter
on~$\powset_\kappa(H_\theta)$.
\end{dfn}
We will show that for the forcings $P$ we consider, the set $\mathcal H(P)$
contains all small elementary submodels of $H_\theta$, is therefore
closed unbounded, hence an element of every (fine) normal ultrafilter 
on $\powset_\kappa(H_\theta)$.  (See \cite[chap.~22 and 25.4]{Kanamori:1994}.)

\begin{dfn}\label{def.pivot.1}
   Let $G_1 \subseteq \QA$. 
   We call a triple $(\delta_1,S_1,E_1)$ a
    \emph{pivot} for $G_1$ if the following hold
     (where we write $\delta_2$ for the first inaccessible above $\delta_1$):
  \begin{itemize}
   \item $\delta_1 < \kappa$ (usually a limit ordinal).  % \in S_\inc^\kappa$.  
   \item $S_1, E_1$ are disjoint subsets of $\delta_1$, $E_1$ is club in $\delta_1$, $S_1$ is nowhere stationary in $\delta_1$. 
   \item $G_1\subseteq \QA$, $|G_1| < \delta_2$, $G_1$ is a filter.
   \item For all $p=(\epsilon,S,E) \in G_1$, 
     $(S_1, E_1)$ is ``stronger'' than $p$ in the following sense:
   \begin{itemize}
   \item $\epsilon < \delta_1$.
   \item $S\cap \epsilon = S_1 \cap \epsilon$, $E\cap \epsilon = E_1\cap \epsilon$. 
   \item $S \cap \delta_1 \subseteq S_1$. 
   \item $E \cap \delta_1 \supseteq E_1$. 
   \end{itemize}
  \end{itemize}
\end{dfn}
Note:  When we say that $G_1$ has a pivot, it is implied that $G_1$ 
is a filter of small cardinality. 

\begin{lem}[Master conditions in $\QA$] \label{c13-1}
 Assume that  $G_1 \subseteq \QA$ has a pivot.  
  Then $G_1$ has a lower bound in $\QA$, i.e.,
  $(\exists p^*\in \QA)\ (\forall p\in G_1)\ p^*\le p$. 
\end{lem}
\begin{proof}
   Let $(\delta_1,S_1,E_1)$ be a pivot for~$G_1$. 

   We let $p^*:= (\delta_1, S^*, E^*)$, where 
  \begin{itemize}
   \item  $S^* \cap \delta_1:= S_1 \cap \delta_1$. 
   \item  $E^* \cap \delta_1:= E_1 \cap \delta_1$. 
   \item $S^* \setmin \delta_1:= \bigcup_{(\epsilon,S,E)\in G_1} S\setmin \delta_1$.
   \item $E^* \setmin \delta_1:= \bigcap_{(\epsilon,S,E)\in G_1} E\setmin \delta_1$. 
  \end{itemize}
Note that the ideal of nowhere stationary subsets of $[\delta_1, \kappa)$
is $\delta_2$-closed, so $S^*$ is indeed nowhere stationary above $\delta_1$. (Also nowhere stationary below and up to~$\delta_1$, because $S_1$ had this
property.)

Hence $p^*$ is indeed a condition.  It is clear that $p^*$ is stronger 
than all $p\in G_1$.
\end{proof}
\begin{cor}
   Let $N\prec H_\theta$, $N\in \powset_\kappa(H_\theta)$, $\QA\in N$, 
   $N\cap \kappa\in \kappa$. 

   Then $N\in \mathcal H(\QA)$ (see Definition~\ref{def.almost}). 
\end{cor}
\begin{proof}
  Let $G\subseteq \QA\cap N$ be $(N,\QA)$-generic.   Let $\delta_1:=N\cap
\kappa$, and let $(S_1, E_1)$ be the generic object determined by $G$ as in
\ref{c6}. Then $(\delta_1,S_1,E_1)$ is a pivot for~$G$,
so by \ref{c13-1} we can find a lower bound for  $G$ in $\QA$.
\end{proof}

\subsection{Amoeba forcing, part 2}
\begin{dfn}
	\label{c7}
	Let $S \seq S_\inc^\kappa$.
    % and let $\langle \delta_j : j < \kappa \rangle$
    % enumerate $S$ in increasing order.
	Let $\QQ_{\kappa,S}^{\am 2}$ to be the forcing consisting
	of pairs $(\epsilon, \vec A )$ where:
	\begin{enumerate}
		\item 
		$\epsilon < \kappa$
		\item 
		$\vec A = (A_\delta:\delta\in S) \  \in\ 
        \prod_{\delta\in S } \id(\QQ_{\delta})$.
	\end{enumerate}
	For $p = (\epsilon_p, \vec A_p)$, $q = (\epsilon_q, \vec A_q)$ we define
	$q\le p$ iff either $q=p$ or:
	\begin{enumerate}
		\item 
		$\epsilon_p <   \epsilon_q$.
		\item 
		$ \vec A_p \on (S\cap  \epsilon_p ) = \vec A_q \on (S\cap \epsilon_p)$
		\item 
		For all $\delta\in S $ $A_p(\delta) \seq A _q(\delta)$.
	\end{enumerate}	
\end{dfn}

\begin{lem}
	\label{c8}
	Let $G$ be a $\QQ_{\kappa,S}^{\am 2}$-generic filter, let
	$$
	\vec A ^* = (A^*_\delta:\delta\in S)  = \bigcup_{(\epsilon,\vec A ) \in G} \vec A \on \epsilon 
     \ \ \in \ \  \prod_{\delta\in S  } \id(\QQ_{\delta})
	$$ 
	Then:
    \begin{enumerate}
      \item For all $(B_\delta:\delta\in S)$, where each $B_\delta\subseteq
       2^\delta$ is in $\id(\QQ_\delta)$, we have 
       $\forces (\forall ^\infty \delta) \ B_\delta \subseteq A^*_\delta$. 
      \item For all $B\in \id^-_0(\QQ_{\kappa,S})$ 
           we have $B \subseteq \set^-_0(\vec A^*)$.
    \end{enumerate}
\end{lem}

\begin{proof}
\begin{enumerate}
   \item 
      Let $p=(\epsilon,\vec A ) \in \QQ_{\kappa,S}^{\am 2}$.
	  Find $(\epsilon,\vec A') \in \QQ_{\kappa,S}^{\am 2}$ be such that:
	  \begin{enumerate}
		  \item 
		  $\vec A  \on (S\cap \epsilon) = \vec A' \on(S\cap  \epsilon)$.
		  \item 
		  For all $\delta \in S$ with $\delta  \geq \epsilon$ let
		  $ A'_\delta = A_\delta  \cup B_\delta$.
	  \end{enumerate}
	% Now check that $\nu \leq \eta$ and $\nu$ forces the lemma to be true.
      Now check
      that $A \subseteq \set_0^-(\vec A') \subseteq \set_0^-(\vec A^*)$

	Because $p$  was arbitrary we are done.
    \item Follows from 1.\qedhere
\end{enumerate}
\end{proof}

\begin{lem}
	\label{c9}
	Let $S \seq S_\inc^\kappa$. Then:
	\begin{enumerate}
		\item
		$\QQ_\kappa^{\am 2}$ is $\kappa$-strategically closed.
		\item
		$\QQ_\kappa^{\am 2}$ is $\kappa$-linked.
		\item 
		$\QQ_\kappa^{\am 2}$ satisfies $(*)_\kappa$.
	\end{enumerate}
\end{lem}

\begin{proof}
	Similar to~\ref{c16}.
\end{proof}

\begin{dfn}
	\label{c10}
	Let $\QQ_\kappa^{\am} := \QQ_\kappa^{\am 1} * \QQ_{\kappa, S^*}^{\am 2}$
	where $S^*$ is the generic object from~$\QA$  as in~\ref{c6}.
\end{dfn}

\begin{dis}
	\label{c11}
	Note that $\QQ_\kappa^{\am}$ here is not the same as the amoeba forcing
	$\QQ_\kappa^{\am}$ defined in \cite{Sh:1004}.
	But as we see in~\ref{c12} it is a modularized variant.
\end{dis}

\begin{lem}
	\label{c12}
	There exists $A^* \in \id^-(\QQ) \cap \mathbf V^{\QQ_\kappa^\am}$ such that:
	\begin{enumerate}
		\item 
		For every $A \in \mathbf V  \cap \id^-(\QQ_\kappa)$ we have
		$A \seq A^*$.
		\item 
		If $\kappa$ is weakly compact then 
		for every $A \in \mathbf V \cap \id(\QQ_\kappa)$ we have
		$A \seq A^*$.
	\end{enumerate}
\end{lem}

\begin{proof}
	$ $
	\begin{enumerate}
		\item 
		Combine~\ref{c6} and~\ref{c8} and check that 
		$A^* = \set_0^-(\langle A^*_\delta : \delta \in S^* \rangle)$ is as required.
		\item 
		By (1.) and~\ref{d6}. \qedhere
	\end{enumerate}
\end{proof}

The generic null set added by Amoeba forcing will cover all ground model sets sets in $\id^-$.   If $\kappa$ is weakly compact, then we also cover all 
$\id$ sets. So we are interested in 
keeping $\kappa$ weakly compact after our Amoeba iteration. 

\begin{dfn}\label{def.pivot.2}
   Let  $S\subseteq S_\inc^\kappa$ be nowhere stationary,
   and let $G_1 \subseteq \QS$. 

   We call a pair $(\delta_1, \vec A_1) $ a
    \emph{pivot} for $G_1$ if the following hold
     \begin{itemize}
        \item $\delta_1\in S_\inc^\kappa\setmin  S$. 
        \item $\vec A_1 = (A_{1,\delta}:\delta\in S\cap \delta_1)
                \in\prod_{\delta\in S\cap \delta_1}  \id(\QQ_\delta)$
        \item $G_1\subseteq \QS$, $|G_1| < \delta_2$, $G_1$ is a filter
            (where again $\delta_2$ is 
            the smallest inaccessible $>\delta_1$).
        \item  For all $p:=(\epsilon,\vec B)\in G_1$:
         \\ $\epsilon<\delta_1$,
          and  $(\delta_1,\vec A_1) $ is ``stronger'' than $p$
          in the sense that:
          \begin{itemize}
          \item 
                 $(\forall \delta < \delta_1)  
                  \ B_\delta \subseteq A_{1,\delta} $. 
          \item $(\forall \delta < \epsilon)
                  \ B_\delta = A_{1,\delta} $. 
          \end{itemize}
     \end{itemize}
\end{dfn}
\begin{lem}[Master conditions in $\QS$] \label{c13-2}
 Assume that   $S$ is nowhere stationary, and $G_1\subseteq \QS$
  has a pivot.   
  Then the set $G_1$ has a lower bound in $\QS$, i.e.,
  $(\exists p^*\in \QS)\ (\forall p\in G_1)\ p^*\le p$. 
\end{lem}
\begin{proof}
   Similar to the proof of Lemma~\ref{c13-1}.

   Let $(\delta_1, \vec A_1)$ be a pivot. 
   We define a condition $p^* = (\delta_1, \vec A ^*)$ as follows: 
  \begin{itemize}
   \item $( \forall \delta  \in S \cap \delta_1) \  A^*_\delta:= A_{1,\delta}$.
   \item $(\forall \delta \in S \setmin \delta_1) \ A^*_\delta:= \bigcup_{(\epsilon,\vec A)\in G_1}  A_\delta$. 
  \end{itemize}
  Why is $p$  condition?
  Because  for all $\delta\in \kappa\setmin \delta_1$, the 
  ideal $\id(\QQ_\delta)$ is $\delta_1$-complete, so 
   the set $\bigcup_{(  \epsilon,\nu)\in G_1} \nu(\delta)$ is in the ideal. 

   It is clear that $p^*\le p$ for all $p\in G_1$.
\end{proof}

\begin{cor}
   Let $N\prec H_\theta$, $N\in \powset_\kappa(H_\theta)$, $\QQ\in N$, 
   $N\cap \kappa\in \kappa$. 

   Then $N\in \mathcal H(\QQ)$ (see Definition~\ref{def.almost}). 
   \qed
\end{cor}

\subsection{Iterated Amoeba forcing}
\begin{notation}
   For every forcing notion $\PP$ we write $\Gamma_\PP$ for the canonical name
   of the generic filter on~$\PP$.
\end{notation}
\begin{dfn}
\label{c20}
\
\begin{enumerate}
   \item 
	Let $\mu$ be an ordinal and let
	$\PP$ be the limit of
	a ${<}\kappa$-support iteration 
	$\vec \PP= \langle \PP_\alpha, \dot \RR_\alpha : \alpha < \mu \rangle$.
    
    We call the iteration $\vec\PP$ and its limit $\PP$ 
    \emph{relevant}, if the following hold:
	For every $\alpha < \mu$ we have either
	\begin{enumerate}
		\item
		$\PP_\alpha \forces$``$\dot \RR_\alpha = \QA $'' or
		\item
		$\PP_\alpha \forces$``$\dot \RR_\alpha = \QS$ for some nowhere stationary $S \seq S_\inc^\kappa$'' or
		\item
		$\PP_\alpha \forces$``$\dot \RR_\alpha$ is ${<}\kappa$-directed closed''.		
	\end{enumerate}	
    (In particular, any $\mathord<\kappa$-directed closed forcing is 
    an example of a relevant iteration.)
    \item 
    Let $G_0\subseteq \PP$ be a filter.   For $\alpha<\mu$ 
    we will write $G_0\on \alpha$ for the set
     $\{ p\on \alpha:p\in G_0\}$, and 
     $G_0(\alpha)$ will be a $\PP_\alpha$-name  for the set 
      $\{ p(\alpha): p\in G_0\}$. 
      
      We remark that $G_0\on (\alpha+1)$ is a subset of $\PP_\alpha*\RR_\alpha$, so  the empty condition 
      of $\PP_\alpha$ forces
     %\begin{quote}
        ``If $G_0\on \alpha \subseteq \Gamma_{\PP_\alpha}$, 
        then $G_0(\alpha) \subseteq \RR_\alpha$.''
     %\end{quote}
   \item Let $G_0\subseteq \PP$ be a filter.  A sequence 
   $\langle \eta_\alpha:\alpha<\mu \rangle $   (where each $\eta_\alpha$
   is a $\PP_\alpha$-name) 
   is called a \emph{pivot} for $G_0$ if for all 
   $\alpha<\mu$ the following statement is forced:
\begin{quotation}
     If $G_0\on \PP_\alpha \subseteq \Gamma_{\PP_\alpha}$, then: 
       \begin{itemize}
          \item  $\RR_\alpha$ is $\mathord<\kappa$-directed closed, $\eta(\alpha)=0$.
          \item or:  $\eta(\alpha)$ is a pivot (in the sense of Definitions~\ref{def.pivot.1} or~\ref{def.pivot.2}, respectively) for $G_0(\alpha)\subseteq \RR_\alpha$.
       \end{itemize}
\end{quotation}
\end{enumerate}
\end{dfn}
\begin{lem}[Existence of master conditions in iterations]
	\label{c15}
    Assume that $\PP$ is the limit of a relevant iteration.  Let
    $G_0\subseteq \PP$ be a filter, and assume that there is a
    pivot for~$G_0$. 

	Then there exist $p^* \in \PP$ such that 
	$$
	(\forall p \in G_0)\ p^* \leq p.
	$$
\end{lem}
\begin{proof}
	We will define $p^*$ by induction, in each coordinate appealing 
to Lemma~\ref{c13-1} or~\ref{c13-2}, as appropriate.  (Note that 
fewer than $\kappa$ coordinates appear in the conditions in~$G_0$, 
so the resulting condition will have support of size  $<\kappa$.)
\end{proof}
\begin{cor}\label{almost.amoeba}
   Let $N\prec H_\theta$, $N\in \powset_\kappa(H_\theta)$, $N\cap \kappa\in \kappa$.  
Let $P\in N$ be a relevant iteration. 

   Then $N\in \mathcal H(P)$ (see Definition~\ref{def.almost}). 

  Hence by \cite[Theorem 9]{Ko:2006}: If $\kappa$ is supercompact, then 
  after forcing with a modified Laver preparation we obtain a model 
  in which $\kappa$ is not only supercompact, but moreover this supercompactness cannot be destroyed by almost $\kappa$-directed closed forcing, 
 so in particular not by relevant iterations.
 \qed
\end{cor}

\begin{dfn}
	\label{s29}
	Let $\mu$ be an ordinal.
	Let $\AA_{\kappa, \mu}$ be the limit of
	the ${<}\kappa$-support iteration 
	$\langle \AA_{\kappa,\alpha}, \dot \RR_\alpha : \alpha < \mu \rangle$
	where for every $\alpha < \mu$ we have:
	$$
	\AA_{\kappa,\alpha} \forces \dot \RR_\alpha = \begin{cases}
	\QQ_\kappa^\text{am} & \alpha \text{ even} \\
	\HH_\kappa & \alpha \text{ odd.}
	\end{cases}
	$$
\end{dfn}

\begin{fact}
	\label{s40}
	$\AA_{\kappa,\mu}$ is an iteration satisfying the requirements of~\ref{c15}. \qed
\end{fact}

\begin{lem}
	\label{s41}
	Let $\mu$ be an ordinal. Then $\AA_{\kappa,\mu}$ satisfies
	the stationary $\kappa^+$-Knaster condition and in particular
	$\AA_{\kappa,\mu}$ satisfies the $\kappa^+$-c.c.
\end{lem}

\begin{proof}
	By \ref{c16}, \ref{c9}, \ref{a6}, \ref{a1}.
\end{proof}

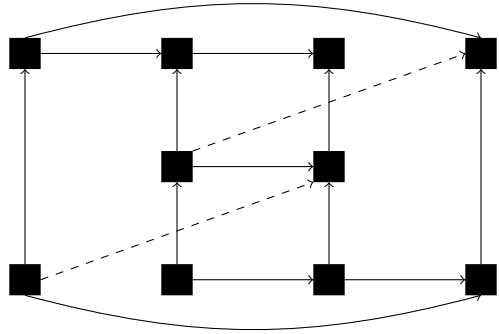
\begin{figure}[h]
	\centering
	
	\begin{tikzpicture}[
	big/.style={rectangle, draw=black!100, fill=black!100, thick, minimum size=1em},
	small/.style={rectangle, draw=black!100, fill=black!0, thick, minimum size=1em},
	unknown/.style={rectangle, draw=black!100, fill=black!30, thick, minimum size=1em},
	]
	%Nodes
	\node[big]	at (0, 0) (addn)   { };
	\node[big]	at (0, 3) (covn)   { };
	\node[big]		at (2, 0) (addm)   { };
	\node[big]		at (2, 3) (nonm)   { };
	\node[big]		at (4, 0) (covm)   { };
	\node[big]		at (4, 3) (cfm)    { };
	\node[big]		at (6, 0) (nonn)   { };
	\node[big]		at (6, 3) (cfn)    { };
	
	\node[big]		at (2, 1.5) (b) { };
	\node[big]		at (4, 1.5) (d) { };
	
	%Lines
	\draw[->] (addn.north) -- (covn.south);
	\draw[->] (addm.north) -- (b.south);
	\draw[->] (b.north)    -- (nonm.south);
	\draw[->] (covm.north) -- (d.south);
	\draw[->] (d.north)    -- (cfm.south);
	\draw[->] (nonn.north) -- (cfn.south);
	
	\draw[->] (covn.east)  -- (nonm.west);
	\draw[->] (addm.east)  -- (covm.west);
	\draw[->] (b.east)     -- (d.west);
	\draw[->] (nonm.east)  -- (cfm.west);
	\draw[->] (covm.east)  -- (nonn.west);

	\draw[->,dashed] (addn.east)  -- (d.south west);
	\draw[->,dashed] (b.north east)  -- (cfn.west);
	
	\draw[->] (addn.south)  to [out=345,in=195] (nonn.south);
	\draw[->] (covn.north)  to [out=15,in=165] (cfn.north);
	\end{tikzpicture}
	
	\caption{The Amoeba model}
\end{figure}

\begin{thm}
	\label{s19}
	Let $\mathbf V \models 2^\kappa = \kappa^+$ and let $\kappa$ be supercompact, indestructible in the sense of
	\ref{def.almost}.
	Then $\mathbf V^{\AA_{\kappa,\kappa^{++}}}$ satisfies:
	\begin{enumerate}
		\item
		$2^\kappa = \kappa^{++}$
		\item
		$\add(\QQ_\kappa) = \kappa^{++}$
		\item
		$\add(\Cohen_\kappa) = \kappa^{++}$.
	\end{enumerate}
\end{thm}

\begin{proof}
	$ $
	\begin{enumerate}
		\item Should be clear.
		\item
		By (1.) is suffices to show $\add(\QQ_\kappa) \geq \kappa^{++}$.
		So towards contradiction assume $\add(\QQ_\kappa) = \kappa^+$
		and let $\langle B_i : i < \kappa^+ \rangle$ witness it.
		Remember $\AA_{\kappa,\kappa^{++}}$ satisfies the $\kappa^+$-c.c.
		by \ref{s41}.
		So there exists $\alpha < \kappa^{++}$ such that
		$B_i \in \mathbf V^{\PP_\alpha}$ for every $i < \kappa^+$.
		But by~\ref{c12} there exists $A \in V^{\PP_{\alpha+2}} \cap \id(\QQ_\kappa)$
		such that
		$B_i \seq A$ for every $i < \kappa^+$.
		By~\ref{m0} also $\mathbf V^{\AA_{\kappa,\kappa^{++}}} \models A \in \id(\QQ_\kappa)$.
		Contradiction.		
		\item
		Argue as in~\ref{n4}.
		\qedhere
	\end{enumerate}
\end{proof}

\subsection{The Short Amoeba Model}
\begin{figure}[h]
	\centering
	
	\begin{tikzpicture}[
	big/.style={rectangle, draw=black!100, fill=black!100, thick, minimum size=1em},
	small/.style={rectangle, draw=black!100, fill=black!0, thick, minimum size=1em},
	unknown/.style={rectangle, draw=black!100, fill=black!30, thick, minimum size=1em},
	]
	%Nodes
	\node[small]	at (0, 0) (addn)   { };
	\node[small]	at (0, 3) (covn)   { };
	\node[small]		at (2, 0) (addm)   { };
	\node[small]		at (2, 3) (nonm)   { };
	\node[small]		at (4, 0) (covm)   { };
	\node[small]		at (4, 3) (cfm)    { };
	\node[small]		at (6, 0) (nonn)   { };
	\node[small]		at (6, 3) (cfn)    { };
	
	\node[small]		at (2, 1.5) (b) { };
	\node[small]		at (4, 1.5) (d) { };
	
	%Lines
	\draw[->] (addn.north) -- (covn.south);
	\draw[->] (addm.north) -- (b.south);
	\draw[->] (b.north)    -- (nonm.south);
	\draw[->] (covm.north) -- (d.south);
	\draw[->] (d.north)    -- (cfm.south);
	\draw[->] (nonn.north) -- (cfn.south);
	
	\draw[->] (covn.east)  -- (nonm.west);
	\draw[->] (addm.east)  -- (covm.west);
	\draw[->] (b.east)     -- (d.west);
	\draw[->] (nonm.east)  -- (cfm.west);
	\draw[->] (covm.east)  -- (nonn.west);

	\draw[->,dashed] (addn.east)  -- (d.south west);
	\draw[->,dashed] (b.north east)  -- (cfn.west);
	
	\draw[->] (addn.south)  to [out=345,in=195] (nonn.south);
	\draw[->] (covn.north)  to [out=15,in=165] (cfn.north);
	\end{tikzpicture}
	
	\caption{The short Amoeba model}
\end{figure}
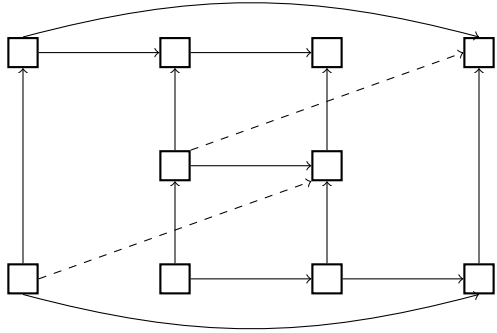

\begin{thm}
	\label{s20}
	Let $\mathbf V \models 2^\kappa = \kappa^+$ and let $\kappa$ be supercompact, indestructible in the sense of
	\ref{def.almost}. Let $\mu = \kappa^{++}\cdot\kappa^+$.
	Then $\mathbf V^{\AA_{\kappa,\mu}}$ satisfies:
	\begin{enumerate}
		\item
		$2^\kappa = \kappa^{++}$
		\item
		$\cf(\QQ_\kappa) = \kappa^+$.
		\item
		$\mathfrak d_\kappa = \kappa^+$.
		\item
		$\cf(\Cohen_\kappa) = \kappa^+$.
	\end{enumerate}
\end{thm}

\begin{proof}
	$ $
	\begin{enumerate}
		\item
		Should be clear.
		\item
		Let $\langle \mu_i : i < \kappa^+ \rangle$ be a cofinal sequence in
		$\mu$ such that for each $i < \kappa^+$ we have $\mu_i$ is even.
		Let $A_i$ be the null set added by~$\dot \RR_{\mu_i}$.
		Easily by~\ref{c12} the sequence $\langle A_i : i < \kappa^+ \rangle$
		is cofinal in~$\id(\QQ_\kappa)$.
		\item
		Let $\eta_i$ be the Hechler real added by~$\dot \RR_{\mu_i + 1}$.
		Easily by~\ref{n2} the sequence $\langle \eta_i : i < \kappa^+ \rangle$
		is dominating. 
		\item
		Assume $\cf(\Cohen_\kappa) > \kappa^+$. Then by (3.) and~\ref{z8} and (2.)
		$\cf(\Cohen_\kappa) \leq \non(\QQ_\kappa) \leq \cf(\QQ_\kappa) = \kappa^+$.
		Contradiction.
		\qedhere
	\end{enumerate}
\end{proof}

\subsection{Cohen-Amoeba Forcing}

\begin{dfn}
	\label{s21}
	Let $\CC_\kappa^\am$ be the set of all pairs $(\alpha, A)$ such that:
	\begin{enumerate}
		\item 
		$\alpha < \kappa$.
		\item
		$A \seq \tle \kappa$ is a tree.
		\item
		$[A] \seq 2^\kappa$ is non-empty nowhere dense.
	\end{enumerate}
	For $p = (\alpha_p, A_p)$, $ q = (\alpha_q, A_q)$, $ p, q \in \CC_\kappa^\am$
	we define $q$ stronger than $p$ if:
	\begin{enumerate}
		\item 
		$\alpha_q \geq \alpha_p$.
		\item
		$A_q \supseteq A_p$.
		\item
		$A_q \on \alpha_p = A_p \on \alpha_p$.
	\end{enumerate}
	We call $\CC_\kappa^\am$ the Cohen-Amoeba forcing.
	
	Note that 	$\CC_\kappa^\am$ is a straightforward generalization of
	the universal meager forcing defined in \cite[3.1.9]{BJ:1995}.
\end{dfn}

\begin{lem}
	\label{s31}
	Let $\langle A_i : i < i^* < \kappa \rangle$ be a family of
	nowhere dense subsets of $2^\kappa$. Then $A = \bigcup_{i < i^*} A_i$ is nowhere dense.
\end{lem}

\begin{proof}
	For $i < i^*, s \in \tle \kappa$ let $t(i, s) \in \tle \kappa$ be such that
	\begin{enumerate}
		\item
		$s \trianglelefteq t(i,s)$.
		\item 
		$A_i \cap [t(i,s)] = \emptyset$.
	\end{enumerate}
	Let $s \in A$ and we define
	an increasing sequence $\langle \eta_i : i < i^*\rangle$ as such that:
	\begin{enumerate}
		\item 
		$\eta_0 = s$.
		\item
		$i = j + 1 \Rightarrow \eta_i = t(j, \eta_j)$.
		\item
		If $i$ is a limit ordinal then
		$\eta_i = \bigcup_{j < i} \eta_j$.
	\end{enumerate}
	Let $\eta = \bigcup_{i < i^*} \eta_i$ and check:
	\begin{enumerate}
		\item
		$s \trianglelefteq \eta$.
		\item
		$A \cap [\eta] = \emptyset$.
	\end{enumerate}
	Because $s$ was arbitrary we are done.
\end{proof}

\begin{lem}
	\label{s30}
	$ $
	\begin{enumerate}
		\item
		$\CC_\kappa^\am$ is ${<}\kappa$-directed closed.
		\item
		$\CC_\kappa^\am$ is $\kappa$-linked.
		\item
		$\CC_\kappa^\am$ satisfies $(*)_\kappa$. \qed
	\end{enumerate}	
\end{lem}

\begin{proof}
	$ $
	\begin{enumerate}
		\item 
		Easy using \ref{s31}.
		\item
		Should be clear.
		\item
		By (1), (2), and~\ref{a8}.
	\end{enumerate}
\end{proof}

\begin{lem}
	\label{s22}
	Let $G$ be generic for $\CC_\kappa^{\text{am}}$ and let
	$N = \bigcup_{(\alpha, A) \in G}A$. Then for the set
	$$
	M = \{\eta \in 2^\kappa : (\exists \nu \in N)\ 
	\nu =^*\eta \} 
	$$
	we have:
	\begin{enumerate}
		\item 
		$M$ is meager.
		\item
		$M$ covers every meager set $X \in \mathbf V$. 
		\\More precisely: 
		   for every family $(X_i:i<\kappa)\in\mathbf V$
		   of nowhere dense 
		   trees it is forced that $(\forall i<\kappa)\, [X_i]\subseteq
		   M$ holds.
	\end{enumerate}
\end{lem}

\begin{proof}
	$ $
	\begin{enumerate}
		\item 
		It suffices to show that $M$ is nowhere dense. We check that for each
		$s \in \tle \kappa$ the set
		$$
		D_s = \{
		q \in \CC_\kappa^{\text{am}} : (\exists t \trianglerighteq s)\ q \forces
		\text{``}N \cap [t] = \emptyset\text{''}
		\}
		$$
		is dense in~$\CC_\kappa^{\text{am}}$. Indeed for any
		$(\alpha, A) \in \CC_\kappa^{\text{am}}$ there exists $t \trianglerighteq s$
		such that $A \cap  [t] = \emptyset$. Now easily $(\max(\alpha, |t|), A)
		\in D_s$.
		\item
		Let $X \seq \tle \kappa$ such that $[X]$ is nowhere dense and let $(\alpha, A) \in \CC_\kappa^{\text{am}}$.
		Without loss of generality we may assume $|X \cap 2^\alpha| = 1$ (otherwise
		we just split up~$X$).
		Now find $\rho \in A \cap 2^\alpha$ and let
		$$
		X' = \{
		\eta \in 2^\kappa : (\exists \nu \in X)\ \eta =^* \nu, \eta \on \alpha = \rho
		\}.
		$$
		Easily $q = (\alpha, A \cup X') \in \CC_\kappa^{\text{am}}$
		and $q$ forces $X$ to be covered by~$M$.
		\qedhere
	\end{enumerate}
\end{proof}

\begin{thm}
	\label{s23}.
	Let $\mathbf V \models 2^\kappa = \kappa^+$. Let $\PP = \{\PP_i, \dot \RR_i : i < \mu\rangle$ be the
	limit of a ${<}\kappa$-support iteration such that
	that $\PP_i \forces$``$\dot \RR_i = \CC_\kappa^{\text{am}}$'' for each $i < \mu$.
	Then $\mathbf V^\PP$ satisfies:
	\begin{enumerate}
		\item
		If $\mu = \kappa^{++}$ then $\add(\Cohen_\kappa) = \kappa^{++}$.
		\item
		If $\cf(\mu) = \kappa^+$ then $\cf(\Cohen_\kappa) = \kappa^+$.
	\end{enumerate}
\end{thm}

\begin{proof}
	$ $
	\begin{enumerate}
		\item 
		Use~\ref{s22} and argue as in~\ref{s19}(2.).
		\item
		Use~\ref{s22} and argue as in~\ref{s20}(2.).
		\qedhere
	\end{enumerate}
\end{proof}

\begin{cor}
	\label{s24}
	We could use $\CC_\kappa^{\text{am}}$ instead of $\HH_\kappa$ for odd iterants
	in the definition of $\AA_{\kappa, \mu}$ in~\ref{s29} to achieve the same results in~\ref{s19} and~\ref{s20}
	in regard to the characteristics of the diagram.
\end{cor}

\subsection{Bounded Perfect Tree Forcing}

We give a $\kappa$-support alternative to the short Hechler model.

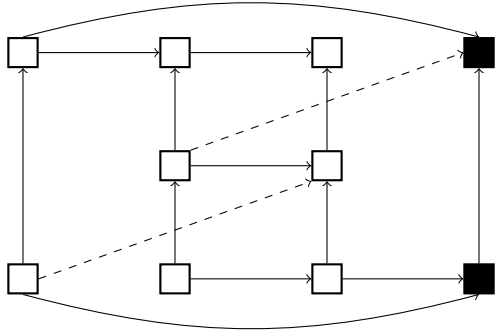
\begin{figure}[h]
	\centering
	
	\begin{tikzpicture}[
	big/.style={rectangle, draw=black!100, fill=black!100, thick, minimum size=1em},
	small/.style={rectangle, draw=black!100, fill=black!0, thick, minimum size=1em},
	]
	%Nodes
	\node[small]	at (0, 0) (addn)   { };
	\node[small]	at (0, 3) (covn)   { };
	\node[small]	at (2, 0) (addm)   { };
	\node[small]	at (2, 3) (nonm)   { };
	\node[small]	at (4, 0) (covm)   { };
	\node[small]	at (4, 3) (cfm)    { };
	\node[big]		at (6, 0) (nonn)   { };
	\node[big]		at (6, 3) (cfn)    { };
	
	\node[small]	at (2, 1.5) (b) { };
	\node[small]		at (4, 1.5) (d) { };
	
	%Lines
	\draw[->] (addn.north) -- (covn.south);
	\draw[->] (addm.north) -- (b.south);
	\draw[->] (b.north)    -- (nonm.south);
	\draw[->] (covm.north) -- (d.south);
	\draw[->] (d.north)    -- (cfm.south);
	\draw[->] (nonn.north) -- (cfn.south);
	
	\draw[->] (covn.east)  -- (nonm.west);
	\draw[->] (addm.east)  -- (covm.west);
	\draw[->] (b.east)     -- (d.west);
	\draw[->] (nonm.east)  -- (cfm.west);
	\draw[->] (covm.east)  -- (nonn.west);

	\draw[->,dashed] (addn.east)  -- (d.south west);
	\draw[->,dashed] (b.north east)  -- (cfn.west);
	
	\draw[->] (addn.south)  to [out=345,in=195] (nonn.south);
	\draw[->] (covn.north)  to [out=15,in=165] (cfn.north);
	\end{tikzpicture}
	
	\caption{The bounded perfect tree model}
\end{figure}

\begin{dfn}
	\label{s0}
	Let:
	\begin{enumerate}
		\item
		$S \seq \kappa \cap S_\text{inc}$, $\sup(S) = \kappa$,
		$\partial \in S \Rightarrow \partial > \sup(\partial \cap S_\text{inc})$
		
		\item
		$\langle \partial_\epsilon : \epsilon < \kappa \rangle$ enumerates $S$ in increasing order.
		\item
		$\theta_\epsilon = 2^{\partial_\epsilon}$ for $\epsilon < \kappa$.
		
		\item
		$T = \bigcup_{\zeta < \kappa} T_\zeta$ where
		$T_\zeta = \prod_{\epsilon < \zeta} \theta_\epsilon$.
	\end{enumerate}
	
	We define $\TT_\kappa^S$ to be the set of all $p \seq T$ such that:
	\begin{enumerate} [~~~(a)]
		\item
		For all $\eta \in p$ we have
		$\nu \trianglelefteq \eta \Rightarrow \nu \in p$.
		\item
		There exists a club $E \seq \kappa$ such that for all $\eta \in p$:
		$$
		\suc_p(\eta) = 
		\{i < \theta_{\lh(\eta)} : \eta ^\frown i \in p\} =
		\begin{cases}
		\theta_{\lh(\eta)} &\quad\text{ if } \lh(\eta) \in E \\
		\{p ^\frown i^*\}  &\quad\text{ if } \lh(\eta) \not \in E, \text{ for some } i^* < \theta_{\lh(\eta)}	
		\end{cases}
		$$
		\item 
		No branches die out in $p$. I.e.
		If $\zeta$ is a limit ordinal and $\eta \in T_\zeta$ then:
		$$
		\eta \in p \Leftrightarrow (\forall \epsilon < \zeta)\ \eta \on \epsilon \in p.
		$$ 
	\end{enumerate}
	So $\TT_\kappa^S$ is the forcing of all subtrees of $T$ that split fully
	on a club $E \seq \kappa$ of levels and otherwise do not split.
	The order is defined the usual way, i.e.\ for $p, q \in \TT_\kappa^S$
	we have $q$ stronger than $p$ iff $q \seq p$.
	Because for our purposes every $S$ works
	we will simply write $\TT_\kappa$ instead of~$\TT_\kappa^S$.
	
\end{dfn}

\begin{dfn}
	\label{s0.1}
	Let $\TT_{\kappa, \mu}$ be the limit of
	the $\kappa$-support iteration 
	$\langle \TT_{\kappa,\alpha}, \dot \RR_\alpha : \alpha < \mu \rangle$
	where $\TT_{\kappa,\alpha} \forces$``$\dot \RR_\alpha = \TT_\kappa$''
	for every $\alpha < \mu$.
\end{dfn}

\begin{lem}
	\label{s1}
	$ $
	\begin{enumerate}
		\item
		$\TT_\kappa$ is ${<}\kappa$-directed closed.
		\item
		$\TT_{\kappa,\kappa^{++}}$ is ${<}\kappa$-directed closed.
	\end{enumerate}
\end{lem}

\begin{proof}
	$ $
	\begin{enumerate}
		\item
		Let $D$ be a directed subset of $\TT_\kappa$ of size $<\kappa$. Intersecting 
		the club sets associated with each $p \in D$ will give us a club 
		set $E$.   Letting $q$ be the intersection of all $p\in D$, we 
		claim that $q$ is a condition.  It is then clear that $q$ is a lower
		bound for $D$.

		Clearly $q$ is nonempty and satisfies condition \ref{s0} (a), (c). It remains
		to verify (b).	
		Let $\eta\in q$. 
		
			\underline{Case 1:} 
			$\lh(\eta)\in E$.  So $\lh(\eta)\in E_p$ for all $p\in D$, 
			hence $\suc_q(\eta) = \bigcap_{p\in D} \suc_p(\eta) = \theta_{\lh(\eta)}$.
			
			\underline{Case 2:}
			$\lh(\eta)\notin E$.  So there is some $p^*\in D$ and
			some  $i^*$ such that $\suc_{p^*}(\eta) = \{i^*\}$.  As $D$ is directed, 
			and $\eta\in p$ for all $p\in D$, we also have $\eta^\frown i^*\in p$
			for all $p\in D$.  Hence $\suc_q(\eta) = \bigcap_{p\in D} \suc_p(\eta) = 
			\{i^*\}$, as required.
		\item
		By~\ref{a22}.\qedhere
	\end{enumerate}
\end{proof}

\begin{dfn}
	\label{s2}
		Let $\alpha < \kappa$, $p, q \in \TT_\kappa$ and let
		$\langle e_i : i < \kappa\rangle$
		be an enumeration of the club of splitting levels of~$p$.
		We define
		$$
		q \leq_\alpha p \quad\text{iff}\quad q \leq p \landx
		q \cap \tleq {e_\alpha} = p \cap \tleq {e_\alpha}.
		$$
\end{dfn}

\begin{lem}
	\label{s3}
	Let $\squ p = \langle p_i : i < \kappa\rangle$ be a sequence of conditions in
	$\TT_\kappa$ such that $i < j < \kappa \Rightarrow p_j \leq_i p_i$.
	Then $\squ p$ has a lower bound $q \in \TT_\kappa$.
\end{lem}
\begin{proof}
	It is easy to check that $q = \bigcap_{i< \kappa} p_i$
	is a condition in $\TT_\kappa$ and a lower bound for $\squ p$.
\end{proof}

\begin{dfn}
	\label{s4}
	We refer to sequences as in~\ref{s3} as {\em fusion sequences}.
\end{dfn}

\begin{lem}
	$\ $
	\label{s27}
	\begin{enumerate}[(a)]
		\item
		White has a winning strategy for $\mathfrak F_\kappa(\TT_\kappa,p)$
		for every $p \in \TT_\kappa$.
		\item
		White has winning strategy for $\mathfrak F_\kappa(\TT_{\kappa,\kappa^{++}},p)$
		for every $p \in \TT_{\kappa,\kappa^{++}}$.
	\end{enumerate}
\end{lem}

\begin{proof}
	$\ $
	\begin{enumerate}[(a)]
		\item 
		We are going to construct a fusion sequence $\langle p_\zeta : \zeta < \kappa \rangle$
		and a winning strategy for White such that
		\begin{enumerate}[(1)]
			\item 
			$p_0 = p$.
			\item
			In the $\zeta$-round White plays $\mu_\zeta = |p_\zeta \cap T_\beta|$
			and $p_{\zeta, i} = p^{[\eta_{\zeta,i}]}$ where
			$\langle \eta_{\zeta,i} : i < \mu_\zeta \rangle$ enumerates
			$p_\zeta \cap T_\beta$ and
			$\beta$ is the $\zeta$-th splitting level of $p_\zeta$.
			\item 
			$p_{\zeta+1} = \bigcup_{i < \mu_\zeta} p'_{\zeta,i}$ where
			$p'_{\zeta,i}$ are the moves played by Black.
			\item 
			For $\delta$ a limit ordinal $p_\delta = \bigcap_{\zeta < \delta} p_\zeta$.
		\end{enumerate}
		Now use~\ref{s3} and check that $q = \bigcap_{\zeta < \kappa} p_\zeta$
		witnesses that White wins.
		\item 
		By \ref{h3}.
		\qedhere
	\end{enumerate}
\end{proof}

\begin{lem}
	\label{s5}
	$ $
	\begin{enumerate}[(a)]
		\item
		$\TT_{\kappa,\kappa^{++}}$ does not collapse $\kappa^+$
		\item
		Let $N$ be a $\kappa$-meager set in $\mathbf V^{\TT_{\kappa,\kappa^{++}}}$. Then there exists a
		$\kappa$-meager set $M \in \mathbf V$ such that $N \seq M$.
		\item
		In particular: If $\mathbf V \models 2^\kappa = \kappa^+$ then
		$\mathbf V^{\TT_{\kappa,\kappa^{++}}} \models \cov(\Cohen_\kappa) = \kappa^+$.
	\end{enumerate}
\end{lem}

\begin{proof}
	By \ref{s27}, \ref{h2}.
\end{proof}

\begin{lem}
	\label{s8}
	If $\mathbf V \models 2^\kappa = \kappa^+$ then:
	\begin{enumerate}[(a)]
		\item
		$\TT_\kappa$ satisfies the $\kappa^{++}$-c.c.
		\item
		$\TT_{\kappa,\kappa^{++}}$ satisfies the $\kappa^{++}$-c.c.
	\end{enumerate}
\end{lem}

\begin{proof}
	$\ $
	\begin{enumerate}[(a)]
		\item
		By our assumption:
		$|\TT_\kappa| = \kappa^+$.
		\item
		By \ref{s27}, \ref{s16} and the Solovay-Tennenbaum theorem (see \cite{ST:1971}).
		\qedhere
	\end{enumerate}
\end{proof}

\begin{lem}
	$ $
	\label{s9}
	\begin{enumerate}[(a)]
		\item 
		$\TT_\kappa \forces (2^\kappa)^V \in \id^-(\QQ_\kappa)$.
		\item 
		$\mathbf V^{\TT_{\kappa,\kappa^{++}}} \models \non(\id^-(\QQ_\kappa)) \geq \kappa^{++}$.
		\item 
		$\mathbf V^{\TT_{\kappa,\kappa^{++}}} \models \non(\id(\QQ_\kappa)) \geq \kappa^{++}$.
	\end{enumerate}
\end{lem}
\begin{proof}
	$ $
	\begin{enumerate}[(a)]
		\item 
		Let $\langle A_{\epsilon,i} : i < \theta_\epsilon\rangle$ be a covering sequence in
		$\id(\QQ_{\partial_\epsilon})$. Let $\dot \nu$ be a name for the generic $\kappa$-real
		added by $\TT_\kappa$ and define $\squ \Lambda = \langle \Lambda_\partial : \partial \in S \rangle$
		such that $\set_0(\Lambda_{\partial_\epsilon}) = A_{\epsilon, \dot \nu(\epsilon)}$.
		Now $\Lambda$ witnesses $(2^\kappa)^{\mathbf V} \in \id^-(\QQ_\kappa)$ in
		$\mathbf V^{\TT_\kappa}$.
		\item 
		Remember that by~\ref{s8} all Borel sets appear in
		$\mathbf V^{\TT_{\kappa,\alpha}}$ for some $\alpha < \kappa^{++}$.
		So (b) follows from (a), remembering~\ref{s1}, \ref{a21}, \ref{m0}.
		\item 
		Remember $\id^-(\QQ_\kappa) \seq \id(\QQ_\kappa)$ hence
		$\non(\id^-(\QQ_\kappa)) \leq \non(\id(\QQ_\kappa))$.
		So this follows from (b). \qedhere
	\end{enumerate}
\end{proof}

\begin{dis}
	\label{s10}
	The coverings in~\ref{s9} could be just be sequences of singletons.
	So we could say that the lemma speaks on some ideal $\id^{--}$ that is defined
	similar to~$\id^-$ just with singletons (or maybe sets of size at most $\kappa$) instead of $\id(\QQ_\delta)$-sets on each level.
	So we really show $\non(\id^{--}(\QQ_\kappa)) \geq \kappa^{++}$.
\end{dis}

\begin{thm}
	\label{s28}
	If $\mathbf V \models 2^\kappa = \kappa^+$ then
	$\mathbf V^{\TT_{\kappa,\kappa^{++}}} \models 2^\kappa = \kappa^{++}$.
	\qed
\end{thm}

\newpage
\section{Slaloms} \label{slaloms}

It is well known that slaloms can be used to characterize the additivity and
cofinality of measure in the classical case, see for example \cite{BJ:1995}.
In \cite{BBFM:2016}
this result motivates a definition: The cardinals 
add(null) and cof(null) are replaced 
 by the appropriate additivity and covering 
numbers for slaloms. 

This raises the question how the characteristics introduced there related to the
characteristics of $\id(\QQ_\kappa)$ discussed here.
In particular one might wonder if the generalized characterization of the additivity
of null by slaloms is equal to the additivity of~$\id(\QQ_\kappa)$.
It turns out that for partial slaloms the answer is negative. We conjecture that for total slaloms the answer is negative too, see~\ref{i8} and~\ref{i9} respectively.

\subsection{Recapitulation}
Let us start with some results and definitions from \cite{BBFM:2016} (for more details
and proofs see there).
Since there also successor cardinals $\kappa$ are considered, let us 
remind the reader of that in this paper the cardinal 
 $\kappa$ is always  (at least) inaccessible.

\begin{dfn}
	\label{i1}
	Let $h \in \kappa^\kappa$ be an unbounded function. We define
	$$
	\mathcal C_h = \{\phi \in ([\kappa]^{{<}\kappa})^\kappa :
	(\forall i < \kappa)\ \phi(i) \in [\kappa]^{|h(i)|}\}.
	$$
	For $\phi \in \mathcal C_h$, $ f \in \kappa^\kappa$ we define
	$$
	f \in^* \phi \quad\Leftrightarrow\quad
	(\forall^\infty i < \kappa)\ f(i) \in \phi(i).
	$$
	Finally let:
	\begin{enumerate}
		\item 
		$
		\add(\hsl)  = \min \{
		|\mathcal F| : \mathcal F \seq \kappa^\kappa, (\forall \phi \in \mathcal C_h)
		(\exists f \in \mathcal F)\ f \not \in^* \phi
		\}
		$.
		\item
		$
		\cf(\hsl)  = \min \{
		|\Phi| : \Phi \seq \mathcal C_h, (\forall f \in \kappa^\kappa)
		(\exists \phi \in \Phi)\ f \in^* \phi 
		\}
		$.
	\end{enumerate}
\end{dfn}

\begin{dfn}
	\label{i2}
	We may also consider partial slaloms. Let $h \in \kappa^\kappa$ be unbounded
	and define
	$$
	\pp\mathcal C_h = \{
	\phi : (\exists \psi \in \mathcal C_h)\ \phi \seq \psi, |\dom(\phi)| = \kappa
	\}.
	$$
	Again for $\phi \in \pp\mathcal C_h$, $ f \in \kappa^\kappa$ we define
	$$
	f \plocal \phi \quad\Leftrightarrow\quad
	(\forall^\infty i \in \dom(\phi))\ f(i) \in \phi(i).
	$$
	Finally let:
	\begin{enumerate}
		\item 
		$
		\add^{\ppp}(\hsl) = \min \{
		|\mathcal F| : \mathcal F \seq \kappa^\kappa, (\forall \phi \in \pp \mathcal C_h)
		(\exists f \in \mathcal F)\ f \pp{\not \in}^* \phi
		\}
		$.
		\item
		$
		\cf^{\ppp}(\hsl)  = \min \{
		|\Phi| : \Phi \seq \pp \mathcal C_h, (\forall f \in \kappa^\kappa)
		(\exists \phi \in \Phi)\ f \plocal \phi 
		\}
		$.
	\end{enumerate}
\end{dfn}

\begin{dis}
	\label{i19}
	Note that in \cite{BBFM:2016} the notation
	$\add(\hsl) = \mathfrak b_h(\local)$,
	$\cf(\hsl) = \mathfrak d_h(\local)$
	and similarly
	$\add^{\ppp}(\hsl) = \mathfrak b_h(\plocal)$,
	$\cf^{\ppp}(\hsl) = \mathfrak d_h(\plocal)$	
	is used.
	
\end{dis}

\begin{lem}
	\label{i17}
	Let $h \in \kappa^\kappa$ be unbounded. Then:
	\begin{itemize}
		\item
		$\add(\hsl)\leq
		\add^{\ppp}(\hsl)
		\leq \add(\Cohen_\kappa)$.
		\item
		$\cf(\hsl) \geq
		\cf^{\ppp}(\hsl)
		\geq \cf(\Cohen_\kappa)$.
		\qed
	\end{itemize}
\end{lem}

\begin{lem}
	\label{i18}
	Let $h, g \in \kappa^\kappa$ be unbounded. Then:
	\begin{itemize}
		\item
		$\add^{\ppp}(\hsl) =
		\add^{\ppp}(g\text{-slalom})$.
		\item
		$\cf^{\ppp}(\text{-slalom}) =
		\cf^{\ppp}(g\text{-slalom})$.
		\qed
	\end{itemize}

\end{lem}

\begin{dis}
	\label{i20}
	So for partial slaloms we may
	ignore $h$ and write $\add^{\ppp}(\kappa)$ instead of $\add^{\ppp}(\hsl)$
	and similarly $\cf^{\ppp}(\kappa)$ instead of~$\cf^{\ppp}(\hsl)$.
\end{dis}

\begin{figure}[p]
	\centering
	
	\begin{tikzpicture}[
	big/.style={rectangle, draw=black!0, fill=black!100, thick, minimum size=1em},
	small/.style={rectangle, draw=black!0, fill=black!0, thick, minimum size=1em},
	]
	%Nodes
	\node[small]	at (-.5, -4) (kappaplus)   {$\kappa^+$};
	\node[small]	at (-.5, 0) (addn)   {$\add(\QQ_\kappa)$};
	\node[small]	at (-.5, 4) (covn)   {$\cov(\QQ_\kappa)$};
	\node[small]	at (3, 0) (addm)   {$\add(\Cohen_\kappa)$};
	\node[small]	at (3, 4) (nonm)   {$\non(\Cohen_\kappa)$};
	\node[small]	at (6, 0) (covm)   {$\cov(\Cohen_\kappa)$};
	\node[small]	at (6, 4) (cfm)    {$\cof(\Cohen_\kappa)$};
	\node[small]	at (9.5, 0) (nonn)   {$\non(\QQ_\kappa)$};
	\node[small]	at (9.5, 4) (cfn)    {$\cof(\QQ_\kappa)$};
	\node[small]	at (9.5, 8) (cont)   {$2^\kappa$};
	
	\node[small]	at (3, 2) (b) {$\mathfrak b_\kappa$};
	\node[small]	at (6, 2) (d) {$\mathfrak d_\kappa$};
	
	\node[small]	at (3, -2) (addpartslalom) {$\add^{\ppp}(\kappa)$};
	\node[small]	at (3, -4) (addslalom) {$\add(\hsl)$};
		
	\node[small]	at (6, 6) (cfpartslalom) {$\cf^{\ppp}(\kappa)$};
	\node[small]	at (6, 8) (cfslalom) {$\cf(\hsl)$};

	%\node[small]	at (1.5, 1) (addnst) {$\add(\nst_\kappa^{\pr})$};
	%\node[small]	at (7.5, 3) (cofnst) {$\cof(\nst_\kappa^{\pr})$};
		
	%Lines
	\draw[->] (addn.north) -- (covn.south);
	\draw[->] (addm.north)  -- (b.south);
	\draw[->] (b.north)    -- (nonm.south);
	\draw[->] (covm.north) -- (d.south);
	\draw[->] (d.north)    -- (cfm.south);
	\draw[->] (nonn.north) -- (cfn.south);
	
	\draw[->] (covn.east)  -- (nonm.west);
	\draw[->] (addm.east)  -- (covm.west);
	\draw[->] (b.east)     -- (d.west);
	\draw[->] (nonm.east)  -- (cfm.west);
	\draw[->] (covm.east)  -- (nonn.west);

	\draw[->] (kappaplus.north)  -- (addn.south);
	\draw[->] (kappaplus.east)  -- (addslalom.west);
	
	\draw[->] (addslalom.north)  -- (addpartslalom.south);
	\draw[->] (addpartslalom.north)  -- (addm.south);
	
	\draw[->] (cfn.north)  -- (cont.south);
	
	\draw[->] (cfpartslalom.north)  -- (cfslalom.south);
	\draw[->] (cfm.north)  -- (cfpartslalom.south);
	
	\draw[->] (cfslalom.east)  -- (cont.west);
	
	\draw[->,dashed] (addn.north east)  -- (d.west);
	
	\draw[->,dashed] (b.north east)  -- (cfn.west);

	\draw[->,dotted] (addn.south)  to [out=345,in=195] (nonn.south);
	\draw[->,dotted] (covn.north)  to [out=15,in=165] (cfn.north);
	
	\draw[->,dotted] (covn.south east)  -- (nonn.north west);

	\end{tikzpicture}
	
	\caption{
		The combined diagram: characteristics related to slaloms and~$\id(\QQ_\kappa)$.
		Remember that the dashed lines connected to~$\mathfrak b_\kappa, \mathfrak d_\kappa$
		require $\kappa$ Mahlo (or at least $S_{\pr}^\kappa$ stationary).
	}
\end{figure}
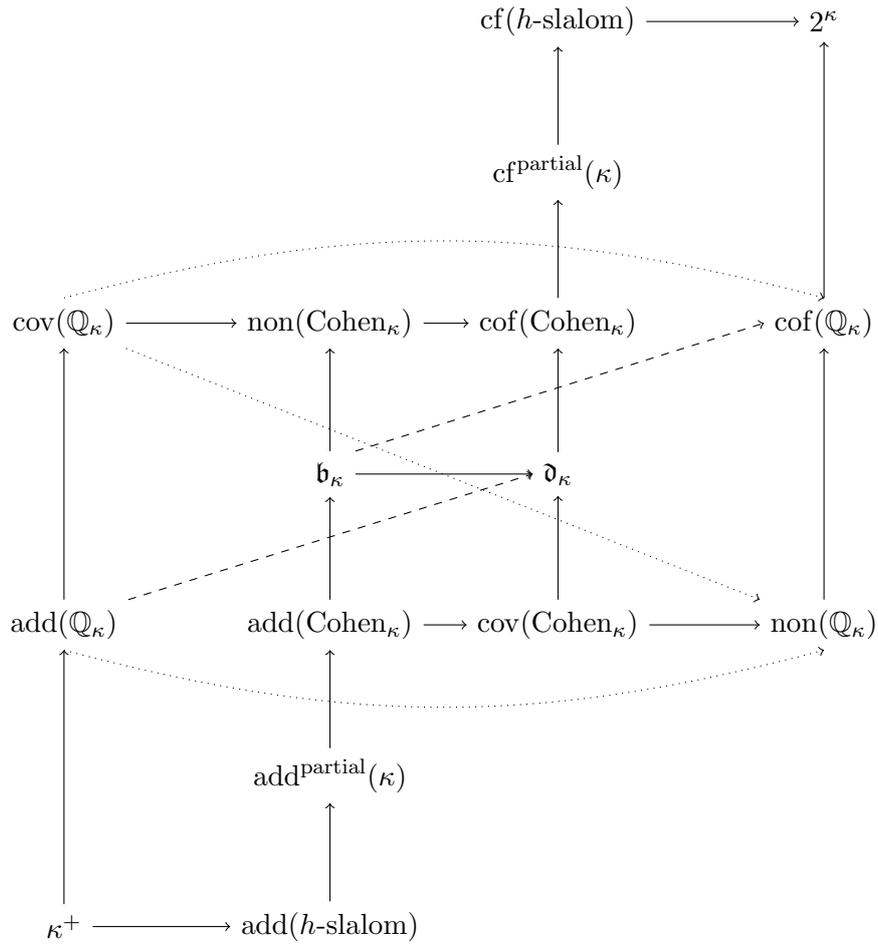

\subsection{Separating Partial Slaloms from~$\id(\QQ_\kappa)$}
The following forcing is used in \cite{BBFM:2016} to show
$\CON(\add(h\text-{slalom}) < \add^{\ppp}(\kappa))$. We 
are going to investigate its effect on~$\id(\QQ_\kappa)$.
\begin{dfn}
	\label{i5}
	Consider the forcing $\pp\LL_\kappa$ consisting of all pairs $(\phi, A)$
	such that
	\begin{enumerate}
		\item 
		$\phi \in \pp \mathcal C_{\id}^\kappa$.
		\item
		$A \seq \kappa^\kappa$, $ |A| < \kappa$.
	\end{enumerate}
	For $p = (\phi_p, A_p), q = (\phi_q, A_q)$, $p, q \in \pp\LL_\kappa$ 
	we define $q$ stronger than $p$ if:
	\begin{enumerate}
		\item 
		$\phi_q \supseteq \phi_p$.
		\item
		$\big(\supp(\phi_q) \setmin \supp(\phi_p)\big) \cap \sup(\supp(\phi_p)) = \emptyset$.
		\item
		$A_q \supseteq A_p$.
		\item
		$i \in \big(\supp(\phi_q) \setmin \supp(\phi_p)\big)$, $ f \in A_p
		\quad\Rightarrow\quad f(i) \in \phi_q(i)$.
	\end{enumerate}
	If $G$ is a $\pp\LL_\kappa$ generic filter then
	$$
	\phi^* = \bigcup_{(\phi, A) \in G} \phi
	$$
	is a partial slalom and we call $\phi$ a generic partial slalom.
	So the intended meaning of $(\phi, A) \in \pp\LL_\kappa$
	is the promise that the generic partial slalom $\phi^*$ will satisfy:
	\begin{enumerate}
		\item 
		$\phi \trianglelefteq \phi^*$.
		\item
		$f \plocal \phi^*$ for every $f \in A$.
	\end{enumerate}
\end{dfn}

\begin{lem}
	\label{i6}
	Let $\PP$ be the limit of the ${<}\kappa$-support iteration
	$\langle \PP_i, \dot \RR_i : i < \kappa^{++} \rangle$ where for each $i < \kappa$ we have:
	$$
	\PP_i \forces \dot \RR_i = \pp\LL_\kappa.
	$$
	Then:
	\begin{enumerate}
		\item
		$\PP$ satisfies~$(*)_\kappa$.
		\item
		For each $i < \kappa^{++}$ the forcing $\PP_i$ is $\kappa$-centered$_{<\kappa}$
	\end{enumerate}
\end{lem}

\begin{proof}
	$ $
	\begin{enumerate}
		\item
		Check that $\pp\LL_\kappa$ satisfies $(*)_\kappa$ and use~\ref{a6}.
		\item
		Check that
		$$
		\pp\LL_\kappa = \bigcup_{\phi \in \pp \mathcal C^\kappa} \{(\phi, A) : A \in [\kappa]^{<\kappa}\}
		$$
		and use~\ref{b5}.
		\qedhere
	\end{enumerate}
\end{proof}

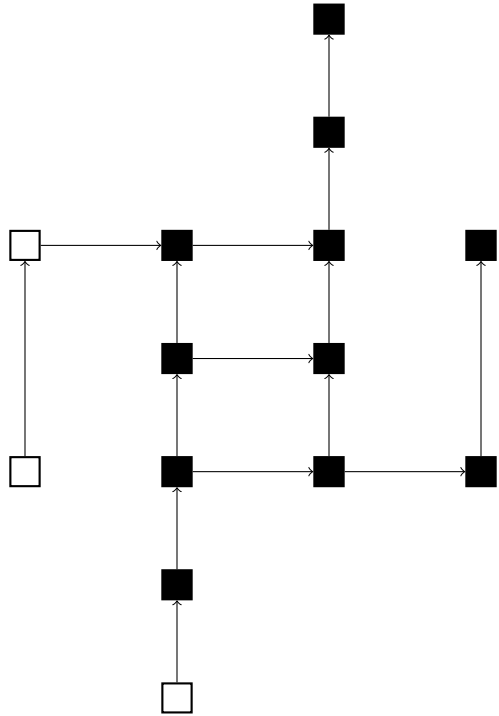
\begin{figure}[h]
	\centering
	
	\begin{tikzpicture}[
	big/.style={rectangle, draw=black!100, fill=black!100, thick, minimum size=1em},
	small/.style={rectangle, draw=black!100, fill=black!0, thick, minimum size=1em},
	unknown/.style={rectangle, draw=black!100, fill=black!30, thick, minimum size=1em},
	]
	%Nodes
	\node[small]	at (0, 0) (addn)   { };
	\node[small]	at (0, 3) (covn)   { };
	\node[big]	at (2, 0) (addm)   { };
	\node[big]		at (2, 3) (nonm)   { };
	\node[big]	at (4, 0) (covm)   { };
	\node[big]		at (4, 3) (cfm)    { };
	\node[big]		at (6, 0) (nonn)   { };
	\node[big]			at (6, 3) (cfn)    { };
	
	\node[big]	at (2, 1.5) (b) { };
	\node[big]		at (4, 1.5) (d) { };
	
	\node[big]	at (2, -1.5) (addpartslalom) { };
	\node[small]	at (2, -3) (addslalom) { };
	
	\node[big]	at (4, 4.5) (cfpartslalom) { };
	\node[big]	at (4, 6) (cfslalom) { };
	
	%Lines
	\draw[->] (addn.north) -- (covn.south);
	\draw[->] (addm.north) -- (b.south);
	\draw[->] (b.north)    -- (nonm.south);
	\draw[->] (covm.north) -- (d.south);
	\draw[->] (d.north)    -- (cfm.south);
	\draw[->] (nonn.north) -- (cfn.south);
	
	\draw[->] (covn.east)  -- (nonm.west);
	\draw[->] (addm.east)  -- (covm.west);
	\draw[->] (b.east)     -- (d.west);
	\draw[->] (nonm.east)  -- (cfm.west);
	\draw[->] (covm.east)  -- (nonn.west);

	%\draw[->,dashed] (addn.east)  -- (d.south west);
	%\draw[->,dashed] (b.north east)  -- (cfn.west);
	
	%\draw[->,dashed] (addn.south)  to [out=345,in=195] (nonn.south);
	%\draw[->,dashed] (covn.north)  to [out=15,in=165] (cfn.north);
	
	\draw[->] (addslalom.north)  -- (addpartslalom.south);
	\draw[->] (addpartslalom.north)  -- (addm.south);
		
	\draw[->] (cfpartslalom.north)  -- (cfslalom.south);
	\draw[->] (cfm.north)  -- (cfpartslalom.south);
	\end{tikzpicture}
	
	\caption{The partial slalom model}
	
\end{figure}

\begin{thm}
	\label{i7}
	Let $\mathbf V \models 2^\kappa = \kappa^+$. Then $\mathbf V^\PP$
	satisfies:
	\begin{enumerate}
		\item 
		$\cov(\QQ_\kappa) = \kappa^+$
		\item
		$\add^{\ppp}(\kappa) = \kappa^{++}$
		\item
		$\add(\hsl) = \kappa^{+}$
		\item
		$\add(\Cohen_\kappa) = \kappa^{++}$
		\item
		$2^\kappa = \kappa^{++}$.
	\end{enumerate}
\end{thm}

\begin{proof}
	$ $
	\begin{enumerate}
		\item
		Argue as in~\ref{n4}.
		\item
		Assume $|\mathcal F|$ witnesses $\add^{\ppp}(\kappa) = \kappa^+$.
		Then by the $\kappa^+$-c.c. $\mathcal F$ already appears in some
		$\mathbf V_\alpha$ and the generic partial slalom added by $\RR_\alpha$
		covers every $f \in \mathcal F$. Contradiction.
		\item
		This is shown in \cite{BBFM:2016}. The argument there is similar to
		(1.) in the sense that it is shown that $\kappa$-centered$_{{<}\kappa}$
		forcings do not increase 
		$\add(\hsl) = \kappa^+$.
		\item
		By (3.) and~\ref{i17}.
		\item
		Should be clear.
		\qedhere
	\end{enumerate}
\end{proof}

\begin{cor}
	\label{i8}
	$ $
	\begin{enumerate}
		\item
		$\CON\big(\add(\QQ_\kappa) < \add^{\ppp}(\kappa)\,\big)$.
		\item
		$\add(\QQ_\kappa) = \add^{\ppp}(\kappa)$ is not a ZFC-theorem.
		\qed
	\end{enumerate}
\end{cor}

\subsection{On Total Slaloms and $\id(\QQ_\kappa)$}

The next conjecture follows from 
	 conjecture~\ref{y0} (and may be easier to prove):
\begin{conj}
	\label{i9}\
	\begin{enumerate}
		\item
		$\CON\big(\add(\QQ_\kappa) > \add^{\ppp}(\kappa)\big)$.
		\item
		In particular
		also
		$\CON\big( \, (\forall h\in \kappa^\kappa)\ \add(\QQ_\kappa) > \add(\hsl)\,\big)$.
		\item
		$(\exists h \in \kappa^\kappa)\ \add(\QQ_\kappa) = \add(\hsl)$
		is not a ZFC-theorem.
		\qed
	\end{enumerate}
\end{conj}

\begin{qst}
	Is $\add(\QQ_\kappa) < \add(\hsl)\big)$ consistent?	
	%Saharon: yes, for suitable $h$ see F1580 h17 (or next section). Doubtful.
	%(TODO reword this, or maybe just delete)	
	For a very partial answer see~\ref{i15}.
\end{qst}

\begin{lem}
	\label{i11}
	Let $S \seq S_\inc^\kappa$ be nowhere stationary. Then
	$\add(\hsl) \leq \add(\id^-(\QQ_{\kappa,S}))$ if:
	\begin{enumerate}
		\item 
		$\epsilon < \kappa \Rightarrow h(\epsilon) \leq \min(S \setmin (\epsilon + 1))$
		\item
		or at least the above holds on club $E \seq \kappa \setmin S$.
	\end{enumerate}	
\end{lem}

\begin{proof}
	Let
	$$\mathcal A \seq \{
	\langle A_\delta : \delta \in S \rangle : A_\delta \in  \id(\QQ_\delta)
	\}$$
	and such that $|\mathcal A| < \add(\hsl)$. We are going to find an upper
	bound for $\mathcal A$.
	Let
	$\langle \epsilon_i : i < \kappa \rangle$, $\epsilon_0 = 0$, increasingly enumerate
	a club disjoint from~$S$.	 
	
	For $A \in \mathcal A$ we define $f_A \colon \kappa \to \kappa$
	such that $f(\epsilon)$ codes $A \on (\epsilon_i, \epsilon_{i+1})$.
	Now by our assumption there exists a slalom $\phi$ such that covers all
	$f_A$ i.e.
	$$
	(\forall^\infty i < \kappa )\ f_A(\epsilon_i) \in \phi(\epsilon_i).
	$$
	
	For $\delta \in (\epsilon_i, \epsilon_{i+1})$ define
	\begin{align*}
	A^*_\delta = \cup \{
	X :\ &\text{a code of a sequence }
	\langle A_\sigma : \sigma \in S \cap (\epsilon_i, \epsilon_{i+1}) \rangle\\
	&\text{such that } X = A_\delta \text{ appears in } \phi(\epsilon_i)
	\}.
	\end{align*}
	By our assumption on $h$
	we have $\epsilon_i < \min(S\setmin (\epsilon_i + 1)) \leq \delta$
	so $A^*_\delta$ is the union of at most $\delta$-many
	elements of $\id(\QQ_\delta)$ hence $A^*_\delta \in \id(\QQ_\delta)$
	and $\langle A^*_\delta : \delta \in S \rangle$ is an upper bound for
	$\mathcal A$.
\end{proof}

\begin{cor}
	\label{i15}
	If all of the following holds:
	\begin{enumerate}
		\item
		$\kappa$ is weakly compact.
		\item
		$\add(\nst_\kappa^{\pr}) > \mathfrak \add(\QQ_\kappa)$.
		\item
		$h$ is as in~\ref{i11}.
	\end{enumerate}
	Then $\add(\QQ_\kappa) \leq \add(\hsl)$.
\end{cor}

\begin{proof}
	By~\ref{i11},~\ref{d6} and~\ref{e5}.
\end{proof}

\newpage
\renewcommand{\refname}{}
\section{References}
\bibliographystyle{chicago}
\bibliography{ours}
\end{document}